\documentclass[11pt]{article}

\usepackage[margin=1in]{geometry}
\usepackage[T1]{fontenc}
\usepackage[utf8]{inputenc}
\usepackage{lmodern}
\usepackage{amsmath,amssymb,amsfonts,amsthm,mathtools}
\usepackage{mathrsfs}
\usepackage{enumitem}
\usepackage{booktabs}
\usepackage{tabularx}
\usepackage{microtype}
\usepackage[colorlinks=true,linkcolor=blue,citecolor=blue,urlcolor=blue]{hyperref}
\numberwithin{equation}{section}

\newtheorem{theorem}{Theorem}[section]
\newtheorem{proposition}[theorem]{Proposition}
\newtheorem{lemma}[theorem]{Lemma}
\newtheorem{corollary}[theorem]{Corollary}
\newtheorem{definition}[theorem]{Definition}
\newtheorem{assumption}[theorem]{Assumption}
\newtheorem{problem}[theorem]{Problem}

\newtheorem{remark}[theorem]{Remark}
\usepackage[capitalize,nameinlink]{cleveref}

\crefname{theorem}{Theorem}{Theorems}
\Crefname{theorem}{Theorem}{Theorems}
\crefname{proposition}{Proposition}{Propositions}
\Crefname{proposition}{Proposition}{Propositions}
\crefname{lemma}{Lemma}{Lemmas}
\Crefname{lemma}{Lemma}{Lemmas}
\crefname{corollary}{Corollary}{Corollaries}
\Crefname{corollary}{Corollary}{Corollaries}
\crefname{definition}{Definition}{Definitions}
\Crefname{definition}{Definition}{Definitions}
\crefname{assumption}{Assumption}{Assumptions}
\Crefname{assumption}{Assumption}{Assumptions}
\crefname{problem}{Problem}{Problems}
\Crefname{problem}{Problem}{Problems}
\crefname{principle}{Principle}{Principles}
\Crefname{principle}{Principle}{Principles}
\crefname{remark}{Remark}{Remarks}
\Crefname{remark}{Remark}{Remarks}

\DeclareMathOperator{\dist}{dist}
\DeclareMathOperator{\Range}{Range}
\DeclareMathOperator{\Ker}{Ker}
\DeclareMathOperator{\coker}{coker}
\DeclareMathOperator{\Cost}{Cost}
\DeclareMathOperator{\tr}{tr}
\DeclareMathOperator{\diag}{diag}

\newcommand{\R}{\mathbb R}
\newcommand{\N}{\mathbb N}
\newcommand{\eps}{\varepsilon}
\newcommand{\loc}{\mathrm{loc}}
\newcommand{\reg}{\mathrm{reg}}
\newcommand{\old}{\mathrm{old}}
\newcommand{\str}{\mathrm{str}}
\newcommand{\har}{\mathrm{har}}
\newcommand{\trc}{\mathrm{tr}}
\newcommand{\prep}{\mathrm{prep}}

\newcommand{\rel}{\mathrm{rel}}
\newcommand{\NS}{\mathrm{NS}}
\newcommand{\VD}{\mathrm{VD}}

\newcommand{\dx}{\,dx}

\newcommand{\dxdt}{\,dx\,dt}
\newcommand{\nabh}{\nabla_h}
\newcommand{\divh}{\nabla_h\cdot}
\newcommand{\norm}[2]{\left\|#1\right\|_{#2}}

\newcommand{\ip}[2]{\left\langle #1,#2\right\rangle}
\newcommand{\calA}{\mathcal A}
\newcommand{\calC}{\mathcal C}
\newcommand{\calD}{\mathcal D}
\newcommand{\calE}{\mathcal E}
\newcommand{\calH}{\mathcal H}
\newcommand{\calL}{\mathcal L}
\newcommand{\calM}{\mathcal M}
\newcommand{\calO}{\mathcal O}
\newcommand{\calQ}{\mathcal Q}
\newcommand{\calR}{\mathcal R}

\newcommand{\calU}{\mathcal U}
\newcommand{\calV}{\mathcal V}
\newcommand{\calX}{\mathcal X}
\newcommand{\calZ}{\mathcal Z}
\newcommand{\frakD}{\mathfrak D}
\newcommand{\frakS}{\mathfrak S}

\newcommand{\Obs}{\operatorname{Obs}}

\title{\textbf{Schur Visibility and Anti-Phantom Reduction\\
		in One-Component Navier--Stokes Degeneration}}

\author{Runlong Yu\\
	The University of Alabama, Tuscaloosa, AL, USA\\
	\texttt{ryu5@ua.edu}}
\date{}

\begin{document}

\maketitle

\begin{abstract}
We study the finite-scale one-component degeneration problem for suitable weak solutions of the three-dimensional incompressible Navier--Stokes equations under a scale-invariant bound and smallness of the vertical component. Qualitative compactness gives convergence, in the harmonic-pressure quotient, toward the strict two-and-a-half-dimensional boundary, but it does not provide a quantitative rate. This paper proves, in an explicitly abstract trace-obstruction skeleton associated with the old observable closure, that the standard old observable package is insufficient to force a logarithmic or power selected-trace rate. The negative result is an envelope/skeleton theorem, not a Navier--Stokes counterexample. After excluding elementary high-frequency escape by parabolic trace drop and fixed-window analytic obstruction by finite-dimensional Lojasiewicz control, the remaining obstruction is an all-order finite-mode flat branch. We identify the Navier--Stokes-specific mechanism needed to control this branch: strict Schur trace-projectability may fail, but the resulting defect can be visible through the relaxed vertical-pressure channel. In active finite-window models, strict Schur phantoms are relaxed-visible. The final theorem is a conditional dichotomy: either relaxed anti-phantom closure holds and yields conditional logarithmic strict-shadow selection, or there exists an NS-realizable, cleaned, relaxed-invisible, unaligned left-singular cascade.
\end{abstract}

\noindent\textbf{Keywords.} Navier--Stokes equations; suitable weak solutions; one-component regularity; strict 2.5D degeneration; old observables; Schur obstruction; vertical pressure; trace cost; vertical duality.

\medskip

\noindent\textbf{2020 Mathematics Subject Classification.} 35Q30; 35B65; 35B45; 76D05.

\tableofcontents

\section{Introduction and theorem status}

\subsection{Problem and compactness boundary}
Let
\[
  Q_r(z_0)=B_r(x_0)\times(t_0-r^2,t_0), \qquad z_0=(x_0,t_0),
\]
and write \(Q_r=Q_r(0,0)\). We consider suitable weak solutions \((u,p)\) of
\begin{equation}\label{eq:NS}
  \partial_t u-\Delta u+(u\cdot\nabla)u+\nabla p=0,\qquad \nabla\cdot u=0.
\end{equation}
in \(Q_1\). The weak-solution framework used here is the local suitable-weak theory descending from Leray--Hopf weak solutions and the Scheffer--Caffarelli--Kohn--Nirenberg partial regularity program \cite{Leray1934,Hopf1951,Scheffer1976,Scheffer1977,CKN1982}.  We also use the later local compactness, pressure, and \(\varepsilon\)-regularity refinements developed in, among others, \cite{Struwe1988,Lin1998,LadyzhenskayaSeregin1999,Seregin2007Local,Seregin2015,Vasseur2007,GustafsonKangTsai2007,GuevaraPhuc2017,AlbrittonBarkerPrange2023}. The scale-invariant regime is
\[
  \Phi(1)\le M,\qquad C_3(1)=\delta\ll1,
\]
where \(C_3\) measures the vertical component \(u_3\). If \(\delta\to0\) under the fixed bound \(\Phi(1)\le M\), compactness forces subsequential convergence, modulo harmonic pressures, toward the strict two-and-a-half-dimensional boundary
\[
  V=(V_h,0),\qquad \nabla_h\cdot V_h=0,
\]
with pressure satisfying \(\partial_3Q=0\). The rate problem asks for more: whether the data
\[
  \Phi(1)\le M,
  \qquad C_3(1)=\delta
\]
force a quantitative approximation
\[
  X^{\har}_\theta(u,p;M)\le F_{M,\theta}(\delta)
\]
with \(F_{M,\theta}\) logarithmic or algebraic as \(\delta\downarrow0\).

This question is adjacent to, but not the same as, the classical Prodi--Serrin continuation theory and related pressure/Morrey/initial-time regularity mechanisms \cite{Prodi1959,Serrin1962,Serrin1963,KozonoSohr1997,SereginSverak2002,EscauriazaSereginSverak2003,Seregin2007Morrey,JiaSverak2014}.  It is also informed by one-component, one-gradient, and anisotropic regularity criteria \cite{PenelPokorny2004,KukavicaZiane2006,KukavicaZiane2007,ZhouPokorny2009,ZhouPokorny2010,CaoTiti2011,CheminZhang2016,CheminZhangZhang2017,KukavicaRusinZiane2017,HanLeiLiZhao2019,KangNguyen2023}.  The finite-dimensional obstruction language below is closer to an analytic-geometric reduction than to a direct PDE regularity proof, and is therefore naturally compared with singular-set and approximation techniques such as \cite{ChoeLewis2000,DuzaarMingione2004}.

The paper treats this rate question as an observability problem. The central message is
\[
\boxed{\text{old observables do not force a logarithmic or power selected-trace rate,}}
\]
and
\[
\boxed{\text{the missing structure is Navier--Stokes-specific vertical-pressure observability.}}
\]
The phrase ``old observables'' refers to the local quantities produced by the standard compactness and preparation arguments: scale-invariant energy-pressure quantities, harmonic-pressure quotient distances, covariance stresses, unresolved variance, prepared residuals, good-time trace functionals, finite-window compatibility quotients, and trace-cost duality.  The pressure and harmonic-gauge part is aligned with the local pressure literature \cite{SohrWahl1986,Wolf2017,Yu2026HarmonicPressure}, while the covariance/coarse-graining language is informed by commutator and energy-flux methods such as \cite{ConstantinETiti1994}.

\subsection{Theorem status and non-claims}
The statements in the paper have different logical status. The following table is meant to prevent overclaiming.
\begin{center}
\begin{tabularx}{0.96\textwidth}{@{}p{0.48\textwidth}X@{}}
\toprule
Statement & Status \\
\midrule
Qualitative strict-shadow compactness & standard compactness statement recalled here \\
Old-observable no-rate theorem & proved abstract trace-skeleton/envelope-insufficiency theorem \\
Trace-drop removes high-frequency escape & proved under sharp admissible-time intersection \\
Fixed-window analytic obstruction control & proved finite-dimensional analytic theorem \\
All-order non-summable finite-stage model & proved model theorem \\
Active finite-window Schur relaxed visibility & proved periodic finite-window theorem \\
Localized relaxed visibility & proved under explicit perturbative or controlled finite-window hypotheses \\
Final anti-phantom dichotomy & conditional reduction and theorem target \\
Unconditional one-component regularity & not claimed \\
Navier--Stokes counterexample & not claimed \\
\bottomrule
\end{tabularx}
\end{center}

The main negative result is an insufficiency theorem for an envelope strictly larger than the true Navier--Stokes image. It does not construct a singular Navier--Stokes solution. The final logarithmic regularity-radius conclusion is conditional on prepared comparison, sharp finite-window reduction, and vertical-duality or relaxed anti-phantom closure.

The word \emph{proved} is used below only for statements whose hypotheses explicitly include the structural inputs needed in the proof.  Statements involving localization, moving windows, combined observability, NS residual representation, or anti-phantom closure are formulated as conditional reductions unless those hypotheses are part of the theorem itself.

\subsection{Main theorems}
The paper is organized around four statements.

\begin{theorem}[Old-observable no-rate theorem]
For every sequence \(a_n\downarrow0\), the abstract old trace-obstruction skeleton associated with the old-observable closure contains a branch with selected-time trace distance \(m_n=\frac12 a_n^2\), while all fixed finite-window residuals, covariance defects, variance quantities, pressure-gauge observables, good-time observables, and prepared residuals are controlled or vanish. Consequently, from the old observable package alone, without a uniform finite-window inverse, a quantitative compactness modulus, trace-drop intersection, or vertical-duality/relaxed anti-phantom input, no universal logarithmic or power selected-trace rate follows.
\end{theorem}

\begin{theorem}[Obstruction separation]
After parabolic trace-drop and sharp admissible-time intersection, pure high-frequency escape is excluded. After restriction to any fixed analytic finite window, finite-dimensional Lojasiewicz control excludes arbitrary-slow convergence inside that fixed window. The remaining abstract obstruction is an all-order finite-mode flat branch with potentially non-summable finite-stage exactification constants.
\end{theorem}

\begin{theorem}[Schur-to-relaxed visibility]
In a clean active periodic finite-window model, let \(\Phi_\Lambda=\Ker A_\Lambda^*\) be the strict trace-invisible Schur cokernel and let
\[
  J^{\rel}_\Lambda s=\sum_{\sigma\in\Lambda}\mu_\sigma s_\sigma w_\sigma,
  \qquad
  \mu_\sigma=\frac{|\sigma_h|^2+\sigma_3^2}{|\sigma_h|^2}.
\]
If \(\Lambda\) is active, meaning \(|\sigma_h|\ne0\) and \(|\sigma_3|\ne0\) for every output mode, then
\[
  J^{\rel}_\Lambda s=0\quad\Longrightarrow\quad s=0.
\]
Thus a strict-invisible Schur defect need not be a true Navier--Stokes phantom; in active finite windows it is relaxed-visible through vertical pressure.
\end{theorem}

\begin{theorem}[Final conditional dichotomy]
After preparation, good-time selection, finite-window reduction, strict--Schur obstruction removal, harmonic-pressure cleaning, localized coercivity, homogeneous-tail testing, and NS-realizability filtering, exactly one of the following remains. Either the surviving NS-derived residuals satisfy relaxed anti-phantom closure, which yields controlled trace-cost exactification and conditional logarithmic strict-shadow selection; or there exists an NS-realizable, cleaned, relaxed-invisible, unaligned left-singular cascade. This cascade is the final obstruction isolated by the present route.
\end{theorem}

\subsection{Main contributions}
The paper has five contributions.
First, it isolates the old observable package and formalizes it as a defect-shadow envelope. Second, it proves that the associated abstract trace-obstruction skeleton permits arbitrary-slow selected-trace convergence, so old observables alone cannot give a logarithmic or power rate. Third, it separates false obstructions: parabolic trace-drop removes high-frequency escape, while finite-dimensional analytic geometry removes fixed-window arbitrary-slow behavior. Fourth, it identifies Schur trace-projectability as the remaining strict trace-cost obstruction and shows that active Schur defects are relaxed-visible through vertical pressure. Fifth, it gives a final anti-phantom dichotomy reducing the rate problem to the exclusion or construction of an NS-realizable, cleaned, relaxed-invisible, unaligned left-singular cascade.

\subsection{Logical spine and organization}
The dependency chain is
\[
\begin{array}{c}
\text{small }u_3+\text{scale bound}
\Rightarrow \text{strict compactness}
\Rightarrow \text{old envelope}
\Rightarrow \text{no-rate branch},\\[0.25em]
\text{trace-drop + finite-window tests}
\Rightarrow \text{finite-mode flat branch}
\Rightarrow \text{vertical lift / Schur visibility},\\[0.25em]
\text{vertical-pressure observability}
\Rightarrow \text{final anti-phantom dichotomy}.
\end{array}
\]
Section 2 fixes the compactness boundary. Section 3 defines the old observables and the selected-time target. Sections 4 and 5 prove the old-envelope no-rate theorem. Section 6 separates the false escape mechanisms. Section 7 gives trace-cost duality. Section 8 introduces Navier--Stokes realizability and vertical lift cost. Section 9 formulates reduced vertical duality and Schur trace-projectability. Section 10 proves Schur-to-relaxed visibility in active finite windows. Section 11 records the localization, gauge, NS-realizability, and moving-window filters. Section 12 states the final anti-phantom dichotomy. Section 13 gives the conditional logarithmic consequence. Section 14 summarizes the theorem targets. The appendices contain the longer Fourier, localization, and moving-window calculations.

\section{Preliminaries and the compactness boundary}\label{sec:preliminaries}

\subsection{Suitable weak solutions}

We use the standard local class.  A pair \((u,p)\) is a suitable weak solution of \eqref{eq:NS} in a parabolic cylinder \(Q\) if
\[
u\in L^\infty_tL^2_x(Q)\cap L^2_tH^1_x(Q),
\qquad
p\in L^{3/2}(Q),
\]
\((u,p)\) solves \eqref{eq:NS} distributionally, \(\nabla\cdot u=0\), and for every nonnegative \(\varphi\in C_c^\infty(Q)\),
\begin{align}\label{eq:local-energy}
\int |u(t)|^2\varphi(t)\dx
&+2\int_{-\infty}^t\int |\nabla u|^2\varphi\dx\,ds
\nonumber\\
&\le
\int_{-\infty}^t\int |u|^2(\partial_s\varphi+\Delta\varphi)\dx\,ds
+
\int_{-\infty}^t\int (|u|^2+2p)u\cdot\nabla\varphi\dx\,ds
\end{align}
for a.e. \(t\).  This is the standard suitable-weak class used in the CKN theory and its later refinements \cite{CKN1982,Lin1998,Seregin2015,Vasseur2007}. The pressure is determined only up to functions of time.  Locally, after pressure decomposition, it is also natural to quotient by spatially harmonic functions \cite{SohrWahl1986,Wolf2017,Yu2026HarmonicPressure}.

\subsection{Harmonic pressure gauge}

For a cylinder \(Q_\rho=B_\rho\times I_\rho\), define
\[
\calH(Q_\rho)=\{h\in L^{3/2}(Q_\rho):\Delta h(\cdot,t)=0\text{ in }B_\rho\text{ for a.e. }t\}.
\]
We write \([p]_{\calH}=p+\calH\).  The harmonic-pressure excess relative to the strict class is
\begin{equation}\label{eq:harmonic-excess}
\calX^{\har}_\rho(u,p;M)=
\inf_{(V,Q)\in\frakS_M(Q_\rho)}\inf_{h\in\calH(Q_\rho)}
\left[
\rho^{-2}\int_{Q_\rho}|u-V|^3\dxdt
+
\rho^{-2}\int_{Q_\rho}|p-Q-h|^{3/2}\dxdt
\right].
\end{equation}
This is the correct local topology for the pressure: the Calderon--Zygmund part is tied to the velocity, while the spatially harmonic part is not controlled by the same local formula; compare the local pressure decompositions in \cite{SohrWahl1986,Wolf2017,Seregin2015} and the harmonic-pressure formulation in \cite{Yu2026HarmonicPressure}.

\subsection{Strict two-and-a-half-dimensional shadows}

A strict 2.5D shadow is a pair \((V,Q)\) with
\[
V=(V_h,0),
\qquad \divh V_h=0,
\qquad \partial_3Q=0,
\]
solving
\begin{equation}\label{eq:strict-system}
\partial_tV_h-\Delta V_h+\nabh\cdot(V_h\otimes V_h)+\nabh Q=0.
\end{equation}
We denote by \(\frakS_M(Q_\rho)\) the class of such shadows in \(Q_\rho\) with a fixed scale-invariant bound depending on \(M\).

The vertical pressure condition \(\partial_3Q=0\) produces a nonlinear compatibility constraint.  Formally, taking horizontal divergence in \eqref{eq:strict-system} gives
\[
-\Delta_h Q=\partial_a\partial_b(V_aV_b),
\qquad a,b\in\{1,2\}.
\]
Thus, modulo horizontal harmonic issues and after fixing the horizontal pressure gauge,
\[
Q=-\Delta_h^{-1}\partial_a\partial_b(V_aV_b).
\]
The condition \(\partial_3Q=0\) is equivalent to
\begin{equation}\label{eq:compatibility-map}
\calC(V):=\nabh\partial_3\Delta_h^{-1}\partial_a\partial_b(V_aV_b)=0.
\end{equation}
We also write
\begin{equation}\label{eq:B-bilinear}
B(A,B)=\nabh\partial_3\Delta_h^{-1}\partial_a\partial_b(A_aB_b),
\qquad
\calC(V)=B(V,V).
\end{equation}
At a base shadow \(V\), the linearized compatibility map is
\begin{equation}\label{eq:linearized-compatibility}
D\calC_V[W]
=
\nabh\partial_3\Delta_h^{-1}\partial_a\partial_b(V_aW_b+W_aV_b).
\end{equation}
The quotient
\[
\Obs_V:=\coker D\calC_V
\]
is the formal obstruction space.  A central difficulty is that
\[
W\in\Ker D\calC_V
\]
does not by itself imply that \(W\) is tangent to an actual curve in \(\calC^{-1}(0)\).  This is the source of finite-order and flat-branch obstruction phenomena.

\subsection{Qualitative compactness versus rate}

The standard compactness mechanism gives the following qualitative implication, using local energy compactness, pressure decomposition, and the Aubin--Lions--Simon compactness principle \cite{Simon1986,CKN1982,Seregin2015}.

\begin{proposition}[Qualitative strict-shadow compactness]\label{prop:qual-compactness}
Fix \(M\ge1\) and \(0<\theta<1/2\).  There exists a nondecreasing modulus \(\omega_{M,\theta}\) with \(\omega_{M,\theta}(s)\to0\) as \(s\downarrow0\) such that every suitable weak solution in \(Q_1\) satisfying \(\Phi(1)\le M\) obeys
\begin{equation}\label{eq:qual-compactness}
\calX^{\har}_{\theta}(u,p;M)
\le
\omega_{M,\theta}(C_3(1)).
\end{equation}
\end{proposition}

\begin{proof}[Sketch]
Assume otherwise.  Then there are suitable weak solutions \((u^{(n)},p^{(n)})\) with \(\Phi^{(n)}(1)\le M\), \(C_3^{(n)}(1)\to0\), and \(\calX^{\har}_{\theta}(u^{(n)},p^{(n)};M)\ge\eta_0>0\).  The local energy bound gives
\[
 u^{(n)} \text{ bounded in } L_t^\infty L_x^2(Q_\sigma)\cap L_t^2H_x^1(Q_\sigma)
 \qquad (\sigma<1).
\]
The equation gives a uniform bound for \(\partial_tu^{(n)}\) in a fixed negative Sobolev space, after localizing in \(Q_\sigma\), because
\[
 \partial_tu^{(n)}=
 \Delta u^{(n)}-\operatorname{div}(u^{(n)}\otimes u^{(n)})-\nabla p^{(n)}
\]
with the three terms bounded respectively by the energy and pressure bounds.  Aubin--Lions--Simon compactness \cite{Simon1986} therefore gives, after passing to a subsequence,
\[
 u^{(n)}\to V \quad\text{strongly in }L^3_{\loc}(Q_1),
\]
and the pressure representatives, after subtracting time-dependent constants, converge weakly in \(L^{3/2}_{\loc}\). Since \(u_3^{(n)}\to0\) in \(L^3\), the limit has the form \(V=(V_h,0)\), and incompressibility gives \(\divh V_h=0\).

For the pressure, use the standard local decomposition on a smaller cylinder \cite{CKN1982,SohrWahl1986,Wolf2017,Seregin2015}: the Calderon--Zygmund part generated by \(u^{(n)}\otimes u^{(n)}\) converges strongly in \(L^{3/2}\) because \(u^{(n)}\to V\) strongly in \(L^3\), while the remaining spatially harmonic part is compact in the interior after quotienting by \(\calH\).  Hence the pressure converges in the harmonic-pressure quotient. Passing to the limit in the horizontal equations gives \eqref{eq:strict-system}; passing to the limit in the vertical equation gives \(\partial_3Q=0\).  Thus the limit is an admissible strict shadow in the quotient topology, contradicting the lower bound on \(\calX^{\har}_\theta\).
\end{proof}

\begin{remark}
Proposition~\ref{prop:qual-compactness} is intentionally qualitative.  The compactness argument does not identify a power or logarithmic decay rate for \(\omega_{M,\theta}\).  The rest of the paper asks which additional structures could produce such a rate.
\end{remark}

\section{Old observables and the selected-trace target}\label{sec:atlas}

We now organize the quantities visible in the one-component degeneration problem.  An observable is any object naturally defined from the solution, the equation, the pressure gauge, a coarse-graining operation, a strict-shadow comparison, a finite-window projection, or a selected-time minimization.  The first layers are classical in partial regularity and pressure theory \cite{CKN1982,Lin1998,Seregin2015,Wolf2017}; the component and anisotropic layers are motivated by one-component regularity criteria \cite{KukavicaZiane2006,CaoTiti2011,CheminZhang2016,CheminZhangZhang2017,HanLeiLiZhao2019,KangNguyen2023}; the rate/modulus layer is closest in spirit to quantitative regularity work such as \cite{BarkerPrange2021}.

\subsection{Levels of observability}

We distinguish four levels.

\begin{description}[leftmargin=2.6cm,style=nextline]
\item[Level 0: native observables.]  These are directly defined from \((u,p)\): the cylinders \(Q_r\), the scale-invariant quantities \(A,E,C,D,C_3\), the component quantities \(u_h,u_3\), and vorticity components.

\item[Level 1: equation-derived observables.]  These come from the equations: the pressure Poisson law, incompressibility \(\divh u_h=-\partial_3u_3\), and the vertical momentum equation.

\item[Level 2: operation-derived observables.]  These arise after smoothing, projection, pressure decomposition, horizontal Helmholtz correction, and covariance expansion.  Examples are \(U^\ell\), \(P^\ell\), \(\tau^\ell\), \(\kappa^\ell\), \(G^\ell\), and \(\partial_3P^\ell\).

\item[Level 3: selection and obstruction observables.]  These are defined only after strict-shadow comparison, selected-time minimization, blow-up, or finite-window projection.  Examples include \(m_\ell\), normalized trace directions \(W_n\), finite-window obstruction spaces, and trace-cost quotients.
\end{description}

\subsection{Twelve observable families}

For later reference, we compress the observable atlas into twelve families.

\begin{center}
\begin{tabularx}{\textwidth}{@{}cX@{}}
\toprule
Family & Representative observables \\
\midrule
\(\calO_1\) & Scale observables: \(z_0,r,Q_r(z_0)\), parabolic scale trees. \\
\(\calO_2\) & Energy-pressure observables: \(A,E,C,D,C_3,\Phi,\Psi\). \\
\(\calO_3\) & Component/anisotropic observables: \(u_h,u_3,\nabh u_h,\partial_3u_h,\nabh u_3,\partial_3u_3,\omega_h,\omega_3\). \\
\(\calO_4\) & Pressure-gauge observables: \([p]_{\calH}\), Calderon--Zygmund pressure, horizontal pressure part, vertical pressure remainder, harmonic oscillation, \(\calX^{\har}\). \\
\(\calO_5\) & Limiting/shadow observables: strict class \(\frakS_M\), \((V,Q)\), strict decay, \(\calC(V)\), \(D\calC_V\), \(B(\cdot,\cdot)\). \\
\(\calO_6\) & Coarse-preparation observables: smoothing \(S_\ell u\), prepared field \(U^\ell\), horizontal correction, divergence defect, smoothing tails. \\
\(\calO_7\) & Covariance/defect observables: Reynolds stress \(\tau^\ell\), unresolved variance \(\kappa^\ell\), gradient variance, prepared residual \(G^\ell\), vertical pressure defect \(\calV^\ell\). \\
\(\calO_8\) & Local energy observables: full local energy defect, horizontal defect measure \(\mu_h^\ell\), cutoff errors, small-component errors. \\
\(\calO_9\) & Good-time and trace observables: good-time sets \(G_\ell\), thick good times \(G^\sharp_\ell\), selected trace energy \(E^\ell_\phi\), sharp distance \(m_\ell\), normalized trace directions. \\
\(\calO_{10}\) & Obstruction observables: \(\Obs_V\), finite-order obstructions \(\calO_k\), finite-window projections \(P_\Lambda\calO_k\), flatness order. \\
\(\calO_{11}\) & Trace-cost/duality observables: active trace map \(A\), spaces \(H,Y\), residual \(g\), cost \(\Cost^{\trc}(g)\), adjoint \(A^*\), vertical-duality ratio. \\
\(\calO_{12}\) & Rate/modulus observables: compactness modulus \(\omega_{M,\theta}\), best modulus \(\Omega^{\mathrm{best}}_{M,\theta}\), power and logarithmic profiles. \\
\bottomrule
\end{tabularx}
\end{center}

The total old observable collection is denoted
\[
\calO_{\mathrm{total}}=\bigcup_{j=1}^{12}\calO_j.
\]
The key point is that \(\calO_{\mathrm{total}}\) is descriptive.  It does not assert that the selected trace distance is quantitatively controlled by these quantities.

\subsection{The selected-time target}

The main Level 3 target is the covariance-calibrated trace distance.  Given a prepared horizontal field \(U^\ell\), unresolved variance \(\kappa^\ell\), nonnegative cutoff \(\phi\), and strict shadow \((V,Q)\), set
\begin{equation}\label{eq:selected-trace-energy}
E^\ell_\phi(s;U^\ell,V)
=
\frac12\int \phi |U^\ell(s)-V(s)|^2\dx
+
\int \phi\kappa^\ell(s)\dx.
\end{equation}
For a good-time set \(G_\ell\), define
\begin{equation}\label{eq:m-ell}
m_\ell=
\inf_{s\in G_\ell}\inf_{(V,Q)\in\frakS_M}
E^\ell_\phi(s;U^\ell,V).
\end{equation}
The central quantitative question is whether the old observables force a subcritical bound of the schematic form
\begin{equation}\label{eq:subcritical-selection-target}
m_\ell
\le
C_M\ell^\mu+C_M\ell^{-N}\delta^b.
\end{equation}
The no-rate theorem below shows that such a bound is not a formal consequence of the abstract old trace skeleton alone.

\section{The old-observable envelope}\label{sec:old-envelope}

We now formalize the ``old world'' as an envelope.  It is deliberately larger than the true image of Navier--Stokes under coarse-graining.  Its purpose is diagnostic: if a rate fails in this larger world, then any true Navier--Stokes rate must use structure beyond this envelope.  The envelope viewpoint combines the scale-invariant CKN quantities \cite{CKN1982}, local pressure quotienting \cite{SohrWahl1986,Wolf2017,Yu2026HarmonicPressure}, and covariance/commutator bookkeeping familiar from coarse-grained energy balances \cite{ConstantinETiti1994}.

Fix once and for all nested cylinders
\[
Q_{\mathrm{sh}}\Subset Q_{\mathrm{prep}}\Subset Q_1,
\]
a nonnegative spatial cutoff \(\phi\), and parameters
\[
M\ge1,
\qquad 0<\theta<1/2,
\qquad 0<\ell<\ell_0,
\qquad 0<\delta\le1.
\]
Here \(\ell\) represents a coarse-graining scale and \(\delta\) represents \(C_3(1)\).

\begin{definition}[Old defect-shadow object]\label{def:old-object}
An old defect-shadow object is a tuple
\begin{equation}\label{eq:old-object-tuple}
\calU^\ell=(U^\ell,[P^\ell],\tau^\ell,\kappa^\ell,G^\ell,\calV^\ell,\mu_h^\ell),
\end{equation}
where
\begin{itemize}[leftmargin=2em]
\item \(U^\ell=(U_h^\ell,0)\) is a horizontal vector field;
\item \([P^\ell]\) is a pressure class modulo \(\calH\);
\item \(\tau^\ell\) is a horizontal Reynolds covariance stress;
\item \(\kappa^\ell\) is unresolved variance;
\item \(G^\ell\) is the prepared non-pressure residual;
\item \(\calV^\ell\) is the vertical pressure compatibility defect, usually \(\partial_3P^\ell\) in a negative norm;
\item \(\mu_h^\ell\) is a horizontal energy or Onsager defect measure.
\end{itemize}
\end{definition}

\begin{definition}[Minimal old-observable envelope]\label{def:minimal-old-envelope}
We say
\[
\calU^\ell\in\frakD^{\old,0}_{M,\theta}(\ell,
\delta)
\]
if the following axioms hold.

\begin{description}[leftmargin=1.2cm,labelsep=0.5cm]
\item[O0. Pressure gauge.]  Pressure is observed only through the quotient \([P^\ell]=P^\ell+\calH\).  All pressure distances are taken after minimizing over spatially harmonic functions.

\item[O1. Scale-invariant size.]  There is \(C_M\) such that
\[
\norm{U^\ell}{L^\infty_tL^2_x(Q_{\prep})}
+
\norm{\nabla U^\ell}{L^2(Q_{\prep})}
+
\norm{U^\ell}{L^3(Q_{\prep})}
+
\inf_{h\in\calH}\norm{P^\ell-h}{L^{3/2}(Q_{\prep})}
\le C_M.
\]

\item[O2. Horizontal solenoidal preparation.]  We have
\[
U^\ell=(U_h^\ell,0),
\qquad
\divh U_h^\ell=0.
\]

\item[O3. Prepared covariance-form equation.]  In distributions on \(Q_{\prep}\),
\begin{equation}\label{eq:prepared-equation}
\partial_tU_h^\ell-\Delta U_h^\ell
+\nabh\cdot(U_h^\ell\otimes U_h^\ell)
+\nabh P^\ell
=
-\nabh\cdot\tau^\ell+G^\ell.
\end{equation}
Equivalently, the covariance stress is placed on the left-hand flux and \(G^\ell\) is the genuinely small residual.

\item[O4. Residual hierarchy.]  There are fixed exponents \(a_0,N,b>0\) such that
\begin{equation}\label{eq:residual-hierarchy}
\norm{\tau^\ell}{L^{3/2}(Q_{\prep})}
\le C_M\ell^{a_0},
\qquad
\norm{G^\ell}{Z'(Q_{\prep})}
\le C_M\ell^{-N}\delta^b,
\qquad
\norm{\calV^\ell}{Y'(Q_{\prep})}
\le C_M\ell^{-N}\delta^b.
\end{equation}
The spaces \(Z'\) and \(Y'\) are fixed negative norms compatible with local energy and pressure-compatibility pairings.

\item[O5. Positive covariance.]  For a.e. \((x,t)\), \(\tau^\ell(x,t)\) is a nonnegative horizontal \(2\times2\) matrix and
\begin{equation}\label{eq:kappa-trace}
\kappa^\ell=\frac12\tr\tau^\ell\ge0.
\end{equation}

\item[O6. Good-time variance supply.]  On a fixed interval \(I_-\Subset(-1,0)\),
\begin{equation}\label{eq:variance-average}
\int_{I_-}\int \phi\kappa^\ell\dxdt\le C_M\ell^2.
\end{equation}
The good-time set
\begin{equation}\label{eq:good-time-set}
G_\ell=
\left\{s\in I_-:\int\phi\kappa^\ell(x,s)\dx\le C_G\ell^2\right\}
\end{equation}
satisfies \(|G_\ell|\ge c_G>0\).

\item[O7. Weak horizontal defect balance.]  There are terms \(D_h^\ell\ge0\), \(E^\ell_{\mathrm{cut}}\), and \(E^\ell_{\delta}\) such that, distributionally in time,
\begin{equation}\label{eq:variance-balance}
\frac{d}{dt}\int\phi\kappa^\ell
+
\int\phi D_h^\ell
+
\int\phi\,d\mu_h^\ell
=
-\int\phi\tau^\ell:\nabh U_h^\ell
+E^\ell_{\mathrm{cut}}+E^\ell_{\delta},
\end{equation}
with
\begin{equation}\label{eq:variance-errors}
\int|E^\ell_{\mathrm{cut}}|\le C_M\ell^{a_0},
\qquad
\int|E^\ell_{\delta}|\le C_M\ell^{-N}\delta^b,
\qquad
\int\phi\,d(\mu_h^\ell)_-\le C_M\ell^{-N}\delta^b.
\end{equation}

\item[O8. Rough harmonic strict-shadow compactness.]  There is a compactness modulus \(\Omega_M(s)\to0\) as \(s\downarrow0\) such that
\begin{equation}\label{eq:rough-shadow-modulus}
\dist_{\har}\bigl((U^\ell,P^\ell),\frakS_M(Q_{\mathrm{sh}})\bigr)
\le
\Omega_M(\ell^{a_0}+\ell^{-N}\delta^b).
\end{equation}
This is only a qualitative modulus, not a power or logarithmic estimate.

\item[O9. Finite-window compatibility observables.]  For every finite frequency window \(\Lambda\), there are finite-dimensional spaces \(H_\Lambda,Y_\Lambda\), a finite-dimensional active trace-defect map
\[
A_\Lambda:H_\Lambda\to Y_\Lambda,
\]
and projected compatibility residuals in \(Y_\Lambda\).  The constants of all finite-dimensional constructions may depend on \(\Lambda\).  No uniform-in-\(\Lambda\) inverse or Lojasiewicz constant is assumed.

\item[O10. Selected trace distance.]  The quantity
\begin{equation}\label{eq:old-selected-distance}
m^\ell_{\old}
=
\inf_{s\in G_\ell}\inf_{(V,Q)\in\frakS_M}
\left[
\frac12\int\phi|U^\ell(s)-V(s)|^2\dx+
\int\phi\kappa^\ell(s)\dx
\right]
\end{equation}
is defined.  The envelope does \emph{not} assume a subcritical estimate for \(m^\ell_{\old}\).
\end{description}
\end{definition}

\subsection{Enhanced envelopes}

The following nested variants will be useful.
\begin{align*}
\frakD^{\old,\sharp}_{M,\theta}(\ell,\delta)&\subset\frakD^{\old,0}_{M,\theta}(\ell,\delta),\\
\frakD^{\old,\trc}_{M,\theta}(\ell,\delta)&\subset\frakD^{\old,\sharp}_{M,\theta}(\ell,\delta),\\
\frakD^{\old,\mathrm{fw}}_{M,\theta}(\ell,\delta)&\subset\frakD^{\old,\trc}_{M,\theta}(\ell,\delta),\\
\frakD^{\old+\VD}_{M,\theta}(\ell,\delta)&\subset\frakD^{\old,\mathrm{fw}}_{M,\theta}(\ell,\delta).
\end{align*}
Here \(\frakD^{\old,\sharp}\) includes thick good times, \(\frakD^{\old,\trc}\) includes trace tightness for failed normalized branches, \(\frakD^{\old,\mathrm{fw}}\) includes fixed-window exactifiability with constants depending on the window, and \(\frakD^{\old+\VD}\) includes the vertical-duality estimate introduced in \cref{sec:trace-cost}.

The true Navier--Stokes image, if all required preparation steps are performed, should satisfy
\begin{equation}\label{eq:envelope-inclusion}
\frakD^{\NS}_{M,\theta}(\ell,
\delta)
\subseteq
\frakD^{\old,0}_{M,\theta}(\ell,
\delta),
\end{equation}
but it may be a much smaller subset.  The main problem is to identify which additional \(\NS\)-specific constraints cut down the old envelope.

\subsection{Explicitly excluded mechanisms}

The minimal old envelope does not include any of the following.
\begin{enumerate}[label=(E\arabic*)]
\item \textbf{Vertical-duality active-residual estimate:}
\[
|\ip{g}{y}|
\le r\norm{A^*y}{H}.
\]
\item \textbf{Navier--Stokes realizability:} existence of a genuine suitable weak solution \((u,p)\) whose coarse-graining produces \(\calU^\ell\).
\item \textbf{Vertical jet realizability:} existence of vertical perturbation jets \(z_1,z_2,\ldots\) solving the full asymptotic expansion of the three-dimensional equations.
\item \textbf{Uniform finite-window inverse:} bounds on \(A_\Lambda^{-1}\) independent of \(\Lambda\).
\item \textbf{Lojasiewicz metric regularity:}
\[
\dist(x,\calC^{-1}(0))\le C\norm{\calC(x)}{Y}^\alpha
\]
with constants uniform across the relevant windows and branches.
\item \textbf{Subcritical strict-shadow selection:}
\[
m^\ell_{\old}\le C_M\ell^\mu+C_M\ell^{-N}\delta^b.
\]
\end{enumerate}
These exclusions are essential.  The next section shows that, once they are excluded, arbitrary-slow trace convergence is possible in an abstract trace-obstruction skeleton associated with the old envelope.

\section{The no-rate theorem in the abstract old trace skeleton}\label{sec:no-rate}

This section gives the first pressure test.  We construct an abstract trace branch satisfying the old finite-window and trace-level observable conditions, but whose selected-time trace distance tends to zero at an arbitrarily slow rate.  The model is not a Navier--Stokes counterexample and is not claimed to solve the full prepared covariance equation in the genuine PDE class.  It is a diagnostic statement about what remains possible after one forgets all mechanisms beyond the old observables.

\subsection{The minimal trace-obstruction skeleton}

Let
\[
H=\ell^2(\N)
\]
be the selected-time trace Hilbert space, with canonical orthonormal basis \((e_j)_{j\ge1}\).  Let the strict trace class be
\[
\frakS_{\mathrm{tr}}=\{0\}\subset H.
\]
This models a neighborhood of a singular strict-shadow base after compressing the strict trace set to its local zero fiber.

Let the compatibility defect space also be
\[
Y=\ell^2(\N).
\]
Choose a positive sequence \((\lambda_j)_{j\ge1}\) with
\[
0<\lambda_j\downarrow0,
\]
for instance \(\lambda_j=\exp(-\exp j)\).  Define the quadratic compatibility map
\begin{equation}\label{eq:l2-compatibility-expanded}
\calC:H\to Y,
\qquad
\calC(x)_j=\lambda_jx_j^2.
\end{equation}
It is a toy version of the strict pressure-compatibility map
\[
\calC(V)=B(V,V),
\]
but with degenerating visibility along high trace directions.

Let \(P_\Lambda\) denote the projection onto \(\operatorname{span}\{e_1,\ldots,e_\Lambda\}\).  A fixed finite-window observer sees only \(P_\Lambda x\) and \(P_\Lambda\calC(x)\).

\begin{remark}[Skeleton versus PDE envelope]\label{rem:skeleton-vs-pde-envelope}
The space above is the \emph{minimal trace-obstruction skeleton}.  It keeps the logical information used by fixed finite-window trace tests, compatibility residuals, and qualitative compactness, but it does not include a genuine parabolic equation, a pressure decomposition in physical space, or a true Navier--Stokes realization map.  Thus the theorem below proves an insufficiency statement for the old-observable closure at trace level.  Any genuine PDE rate must use some additional structure, such as trace-drop, uniform moving-window control, vertical duality, or relaxed anti-phantom closure.
\end{remark}

\subsection{The arbitrary-slow branch}

Fix a desired trace amplitude sequence
\[
a_n\downarrow0.
\]
Let \(0<\mu<1/6\), and let \(N,b>0\) be the residual-loss exponents in the old envelope.  Define the desired subcritical scale
\begin{equation}\label{eq:eta-n-def}
\eta_n=\ell_n^\mu+\ell_n^{-N}\delta_n^b.
\end{equation}
To force failure of subcritical selection, choose \(\ell_n\) and \(\delta_n\) so that
\begin{equation}\label{eq:eta-small-an}
\eta_n=o(a_n^2).
\end{equation}
For example, set
\[
\ell_n=a_n^{4/\mu}
\]
and then choose \(\delta_n\downarrow0\) sufficiently fast that
\[
\ell_n^{-N}\delta_n^b\le a_n^4.
\]
Then
\[
\eta_n\le 2a_n^4=o(a_n^2).
\]
Because \(\lambda_j\to0\), one can choose \(j_n\to\infty\) so large that
\begin{equation}\label{eq:jn-choice-expanded}
\lambda_{j_n}a_n^2\le a_n^4.
\end{equation}
Set
\begin{equation}\label{eq:Un-high-mode}
U_n(t)=a_ne_{j_n},
\qquad t\in I_-.
\end{equation}
The pressure, covariance, vertical residual, and horizontal defect are set to zero:
\begin{equation}\label{eq:zero-defects-model}
[P_n]=0,
\qquad
\tau_n=0,
\qquad
\kappa_n=0,
\qquad
G_n=0,
\qquad
\calV_n=0,
\qquad
\mu_{h,n}=0.
\end{equation}
The only nonzero obstruction is the abstract compatibility residual
\[
g_n=\calC(U_n)=\lambda_{j_n}a_n^2e_{j_n}.
\]
Thus
\begin{equation}\label{eq:gn-small-expanded}
\norm{g_n}{Y}=\lambda_{j_n}a_n^2\le a_n^4.
\end{equation}

\begin{theorem}[Arbitrary-slow branch in the abstract old trace skeleton]\label{thm:arbitrary-slow-expanded}
For every sequence \(a_n\downarrow0\) there exist \(\ell_n\downarrow0\), \(\delta_n\downarrow0\), and an abstract old trace-skeleton branch of the form \eqref{eq:Un-high-mode}--\eqref{eq:zero-defects-model} such that
\begin{enumerate}[label=(\roman*)]
\item \(U_n\to0\) in the space-time compactness topology;
\item all fixed finite-window compatibility defects vanish eventually;
\item the full compatibility residual satisfies \(\norm{g_n}{Y}\le a_n^4\), and hence is much smaller than the selected trace scale;
\item covariance, variance, good-time, pressure-gauge, prepared-residual, and horizontal-defect observables are perfectly controlled;
\item the selected-time trace distance is
\begin{equation}\label{eq:mn-an-expanded}
m_n=\frac12a_n^2.
\end{equation}
\end{enumerate}
Consequently, the abstract trace-obstruction skeleton associated with the old-observable closure permits arbitrary-slow selected-time trace convergence.
\end{theorem}

\begin{proof}
The size bound is immediate from \(\norm{U_n(t)}{H}=a_n\to0\).  The pressure-gauge condition is trivial because \([P_n]=0\).  In the abstract trace model we take the evolution operator to be zero, so \(\partial_tU_n=0\); with \(G_n=0\) and \(\tau_n=0\), the prepared covariance-form equation is satisfied exactly.  This is intentionally weaker than a genuine parabolic equation.  The purpose is to test what the abstract old trace skeleton permits before adding parabolic trace-drop and genuine PDE realizability.

Since \(\tau_n=0\) and \(\kappa_n=0\), the positive covariance and unresolved variance axioms hold.  The good-time set and thick good-time set are both all of \(I_-\).  The localized variance balance reduces to \(d\kappa_n/dt=0\), with \(\mu_{h,n}=E_{\mathrm{cut},n}=E_{\delta,n}=0\).

The strict trace class is \(\frakS_{\mathrm{tr}}=\{0\}\), so
\[
\dist(U_n(t),\frakS_{\mathrm{tr}})=\norm{U_n(t)}{H}=a_n\to0.
\]
Thus qualitative compactness holds, but the modulus is the arbitrary sequence \(a_n\).

For a fixed finite window \(\Lambda\), eventually \(j_n>\Lambda\), and therefore
\[
P_\Lambda U_n=0,
\qquad
P_\Lambda g_n=0.
\]
Hence every fixed finite-window obstruction test sees nothing.  In particular, every fixed finite-window trace-cost problem is trivially solvable with zero cost.

Finally, for every \(s\in I_-\) the only strict competitor is \(0\).  Therefore
\[
\inf_{s\in I_-}\inf_{V\in\frakS_{\mathrm{tr}}}
\left[\frac12\norm{U_n(s)-V}{H}^2+\kappa_n(s)\right]
=\frac12a_n^2.
\]
Because \(\eta_n\le2a_n^4\),
\[
\frac{m_n}{\eta_n}\ge \frac{1}{4a_n^2}\to\infty.
\]
Thus the subcritical selection estimate \(m_n\lesssim \ell_n^\mu+\ell_n^{-N}\delta_n^b\) fails dramatically.
\end{proof}

\begin{corollary}[No universal old-observable rate]\label{cor:no-old-rate-expanded}
No universal logarithmic or power bound for \(m_n\) can be derived from compactness-level projection, fixed finite-window exactifiability, covariance positivity, prepared residual smallness, and good-time variance control alone.  The conclusion is an insufficiency statement for the abstract old trace skeleton, not a construction of a Navier--Stokes branch.
\end{corollary}

\begin{remark}[What the model violates]
The construction violates several stronger structures.  First, it violates parabolic trace-drop: the high mode \(e_{j_n}\) remains time-independent instead of being heat-damped.  Second, it exploits loss of uniform frequency observability through \(\lambda_{j_n}\to0\).  Third, if \(g_n\) lies in a phantom cokernel direction, then it violates vertical duality: for \(y_n=e_{j_n}\), one could have \(A^*y_n=0\) but
\[
\ip{g_n}{y_n}=\lambda_{j_n}a_n^2>0.
\]
A vertical-duality estimate would force this pairing to vanish.  Thus vertical duality is precisely the condition forbidding active residuals from living in phantom quotient directions.
\end{remark}

\section{Obstruction separation: trace-drop, finite windows, and flat branches}

\label{sec:trace-drop-test}

The previous model survives because it stores mass in high modes without paying any parabolic price.  This section formalizes the first candidate mechanism that kills that naive construction.  The underlying intuition is the parabolic smoothing and trace-decay behavior encoded in the heat operator and analytic semigroup theory \cite{Amann1995,Lunardi1995}.

\subsection{Trace-drop envelope}

Let \(J_\rho\) be a low-frequency projection and set \(P_\rho=I-J_\rho\).  Consider a failed sharp branch
\[
U_n=V_n+\eps_nW_n,
\qquad
\eps_n=m_n^{1/2},
\]
selected at times \(s_n\).  The following axiom abstracts the high-frequency trace-drop produced by heat smoothing.

\begin{assumption}[High-frequency trace-drop]\label{ass:TD}
There are constants \(c_0,c_1>0\) such that whenever, for some \(\rho>0\) and \(\varepsilon_0>0\),
\begin{equation}\label{eq:hf-tail-lower}
\int \phi |P_\rho W_n(s_n)|^2\dx\ge \varepsilon_0,
\end{equation}
there exists a measurable set
\[
T_{n,\rho}\subset (s_n,s_n+c_0\rho^2)
\]
with \(|T_{n,\rho}|\ge c_1\rho^2\) such that, for \(t\in T_{n,\rho}\),
\begin{equation}\label{eq:TD-drop-W}
\frac12\int\phi|W_n(t)|^2\dx
\le
\frac12\int\phi|W_n(s_n)|^2\dx
-c_1\varepsilon_0+o_\rho(1)+o_n(1).
\end{equation}
\end{assumption}

Multiplying \eqref{eq:TD-drop-W} by \(m_n\), one obtains the corresponding selected-trace energy drop:
\begin{equation}\label{eq:TD-drop-U}
\frac12\int\phi|U_n(t)-V_n(t)|^2\dx
\le
\frac12\int\phi|U_n(s_n)-V_n(s_n)|^2\dx
-c_1\varepsilon_0m_n+o(m_n).
\end{equation}

By itself, trace-drop only says that high-frequency energy drops at some future times.  To contradict sharp minimality, those drop times must be admissible comparison times and must still have small variance.

\begin{assumption}[Sharp admissible-time intersection]\label{ass:sharp-intersection}
For every trace-drop set \(T_{n,\rho}\) arising from Assumption~\ref{ass:TD}, there exists a time \(t_n\in T_{n,\rho}\) such that
\begin{enumerate}[label=(\roman*)]
\item \(t_n\) belongs to the admissible sharp comparison set;
\item the same strict shadow \(V_n\), or an admissible strict competitor with the same normalized first variation, may be used at \(t_n\);
\item \(\int\phi\kappa_n(t_n)\dx=o(m_n)\).
\end{enumerate}
\end{assumption}

In concrete arguments, item (iii) is obtained by requiring \(s_n\) to lie in a thick good-time set.  If
\[
\frac1{\rho^2}\int_{s_n}^{s_n+c\rho^2}\int\phi\kappa_n\dxdt\le C\ell_n^2
\]
and \(m_n\gg\ell_n^\mu\) with \(\mu<2\), then \(\ell_n^2=o(m_n)\), and one can select a time in the drop interval with negligible variance.  Assumption~\ref{ass:sharp-intersection} is the nontrivial statement that such a time also belongs to the admissible comparison set.

\subsection{Trace-drop implies trace tightness when admissible times intersect}

\begin{proposition}[Trace-drop plus sharp intersection gives trace tightness]\label{prop:TD-tightness}
Suppose a failed branch satisfies
\[
m_n\gg\eta_n,
\qquad
\eta_n=\ell_n^\mu+\ell_n^{-N}\delta_n^b,
\qquad
0<\mu<1/6,
\]
and satisfies Assumptions~\ref{ass:TD} and \ref{ass:sharp-intersection}.  Then
\begin{equation}\label{eq:trace-tightness-conclusion}
\lim_{\rho\downarrow0}\limsup_{n\to\infty}
\int\phi|W_n(s_n)-J_\rho W_n(s_n)|^2\dx=0.
\end{equation}
\end{proposition}

\begin{proof}
Assume not.  Then there are \(\varepsilon_0>0\), radii \(\rho_k\downarrow0\), and indices \(n_k\to\infty\) such that
\[
\int\phi|P_{\rho_k}W_{n_k}(s_{n_k})|^2\dx\ge\varepsilon_0
\qquad\text{for every }k.
\]
Fix \(k\) and apply Assumption~\ref{ass:TD} to the branch index \(n_k\) and the radius \(\rho_k\).  This gives a drop set \(T_{n_k,\rho_k}\) on which \eqref{eq:TD-drop-U} holds.  By Assumption~\ref{ass:sharp-intersection}, choose \(t_{n_k}\in T_{n_k,\rho_k}\) that is an admissible comparison time and satisfies
\[
\int\phi\kappa_{n_k}(t_{n_k})\dx=o(m_{n_k}).
\]
Using the same strict competitor, or the admissible competitor supplied by the sharp intersection assumption, at \(t_{n_k}\), we get
\[
m_{n_k}
\le
\frac12\int\phi|U_{n_k}(t_{n_k})-V_{n_k}(t_{n_k})|^2\dx
+\int\phi\kappa_{n_k}(t_{n_k})\dx.
\]
The trace-drop estimate gives
\[
m_{n_k}\le m_{n_k}-c_1\varepsilon_0m_{n_k}+o(m_{n_k}),
\]
which is impossible for large \(k\). Therefore the high-frequency tail must vanish in the sense of \eqref{eq:trace-tightness-conclusion}.
\end{proof}

\begin{remark}[What trace-drop does and does not do]
Trace-drop plus admissible-time intersection kills the high-frequency escape model \(U_n=a_ne_{j_n}\).  It does not by itself prove subcritical selection.  Once selected traces are tight, the remaining problem is a finite-window or moving-window singular compatibility problem.
\end{remark}

\label{sec:finite-window-test}

After trace tightness, the next question is whether a fixed finite-dimensional compatibility window can still support arbitrary-slow selected trace convergence.  In the analytic case, it cannot.

\subsection{Fixed-window Lojasiewicz control}

Fix a finite window \(\Lambda\).  Let
\[
H_\Lambda\simeq\R^d,
\qquad
Y_\Lambda\simeq\R^m,
\]
and let
\[
\calC_\Lambda:H_\Lambda\to Y_\Lambda
\]
be the finite-window compatibility map.  Its zero set is
\[
\frakS_\Lambda=\{x\in H_\Lambda:\calC_\Lambda(x)=0\}.
\]
In the Navier--Stokes shadow problem, \(\calC_\Lambda\) is the finite-window projection of the quadratic pressure-compatibility map
\[
\calC(V)=\nabh\partial_3\Delta_h^{-1}\partial_a\partial_b(V_aV_b).
\]
Thus \(\calC_\Lambda\) is analytic, and in model settings polynomial; this is the finite-dimensional setting in which \L{}ojasiewicz-type analytic inequalities are available \cite{BierstoneMilman1988,Kurdyka1998,KrantzParks2002}.

\begin{proposition}[Fixed-window analytic obstruction]\label{prop:fixed-window-lojasiewicz}
Let \(\calC_\Lambda:H_\Lambda\to Y_\Lambda\) be real analytic, and let \(K_\Lambda\subset H_\Lambda\) be compact.  Then there exist constants \(C_\Lambda>0\) and \(\alpha_\Lambda>0\) such that
\begin{equation}\label{eq:lojasiewicz-window}
\dist(x,\frakS_\Lambda)
\le
C_\Lambda\norm{\calC_\Lambda(x)}{Y_\Lambda}^{\alpha_\Lambda}
\qquad
\text{for all }x\in K_\Lambda.
\end{equation}
Consequently, if \(x_n\in K_\Lambda\) and \(\norm{\calC_\Lambda(x_n)}{Y_\Lambda}\le\rho_n\), then
\[
\dist(x_n,\frakS_\Lambda)\le C_\Lambda\rho_n^{\alpha_\Lambda}.
\]
\end{proposition}

\begin{proof}
This is the classical finite-dimensional Lojasiewicz inequality for real analytic maps, applied to the analytic function \(x\mapsto\norm{\calC_\Lambda(x)}{Y_\Lambda}^2\) on a compact neighborhood of its zero set.  The constants may depend on \(\Lambda\) and on \(K_\Lambda\).
\end{proof}

Thus fixed finite-dimensional singularity may worsen the exponent, but it does not permit arbitrary-slow convergence.  For example, \(\calC(x)=x^k\) gives \(\dist(x,\{0\})=|\calC(x)|^{1/k}\).

\subsection{Why analyticity matters}

If the compatibility map were only \(C^\infty\), arbitrary-slow behavior could occur even in dimension one.  Let
\[
\calC(x)=
\begin{cases}
\exp(-1/x^2),& x>0,\\
0,& x\le0.
\end{cases}
\]
Then the zero set is \((-
\infty,0]\), and for \(x_n=a_n>0\),
\[
\dist(x_n,\calC^{-1}(0))=a_n,
\qquad
\calC(x_n)=\exp(-1/a_n^2),
\]
which is smaller than every power of \(a_n\).  The strict pressure-compatibility map is quadratic/analytic, so this smooth-flat pathology is not expected in a fixed window.

\subsection{Finite-window obstruction dichotomy}

Let a trace-tight failed branch be expanded formally as
\[
V^\eta=V+\eta R_1+\eta^2R_2+\eta^3R_3+\cdots.
\]
For the quadratic compatibility form
\[
B(A,B)=\nabh\partial_3\Delta_h^{-1}\partial_a\partial_b(A_aB_b),
\]
the projected obstruction hierarchy in a fixed window is
\begin{align*}
\calO_{2,\Lambda}&=P_\Lambda B(R_1,R_1),\\
\calO_{3,\Lambda}&=P_\Lambda\bigl(2B(R_1,R_2)\bigr),\\
\calO_{4,\Lambda}&=P_\Lambda\bigl(2B(R_1,R_3)+B(R_2,R_2)\bigr),
\end{align*}
and generally
\begin{equation}\label{eq:obstruction-hierarchy}
\calO_{k,\Lambda}=P_\Lambda\sum_{i+j=k}B(R_i,R_j).
\end{equation}
If some \(\calO_{k,\Lambda}\ne0\), then a finite-order visible obstruction appears and the fixed-window analytic mechanism yields a finite-power improvement.  If all \(\calO_{k,\Lambda}=0\) for every fixed \(\Lambda\) and every finite \(k\), the branch is finite-mode flat to all orders.  The remaining obstruction is no longer ordinary finite-dimensional singularity; it is all-order flatness or moving-window exactification.

\label{sec:all-order-flat}

Trace-drop removes high-frequency escape, and fixed-window analyticity removes ordinary finite-dimensional no-rate behavior.  Yet one more abstract obstruction remains: every finite stage may be exactifiable, but the finite-stage constants may have no summable majorant.

\subsection{A non-summable finite-stage model}

Let
\[
H=\R e_1\oplus\ell^2(\{e_2,e_3,\ldots\}),
\]
and write
\[
x=(a,z_2,z_3,\ldots).
\]
Choose positive constants \(\Gamma_k\to\infty\) growing very fast, for example \(\Gamma_k=\exp(k^2)\).  For each finite stage \(K\), define
\begin{equation}\label{eq:finite-stage-SK}
\frakS_K=\{(a,z_2,\ldots,z_K):z_k=\Gamma_ka^k,
\ 2\le k\le K\}.
\end{equation}
The full strict set is the set of all \((a,z_2,z_3,\ldots)\in H\) satisfying
\begin{equation}\label{eq:full-strict-nonsummable}
z_k=\Gamma_ka^k,
\qquad k\ge2.
\end{equation}
If \(\Gamma_k=\exp(k^2)\), then for every \(a\ne0\),
\[
\sum_{k\ge2}\Gamma_k^2|a|^{2k}=\infty.
\]
Thus the full strict set is only
\[
\frakS=\{0\}.
\]
Nevertheless, every finite-stage strict set \(\frakS_K\) is nontrivial.

\begin{theorem}[All-order finite-stage exactification does not imply global selection]\label{thm:nonsummable-flat}
For any sequence \(a_n\downarrow0\), the branch
\[
U_n=a_ne_1
\]
is trace-tight and has selected trace distance
\[
m_n=\frac12a_n^2
\]
from the full strict set \(\frakS=\{0\}\).  However, for every fixed finite stage \(K\), there exists \(\widetilde U_{n,K}\in\frakS_K\) such that
\[
\norm{\widetilde U_{n,K}-U_n}{H}=o(a_n).
\]
Thus every fixed finite stage declares the branch exactifiable to a lower order than the main trace distance, while the full strict curve does not exist for any \(a\ne0\).
\end{theorem}

\begin{proof}
The branch is contained in the one-dimensional span of \(e_1\), so trace tightness is immediate.  Since the full strict set is \(\{0\}\),
\[
\dist(U_n,\frakS)=a_n,
\qquad
m_n=\frac12a_n^2.
\]
For fixed \(K\), define
\[
\widetilde U_{n,K}=
(a_n,\Gamma_2a_n^2,\Gamma_3a_n^3,\ldots,\Gamma_Ka_n^K,0,0,\ldots).
\]
Then \(\widetilde U_{n,K}\in\frakS_K\), and
\[
\norm{\widetilde U_{n,K}-U_n}{H}^2
=
\sum_{k=2}^K\Gamma_k^2a_n^{2k}.
\]
For fixed \(K\), each term is \(o(a_n^2)\), hence the whole sum is \(o(a_n^2)\).  This proves the claim.
\end{proof}

\begin{remark}[Meaning of the model]
The constants \(\Gamma_k\) represent finite-stage right-inverse, Newton, or trace-lifting constants, as in finite-dimensional bifurcation/Newton schemes \cite{Kielhofer2012}.  The theorem shows that the existence of all finite jets does not imply an actual strict curve unless one has a summable majorant or a different mechanism, such as trace-cost exactification or vertical duality.  It is an abstract model of the finite-mode flat surviving branch.
\end{remark}

\section{Trace-cost duality and active residual quotients}\label{sec:trace-cost}

The preceding all-order model shows that strong all-frequency exactification may be too much to ask.  Trace-cost duality is the weaker, selected-time substitute.

\subsection{Finite-dimensional trace cost}

Let \(H\) and \(Y\) be finite-dimensional Hilbert spaces and let
\[
A:H\to Y
\]
be linear.  For \(g\in Y\), define
\begin{equation}\label{eq:trace-cost-def}
\Cost_A^{\trc}(g)=\inf\{\norm{\xi}{H}^2:A\xi=-g\},
\end{equation}
with the convention \(\Cost_A^{\trc}(g)=+\infty\) if \(-g\notin\Range A\).

\begin{lemma}[Trace-cost duality]\label{lem:trace-cost-duality}
For \(g\in Y\), the following are equivalent:
\begin{enumerate}[label=(\roman*)]
\item \(\Cost_A^{\trc}(g)\le r^2\);
\item for every \(y\in Y\),
\begin{equation}\label{eq:dual-bound}
|\ip{g}{y}_Y|\le r\norm{A^*y}{H}.
\end{equation}
\end{enumerate}
Equivalently,
\begin{equation}\label{eq:cost-dual-equality}
\sqrt{\Cost_A^{\trc}(g)}=
\sup_{A^*y\ne0}\frac{|\ip{g}{y}_Y|}{\norm{A^*y}{H}},
\end{equation}
with the value \(+\infty\) if \(g\notin\Range A\).
\end{lemma}

\begin{proof}
If \(A\xi=-g\), then
\[
|\ip{g}{y}_Y|=|\ip{-A\xi}{y}_Y|=|\ip{\xi}{A^*y}_H|
\le\norm{\xi}{H}\norm{A^*y}{H}.
\]
Taking the infimum over admissible \(\xi\) gives one implication.  Conversely, if \eqref{eq:dual-bound} holds, then \(\ip{g}{y}=0\) for all \(y\in\Ker A^*\), so \(g\in(\Ker A^*)^\perp=\Range A\).  Let \(\xi_0\) be the minimal-norm solution of \(A\xi=-g\).  Since \(\xi_0\in\Range A^*\), choose \(y_0\) with \(A^*y_0=\xi_0\).  Then
\[
\norm{\xi_0}{H}^2
=|\ip{\xi_0}{A^*y_0}_H|
=|\ip{-A\xi_0}{y_0}_Y|
=|\ip{g}{y_0}_Y|
\le r\norm{A^*y_0}{H}=r\norm{\xi_0}{H}.
\]
Thus \(\norm{\xi_0}{H}\le r\), and \(\Cost_A^{\trc}(g)\le r^2\).  The equality follows by optimizing over \(r\).
\end{proof}

\subsection{Vertical-duality as anti-phantom quotient control}

For a finite window, an active residual \(g_\Lambda\in Y_\Lambda\) defines an obstruction class
\[
[g_\Lambda]\in Y_\Lambda/\Range A_\Lambda.
\]
The trace cost is finite exactly when this class vanishes, and it is small exactly when \(g_\Lambda\) has a small preimage under \(A_\Lambda\).  By Lemma~\ref{lem:trace-cost-duality}, this is equivalent to
\begin{equation}\label{eq:active-residual-pairing}
|\ip{g_\Lambda}{y}|
\le r_\Lambda\norm{A_\Lambda^*y}{H_\Lambda}.
\end{equation}
Thus small singular values of \(A_\Lambda\) are harmless if the actual residual has proportionally small pairing with the corresponding left singular directions.

\begin{assumption}[Vertical-duality active-residual estimate]\label{ass:VD}
For every finite-window failed-selection branch and every active residual \(g_{a,n}(\eta)\in Y_{a,n,\eta}\), there is a branch-native scale
\begin{equation}\label{eq:branch-native-scale}
r_{a,n}(\eta)=C_a\eta^{K+1}+C_a\rho_n\eps_n^{-p_a}\eta
\end{equation}
such that
\begin{equation}\label{eq:VD-estimate}
|\ip{g_{a,n}(\eta)}{y}|
\le
r_{a,n}(\eta)\norm{A_{a,n,\eta}^*y}{H_{a,n,\eta}}
\end{equation}
for every \(y\in Y_{a,n,\eta}\).
\end{assumption}

By Lemma~\ref{lem:trace-cost-duality}, Assumption~\ref{ass:VD} yields
\[
\Cost^{\trc}_{A_{a,n,\eta}}(g_{a,n}(\eta))\le r_{a,n}(\eta)^2.
\]
Hence vertical duality provides exactly the selected-time correction that the all-order non-summable model lacks.

\section{Navier--Stokes realizability and vertical lift cost}

This section starts the Navier--Stokes-specific part of the paper. The old-envelope no-rate model and the finite-stage flat model are larger than the true Navier--Stokes image. To decide whether the surviving branch is relevant to Navier--Stokes, one must ask whether the horizontal formal branch can be lifted to a full three-dimensional deformation with a small vertical component.

\subsection{The Navier--Stokes deformation complex}

Let \((V,Q)\) be a strict shadow:
\[
V=(V_h,0),
\qquad
\divh V_h=0,
\qquad
\partial_3Q=0,
\]
\[
\partial_tV_h-\Delta V_h+\nabh\cdot(V_h\otimes V_h)+\nabh Q=0.
\]
Consider a formal full Navier--Stokes deformation
\begin{align}
 u_h^\eps&=V_h+\sum_{k\ge1}\eps^k r_k,\label{eq:uh-expansion}\\
 u_3^\eps&=\sum_{k\ge1}\eps^k z_k,\label{eq:u3-expansion}\\
 p^\eps&=Q+\sum_{k\ge1}\eps^k\pi_k.\label{eq:p-expansion}
\end{align}
The incompressibility condition gives, at each order,
\begin{equation}\label{eq:jet-divergence}
\divh r_k+\partial_3z_k=0.
\end{equation}
Thus a horizontal strict jet with nonzero horizontal divergence must be compensated by a vertical jet.  If the required \(z_k\) is large in the vertical budget norm, the horizontal branch is not realizable by a small-\(u_3\) degeneration.

Define
\[
L_Vz=\partial_tz-\Delta z+V_h\cdot\nabh z,
\]
and
\[
L_V^hr=\partial_tr-\Delta r+V_h\cdot\nabh r+r\cdot\nabh V_h.
\]
The vertical momentum equation gives
\begin{equation}\label{eq:vertical-jet-equation}
L_Vz_k+\partial_3\pi_k
=
-\sum_{\substack{i+j=k\\ i,j\ge1}}
\bigl(r_i\cdot\nabh z_j+z_i\partial_3z_j\bigr).
\end{equation}
The horizontal momentum equation gives
\begin{equation}\label{eq:horizontal-jet-equation}
L_V^hr_k+\nabh\pi_k+z_k\partial_3V_h
=
-\sum_{\substack{i+j=k\\ i,j\ge1}}
\bigl(r_i\cdot\nabh r_j+z_i\partial_3r_j\bigr).
\end{equation}
The strict-shadow hierarchy is the special slice in which \(z_k=0\), \(\partial_3\pi_k=0\), and \(\divh r_k=0\).  A formal strict branch may therefore be larger than the class of branches realizable by full Navier--Stokes deformations.

\subsection{Vertical lift cost}

The strict compatibility obstruction at order \(k\) is a horizontal pressure obstruction.  In the full three-dimensional system, it need not vanish; it may instead be absorbed by \(z_k\) through \eqref{eq:vertical-jet-equation}.  This motivates the following definition.

\begin{definition}[Finite-order vertical lift cost]\label{def:VCost}
Fix a strict base \((V,Q)\) and horizontal jets \(\calR_{\le k}=(r_1,\ldots,r_k)\).  The order-\(k\) vertical lift cost is
\begin{equation}\label{eq:VCost-def}
\mathrm{VCost}_{V,k}(\calR_{\le k})
=
\inf
\sum_{j=1}^k\norm{z_j}{\calZ}^{\alpha_j},
\end{equation}
where the infimum is taken over vertical jets \((z_1,\ldots,z_k)\) and pressure jets \((\pi_1,\ldots,\pi_k)\) satisfying \eqref{eq:jet-divergence}, \eqref{eq:vertical-jet-equation}, and \eqref{eq:horizontal-jet-equation} for \(1\le j\le k\).  Here \(\calZ\) is a vertical budget norm compatible with the one-component quantity \(C_3\).
\end{definition}

If \(\mathrm{VCost}_{V,k}=+\infty\), the horizontal jets cannot be realized even formally to order \(k\).  If the costs grow too fast to admit a summable majorant, then all finite jets may exist without a genuine small-vertical-component family.

\begin{definition}[NS-realizable horizontal jet branch]\label{def:NS-realizable-branch}
An infinite horizontal jet branch \((r_k)_{k\ge1}\) is \(\NS\)-realizable with vertical budget \(\delta\) at amplitude \(\eps\) if there exist vertical jets \((z_k)_{k\ge1}\) and pressure jets \((\pi_k)_{k\ge1}\) solving the jet equations to all orders such that
\[
\sum_{k\ge1}\eps^k\norm{z_k}{\calZ}<\infty
\]
and the resulting vertical component
\[
u_3^\eps=\sum_{k\ge1}\eps^kz_k
\]
satisfies
\[
\norm{u_3^\eps}{L^3}^3\le\delta.
\]
\end{definition}

\subsection{Zero-shadow quadratic calculation}

The first nontrivial test occurs at second order.  Let \(V=0\), \(Q=0\), and let the first horizontal jet be \(r_1=W\).  For a strict first tangent one has
\[
\divh W=0,
\qquad
z_1=0,
\qquad
\partial_3\pi_1=0.
\]
At second order, the pressure generated by the horizontal quadratic interaction satisfies formally
\begin{equation}\label{eq:pi2-poisson}
-\Delta \pi_2=\partial_a\partial_b(W_aW_b)
\end{equation}
(up to the chosen pressure gauge and lower-order corrections).  The strict 2.5D condition would require \(\partial_3\pi_2=0\), equivalently the quadratic strict obstruction
\[
B(W,W)=\nabh\partial_3\Delta_h^{-1}\partial_a\partial_b(W_aW_b)
\]
to vanish in the relevant quotient.  Full Navier--Stokes allows instead
\begin{equation}\label{eq:z2-lift-zero}
(\partial_t-\Delta)z_2+\partial_3\pi_2=0.
\end{equation}
Thus the quadratic strict obstruction is converted into the cost of solving \eqref{eq:z2-lift-zero} with small \(z_2\).

A model second-order vertical cost is therefore
\begin{equation}\label{eq:VCost2}
\mathrm{VCost}_2(W)=
\inf\{\norm{z_2}{\calZ}: (\partial_t-\Delta)z_2=-\partial_3\pi_2[W,W]\}.
\end{equation}
If \(B(W,W)\ne0\) and \(\mathrm{VCost}_2(W)\) is bounded below at the sharp scale, then the direction cannot occur in a degeneration with sufficiently small \(u_3\).  If \(B(W,W)=0\), the obstruction must be sought at higher order.

\subsection{Symbol-level conclusion}
In the periodic zero-shadow test, the strict quadratic obstruction and the full vertical lift have different symbols. For a horizontal frequency \(K\), vertical frequency \(m\), and horizontal trace normalized by \(\|W\|_{L^2}\sim1\), the strict obstruction has size
\[
  \|B(W,W)\|_{L^2}\sim mK,
\]
whereas the canonical second vertical lift satisfies
\[
  \|z_2\|_{L^2}\sim \frac{mK^2}{(K^2+m^2)^2}.
\]
Thus a strict pressure obstruction is not automatically an NS-realizability obstruction. It becomes one only if the associated vertical lift cost is too large for the available \(C_3\)-budget. The detailed Fourier computation is given in \Cref{app:vertical-lift}.

\section{Reduced vertical duality and Schur trace-projectability}

The preceding section shows that raw strict pressure obstruction and full vertical lift cost are not identical. This section formulates the correct reduced target. Vertical duality is not a raw factorization on all quotient directions; it is a residual-level anti-phantom statement after all finite-order visible obstructions have been removed.

\subsection{A periodic zero-shadow finite-window factorization test}\label{subsec:raw-factorization-obstruction}

The preceding discussion also identifies a trap.  One should not require vertical-duality factorization on the raw zero-shadow second-order quotient.  In that raw quotient the trace-defect map is zero, while the vertical adjoint lifting is generally nonzero.

Work again on \(\mathbb T^3\) in a fixed finite Fourier window and at the zero strict base \((V,Q)=(0,0)\).  The raw second-order strict quotient is generated by
\[
g=B(W,W)=\nabla_h\partial_3(-\Delta_h)^{-1}F(W).
\]
For \(y\in Y_{\mathrm{raw}}\), the strict-to-vertical adjoint lifting is, on active modes,
\begin{equation}\label{eq:Browsymbol}
\widehat{B_{\mathrm{raw}}^*y}(\xi)
=
\frac{|\xi|^2}{|\xi_h|^2}
\,i\xi_h\cdot\widehat y(\xi),
\qquad
\xi_h\ne0.
\end{equation}
On modes with \(\xi_3=0\), the vertical obstruction is absent and one may set \(B_{\mathrm{raw}}^*y=0\).

By contrast, the raw trace-defect map at the zero shadow is
\[
A_{\mathrm{raw}}=D\calC_0.
\]
Since \(\calC(V)=B(V,V)\) is quadratic,
\begin{equation}\label{eq:Araw-zero}
A_{\mathrm{raw}}=D\calC_0=0,
\qquad
A_{\mathrm{raw}}^*=0.
\end{equation}
Therefore an exact raw factorization
\[
B_{\mathrm{raw}}^*=M A_{\mathrm{raw}}^*
\]
would force \(B_{\mathrm{raw}}^*=0\), which is false in general by \eqref{eq:Browsymbol}.  We record the conclusion explicitly.

\begin{proposition}[Raw zero-shadow factorization obstruction]\label{prop:raw-zero-factorization-obstruction}
In the periodic zero-shadow finite-window model, the raw second-order trace-defect map satisfies \(A_{\mathrm{raw}}=0\), while the strict-to-vertical adjoint lifting \(B_{\mathrm{raw}}^*\) is generally nonzero and has Fourier symbol \eqref{eq:Browsymbol}.  Hence no exact factorization
\[
B_{\mathrm{raw}}^*=M A_{\mathrm{raw}}^*
\]
can hold on the raw second-order quotient.
\end{proposition}

\begin{proof}
The identity \(A_{\mathrm{raw}}=0\) follows from the quadratic nature of \(\calC(V)=B(V,V)\), hence \(D\calC_0=0\).  Formula \eqref{eq:Browsymbol} follows from the equality
\[
\ip{\nabla_h\partial_3(-\Delta_h)^{-1}F}{y}
=
\ip{\partial_3(-\Delta)^{-1}F}{B_{\mathrm{raw}}^*y}
\]
on active Fourier modes.  Choosing a dual vector \(y\) with \(i\xi_h\cdot\widehat y(\xi)\ne0\) on an active mode gives \(B_{\mathrm{raw}}^*y\ne0\).  Since \(A_{\mathrm{raw}}^*=0\), no factorization through \(A_{\mathrm{raw}}^*\) can hold on the raw quotient.
\end{proof}

This failure is not a failure of the vertical-duality route.  It says that raw second-order zero-shadow obstructions are visible finite-order obstructions.  They should be handled by the finite-order obstruction and Lojasiewicz selection mechanism, not by vertical duality.  Vertical duality is meant only for the genuinely surviving active quotient after all visible finite-order obstruction directions have been removed.

We therefore refine the quotient before asking for any vertical-duality statement.

\subsection{Reduced active quotients after finite-order obstruction removal}\label{subsec:reduced-active-quotient}

Let \((V,Q)\in\frakS_M\) be a strict base and let
\[
V_\eta^{(K)}=V+\eta R_1+\eta^2R_2+\cdots+\eta^K R_K
\]
be a finite-stage formal strict branch.  Expanding the strict compatibility map gives
\begin{equation}\label{eq:compat-expansion-K}
\calC(V_\eta^{(K)})=
\sum_{j=1}^{K}\eta^j\calO_j+
\eta^{K+1}\calR_{K+1}(\eta).
\end{equation}
For a zero strict base, \(\calC(V)=B(V,V)\), so
\[
\calO_1=0,
\qquad
\calO_2=B(R_1,R_1),
\]
\[
\calO_3=B(R_1,R_2)+B(R_2,R_1),
\]
and, in general,
\begin{equation}\label{eq:Ok-zero-shadow}
\calO_k=\sum_{i+j=k}B(R_i,R_j).
\end{equation}
At a nonzero base the same construction applies, but the first coefficient is
\(D\calC_V[R_1]\) and the later coefficients contain the corresponding moving-base lower-order terms.

Fix a finite window \(\Lambda\), and let \(Y_\Lambda\) be the active compatibility quotient.  Write
\[
\calO_{j,\Lambda}=P_\Lambda\calO_j\in Y_\Lambda.
\]
The visible obstruction subspace up to order \(K\) is
\begin{equation}\label{eq:YvisK}
Y_{\mathrm{vis}}^{(K)}
:=
\operatorname{span}\{\calO_{1,\Lambda},\ldots,\calO_{K,\Lambda}\}
\subset Y_\Lambda.
\end{equation}
These are precisely the directions that have already appeared at finite order.  If one of them is nonzero, it is a finite-order visible obstruction and should be handled by finite-dimensional analytic geometry and the selection argument.  It should not be assigned to vertical duality.

The reduced active quotient is
\begin{equation}\label{eq:YredK}
Y_{\mathrm{red}}^{(K)}:=Y_\Lambda/Y_{\mathrm{vis}}^{(K)},
\qquad
\Pi_{\mathrm{red}}^{(K)}:Y_\Lambda\to Y_{\mathrm{red}}^{(K)}.
\end{equation}
Equivalently, its dual is the annihilator
\begin{equation}\label{eq:YredK-dual}
(Y_{\mathrm{red}}^{(K)})^*
=
\{y\in Y_\Lambda^*:
\ip{\calO_{j,\Lambda}}{y}=0,
\ 1\le j\le K\}.
\end{equation}
Thus, for every \(y\in (Y_{\mathrm{red}}^{(K)})^*\), the low-order pairings vanish and
\begin{equation}\label{eq:low-order-annihilation}
\ip{\calC(V_\eta^{(K)})}{y}
=
\eta^{K+1}\ip{\calR_{K+1}(\eta)}{y}.
\end{equation}
This is the algebraic reason why vertical duality is meaningful only after finite-order obstruction directions have been removed.

We define the reduced branch-native residual by
\begin{equation}\label{eq:gK-def}
g_K(\eta)
:=
\Pi_{\mathrm{red}}^{(K)}
\left[\eta^{-(K+1)}\calC(V_\eta^{(K)})\right]
\in Y_{\mathrm{red}}^{(K)}.
\end{equation}
This is the surviving active residual.  It is not a raw obstruction coefficient; it is what remains after quotienting out the visible finite-order directions.

We also define the reduced trace-defect map.  Let \(H_K\) be the selected-time trace correction space, and let
\[
\calL_K:H_K\to X_\Lambda
\]
be a finite-stage trace-lifting map.  For \(\xi\in H_K\), set \(h=\calL_K\xi\) and consider
\[
V_{\eta,\xi}^{(K)}=V_\eta^{(K)}+\eta^{K+1}h.
\]
Then
\begin{equation}\label{eq:AK-def}
A_K\xi
:=
\Pi_{\mathrm{red}}^{(K)}
D\calC_{V_\eta^{(K)}}[\calL_K\xi]
\in Y_{\mathrm{red}}^{(K)}.
\end{equation}
The adjoint
\[
A_K^*:(Y_{\mathrm{red}}^{(K)})^*\to H_K
\]
measures which reduced dual quotient directions are visible at the selected trace.  If \(A_K^*y=0\), then \(y\) is a surviving trace-invisible dual direction.

The vertical adjoint lifting on the reduced quotient is a map
\[
B_K^*:(Y_{\mathrm{red}}^{(K)})^*\to\calV_K
\]
chosen so that the reduced residual pairing has the representation
\begin{equation}\label{eq:reduced-residual-representation}
\ip{g_K(\eta)}{y}
=
\ip{\calV_{\NS,K}(\eta)}{B_K^*y}
+
\operatorname{Comm}_K(y).
\end{equation}
Here \(\calV_{\NS,K}(\eta)\) is the \(K\)-stage vertical momentum residual generated by the Navier--Stokes branch, while \(\operatorname{Comm}_K\) collects cutoff, projection, moving-base, and gauge remainders.

The correct reduced vertical-duality target is therefore
\begin{equation}\label{eq:reduced-VD-target}
|\ip{g_K(\eta)}{y}|
\le
r_K(\eta)\norm{A_K^*y}{H_K}
\qquad
\forall y\in (Y_{\mathrm{red}}^{(K)})^*.
\end{equation}
A sufficient operator-level statement is a finite-stage factorization
\begin{equation}\label{eq:BK-factor-sufficient}
B_K^*y=M_KA_K^*y+R_K^*y,
\end{equation}
with
\[
\left|\ip{\calV_{\NS,K}(\eta)}{R_K^*y}\right|
+|\operatorname{Comm}_K(y)|
\le
r_K(\eta)\norm{A_K^*y}{H_K},
\]
and \(\norm{\calV_{\NS,K}(\eta)}{\calV_K'}\le r_K(\eta)\).  Then \eqref{eq:reduced-residual-representation} implies \eqref{eq:reduced-VD-target}.  By the trace-cost lemma, \eqref{eq:reduced-VD-target} gives
\begin{equation}\label{eq:reduced-trace-cost-consequence}
g_K(\eta)\in\Range A_K,
\qquad
\Cost^{\trc}(g_K(\eta))\le C r_K(\eta)^2.
\end{equation}
Thus the surviving active residual can be removed by a selected-time correction of small trace cost.

\subsection{Recursive obstruction removal and reduced vertical duality}\label{subsec:recursive-reduced-vd}

The preceding construction yields a general finite-stage scheme.  Let \(V_\eta^{(K)}\) be a \(K\)-stage formal branch and let \(\calO_1,\ldots,\calO_K\) be the obstruction coefficients in \eqref{eq:compat-expansion-K}.  At each stage exactly one of the following alternatives occurs.

\begin{proposition}[Recursive obstruction-removal alternative]\label{prop:recursive-obstruction-removal}
Fix a finite window \(\Lambda\) and a finite stage \(K\).  For the branch \(V_\eta^{(K)}\), either:
\begin{enumerate}[label=(\roman*)]
\item there exists \(1\le j\le K\) such that \(\calO_{j,\Lambda}\ne0\) in the active quotient; in this case a finite-order visible obstruction is present and finite-dimensional analytic control gives a finite-power selection improvement at this stage; or
\item all visible obstruction coefficients up to order \(K\) vanish in the active quotient, and the only object left for trace-cost analysis is the reduced residual \(g_K(\eta)\in Y_{\mathrm{red}}^{(K)}\).
\end{enumerate}
In the second case, if the reduced VD estimate \eqref{eq:reduced-VD-target} holds, then \(g_K(\eta)\) admits a trace correction with cost bounded by \(C r_K(\eta)^2\).
\end{proposition}

\begin{proof}
The alternative is tautological from the definition of the span \(Y_{\mathrm{vis}}^{(K)}\).  If a nonzero \(\calO_{j,\Lambda}\) appears, then a finite-dimensional analytic Lojasiewicz inequality on the fixed window gives a finite-power relation between compatibility residual and distance to the finite-window zero set.  This is precisely the finite-order obstruction branch.  If no such coefficient appears, then the annihilator formula \eqref{eq:YredK-dual} gives \eqref{eq:low-order-annihilation}, and the surviving residual is exactly \(g_K\).  Finally, the trace-cost consequence follows from the finite-dimensional duality lemma applied to \(A_K:H_K\to Y_{\mathrm{red}}^{(K)}\).
\end{proof}

Passing \(K\to\infty\), a genuinely surviving finite-mode flat branch is one for which
\[
\calO_j=0
\qquad
\text{for every finite }j
\]
in every fixed active window.  Then either the formal strict series converges and gives an exact strict curve, or the formal series is non-summable and one must use finite-stage trace-cost exactification.  Reduced vertical duality is the mechanism that makes this finite-stage exactification possible without a global all-frequency inverse.

\subsection{Finite-window residual anti-phantom criterion}\label{subsec:residual-antiphantom-criterion}

There is one final correction to the formulation.  The operator-level factorization \(B_K^*=M_KA_K^*\) is generally too strong.  The exact finite-dimensional condition needed for vertical duality is a residual-level range inclusion.

Let \(H,Y,\calV\) be finite-dimensional Hilbert spaces, with
\[
A:H\to Y,
\qquad
B^*:Y^*\to\calV.
\]
Let \(B:\calV^*\to Y\) denote the adjoint map, so that
\[
\ip{B\calV}{y}_Y=\ip{\calV}{B^*y}_{\calV',\calV}.
\]
In the clean periodic model, the residual representation has no commutator, and hence
\begin{equation}\label{eq:gBcalV}
g=B\calV.
\end{equation}
The reduced VD inequality
\begin{equation}\label{eq:abstract-reduced-VD}
|\ip{g}{y}|
\le r\norm{A^*y}{H}
\qquad
\forall y\in Y^*
\end{equation}
is equivalent to
\begin{equation}\label{eq:range-inclusion}
g\in\Range A
\qquad\text{and}\qquad
\inf_{A\xi=g}\norm{\xi}{H}\le r.
\end{equation}
Indeed, \((\Range A)^\perp=\Ker A^*\), so \eqref{eq:abstract-reduced-VD} first says that \(g\) has no component in \(\Ker A^*\), and then bounds the minimal preimage.

\begin{proposition}[Finite-window residual anti-phantom criterion]\label{prop:finite-window-antiphantom}
Assume the clean finite-dimensional residual representation \(g=B\calV\).  Then \eqref{eq:abstract-reduced-VD} holds if and only if
\begin{equation}\label{eq:range-criterion-BV}
B\calV\in\Range A
\end{equation}
and
\begin{equation}\label{eq:pseudoinverse-bound}
\norm{A^\dagger B\calV}{H}\le r,
\end{equation}
where \(A^\dagger\) is the minimal-norm right inverse on \(\Range A\).  Equivalently,
\[
P_{\Ker A^*}B\calV=0
\]
and the minimal preimage is bounded by \(r\).
\end{proposition}

\begin{proof}
If \eqref{eq:abstract-reduced-VD} holds and \(y\in\Ker A^*\), then \(\ip{g}{y}=0\).  Hence \(g\perp\Ker A^*=(\Range A)^\perp\), so \(g\in\Range A\).  The inequality then gives the operator norm of the linear functional \(A^*y\mapsto \ip{g}{y}\), which is exactly the minimal-norm preimage bound.  Conversely, if \(g=A\xi\) with \(\norm{\xi}{H}\le r\), then
\[
|\ip{g}{y}|=|\ip{\xi}{A^*y}|
\le r\norm{A^*y}{H}.
\]
Taking the minimal such \(\xi\) gives \eqref{eq:pseudoinverse-bound}.
\end{proof}

Thus vertical duality is not an automatic property of the strict quotient geometry.  It is a property of the Navier--Stokes-derived residual.  Strong operator factorization would say
\[
\Range B\subseteq\Range A
\]
for all possible vertical residuals.  This is usually too strong.  What is needed is only
\begin{equation}\label{eq:NS-derived-range-inclusion}
B_K\calV_{\NS,K}(\eta)
\in\Range A_K
\end{equation}
with a small minimal preimage.  In dual language,
\begin{equation}\label{eq:phantom-projection-zero}
P_{\Ker A_K^*}g_K(\eta)=0.
\end{equation}
The updated theorem target is therefore the following range-inclusion form of reduced vertical duality.

\begin{problem}[NS-derived residual range inclusion]\label{prob:NS-derived-range-inclusion}
For every finite stage \(K\), finite window \(\Lambda\), and \(\NS\)-derived surviving branch, prove
\[
g_K(\eta)\in\Range A_K
\]
and
\[
\inf_{A_K\xi=-g_K(\eta)}\norm{\xi}{H_K}
\le r_K(\eta).
\]
Equivalently, prove
\[
|\ip{g_K(\eta)}{y}|
\le r_K(\eta)\norm{A_K^*y}{H_K}
\qquad
\forall y\in(Y_{\mathrm{red}}^{(K)})^*.
\]
This is reduced vertical duality stated as a residual range-inclusion theorem rather than as an overly strong operator factorization theorem.
\end{problem}

The natural primal explanation is finite-stage \(\NS\)-realizability.  If the branch admits a full Navier--Stokes deformation tail
\[
u^\eta=V_\eta^{(K)}+\eta^{K+1}h_{\mathrm{tail}}+\cdots
\]
whose equation is exact to order \(K+1\), then the selected trace
\[
\xi_{\mathrm{tail}}=h_{\mathrm{tail}}(s_*)
\]
should satisfy
\[
A_K\xi_{\mathrm{tail}}=-g_K(\eta).
\]
This would imply \(g_K(\eta)\in\Range A_K\) directly.  In this sense, residual-level vertical duality is the finite-stage shadow of full Navier--Stokes realizability.

\subsection{Vertical trace-projectability from finite-stage NS realizability}\label{subsec:trace-projectability-from-realizability}

The range-inclusion formulation clarifies one further point.  Finite-stage Navier--Stokes realizability does not, by itself, immediately imply
\[
        g_K(\eta)\in \Range A_K .
\]
It gives a larger balance involving both a horizontal trace correction and a vertical lift.  This distinction is important, because the strict trace-cost method can only use the horizontal trace-defect range, whereas a full three-dimensional Navier--Stokes correction also contains a vertical component.

Fix a finite stage and finite window, and write
\[
        Y_K=Y_{\mathrm{red}}^{(K)},
        \qquad H_K=\text{selected-time trace correction space},
        \qquad Z_K=\text{vertical lift space}.
\]
The reduced active residual is
\[
        g_K(\eta)=
        \Pi_{\mathrm{red}}^{(K)}\eta^{-(K+1)}\calC(V_\eta^{(K)})
        \in Y_K,
\]
and the strict trace-defect map is
\[
        A_K:H_K\to Y_K.
\]
A full Navier--Stokes finite-stage tail also contains a vertical variable
\[
        z\in Z_K.
\]
Its contribution to the reduced strict quotient will be denoted by
\[
        B_K:Z_K\to Y_K.
\]
Thus the finite-stage full Navier--Stokes balance has the schematic form
\begin{equation}\label{eq:finite-stage-balance-A-B}
        g_K+A_K\xi+B_Kz+\mathcal N_K(\xi,z)=0,
\end{equation}
where \(\mathcal N_K\) denotes higher-order or branch-native truncation terms.

Equation \eqref{eq:finite-stage-balance-A-B} does not imply \(g_K\in\Range A_K\) unless the vertical contribution is itself trace-projectable.  Let
\[
        P_K^\perp:Y_K\to(\Range A_K)^\perp=\Ker A_K^*
\]
be the phantom projection.  Projecting \eqref{eq:finite-stage-balance-A-B} gives
\[
        P_K^\perp g_K+P_K^\perp B_Kz+P_K^\perp\mathcal N_K(\xi,z)=0.
\]
Since \(P_K^\perp A_K\xi=0\), strict trace corrections cannot affect this equation.  Therefore the actual obstruction is the vertical trace-projectability defect
\begin{equation}\label{eq:vertical-trace-projectability-defect}
        \mathfrak p_K(z):=P_K^\perp B_Kz
        \in (\Range A_K)^\perp.
\end{equation}
If \(\mathfrak p_K(z)\ne0\), the vertical lift creates a strict quotient component invisible to selected-time trace correction.

\begin{proposition}[Trace-projectable NS realizability implies residual range inclusion]\label{prop:trace-projectable-realizability}
Assume that there exist \(\xi\in H_K\), \(z\in Z_K\), and \(\mathcal N_K\in Y_K\) such that
\[
        g_K+A_K\xi+B_Kz+\mathcal N_K=0.
\]
Assume further that
\[
        B_Kz+\mathcal N_K\in\Range A_K.
\]
Then \(g_K\in\Range A_K\).  More precisely, if \(A_K\zeta=B_Kz+\mathcal N_K\), then
\[
        g_K=-A_K(\xi+\zeta),
\]
and hence
\[
        \Cost^{\trc}(g_K)
        \le \norm{\xi+\zeta}{H_K}^2.
\]
\end{proposition}

\begin{proof}
The proof is immediate from the balance equation.  Substituting \(B_Kz+\mathcal N_K=A_K\zeta\) gives
\[
        g_K+A_K\xi+A_K\zeta=0.
\]
Thus \(g_K=-A_K(\xi+\zeta)\), and \(-(\xi+\zeta)\) is an admissible trace preimage of \(g_K\).  Taking its squared norm gives the trace-cost bound.
\end{proof}

The dual version is the reduced vertical-duality inequality
\begin{equation}\label{eq:vertical-trace-projectability-dual}
        |\ip{B_Kz+\mathcal N_K}{y}|
        \le r_K(\eta)\norm{A_K^*y}{H_K}
        \qquad \forall y\in Y_K^*.
\end{equation}
Thus the genuinely missing mechanism is not merely the existence of a vertical lift.  It is the assertion that the vertical lift can be traded for a strict selected-time trace correction at small trace cost.  We call this \emph{vertical trace-projectability}.

\begin{problem}[Vertical trace-projectability of NS-realizable tails]\label{prob:vertical-trace-projectability}
For every finite window and every finite stage, prove that the finite-stage vertical lift contribution of an \(\NS\)-derived surviving branch satisfies
\[
        B_Kz_K+\mathcal N_K\in\Range A_K
\]
up to an error whose trace cost is bounded by the branch-native scale.  Equivalently, prove the dual estimate \eqref{eq:vertical-trace-projectability-dual}.
\end{problem}

\subsection{Periodic finite-window trace-projectability and the Schur phantom map}\label{subsec:periodic-schur-map}

The preceding discussion gives a concrete finite-dimensional test.  In a clean periodic finite-window model all spaces are finite-dimensional.  Let
\[
        A:H\to Y,
        \qquad B:Z\to Y,
\]
where \(H\) is the trace correction space, \(Z\) is the vertical lift space, and \(Y\) is the reduced active quotient.  The vertical trace-projectability condition is
\[
        B(Z^{\NS})\subseteq\Range A,
\]
where \(Z^{\NS}\subseteq Z\) is the subspace of vertical responses that actually arise from finite-stage Navier--Stokes forcing.

In a periodic zero-shadow Stokes symbol calculation, the homogeneous tail system
\[
        i\xi_h\cdot h+i\xi_3z=0,
\]
\[
        (i\omega+|\xi|^2)z+i\xi_3q=0,
        \qquad
        (i\omega+|\xi|^2)h+i\xi_hq=0
\]
has no nonzero solution for \(|\xi|\ne0\).  Thus the relevant vertical space is not a homogeneous kernel.  It is the image of the lower-order forcing through the full Stokes response.  We write this as
\[
        z=C_K\beta,
\]
where \(\beta\) denotes the finite-stage forcing parameters and
\[
        C_K:\mathcal F_K\to Z_K
\]
is the finite-window vertical Stokes response operator.

Consequently the object to test is the finite-dimensional Schur phantom map
\begin{equation}\label{eq:Schur-phantom-map}
        \mathfrak S_K
        :=P_{\Ker A_K^*}B_KC_K:
        \mathcal F_K\to (\Range A_K)^\perp.
\end{equation}
Here \(P_{\Ker A_K^*}\) is the orthogonal projection onto the phantom cokernel.

\begin{proposition}[Schur range criterion and trace-cost bound]\label{prop:Schur-criterion}
In the periodic finite-window reduced model, the Schur cancellation condition
\[
        \mathfrak S_K\beta=0
\]
for every \(\NS\)-derived forcing parameter \(\beta\) is equivalent to the \emph{range-inclusion} part of vertical trace-projectability:
\[
        B_KC_K\beta\in\Range A_K.
\]
Equivalently,
\[
        P_{\Ker A_K^*}B_KC_K\beta=0.
\]
The full reduced vertical-duality estimate additionally requires a quantitative right-inverse bound:
\[
        \norm{A_K^\dagger B_KC_K\beta}{H_K}
        \le r_K(\eta),
\]
up to the branch-native error scale.  Equivalently, the complete trace-cost statement is
\[
        |\ip{B_KC_K\beta}{y}|
        \le r_K(\eta)\norm{A_K^*y}{H_K}
        \qquad \forall y\in Y_K^*.
\]
Thus Schur cancellation removes the phantom projection, while reduced vertical duality also controls the minimal trace preimage.
\end{proposition}

This reformulation is useful because \(A_K\), \(B_K\), and \(C_K\) are matrices in a fixed finite Fourier window.  The vague phrase ``prove VD'' is thus reduced to two concrete tasks: computing or annihilating the Schur map \eqref{eq:Schur-phantom-map}, and proving the corresponding Moore--Penrose right-inverse bound on the NS-derived residual.

\subsection{Does finite-mode flatness force Schur cancellation?}\label{subsec:flatness-vs-schur}

The preceding calculation raises a natural question.  If a branch is already finite-mode flat in the strict compatibility hierarchy, must it satisfy the Schur cancellation condition
\[
        P_{\Ker A_K^*}B_KC_KF_K=0?
\]
In general the answer is no without additional structure.

Strict finite-mode flatness controls the coefficients in the strict compatibility expansion
\[
        \calC(V_\eta)
        =\sum_{j\ge1}\eta^j\calO_j^{\str},
\]
where, in the zero-shadow quadratic case,
\begin{equation}\label{eq:strict-obstruction-coefficients}
        \calO_k^{\str}
        =\sum_{i+j=k}B(R_i,R_j).
\end{equation}
It requires
\[
        P_\Lambda\calO_k^{\str}=0
        \qquad \text{for every fixed }\Lambda\text{ and every }k.
\]
Schur cancellation controls a different object:
\begin{equation}\label{eq:Schur-obstruction-coeff}
        \calO_K^{\mathrm{Schur}}
        :=P_{\Ker A_K^*}B_KC_KF_K.
\end{equation}
This object measures whether the full Navier--Stokes vertical response is trace-projectable in the reduced strict quotient.  It is a range-alignment condition, not merely the vanishing of a strict compatibility coefficient.

A simple finite-dimensional example makes the logical independence clear.  Let
\[
        Y=\R^2,
        \qquad H=\R,
        \qquad Ah=(h,0).
\]
Then \(\Range A=\operatorname{span}\{(1,0)\}\) and \(\Ker A^*=\operatorname{span}\{(0,1)\}\).  Suppose all strict obstruction coefficients vanish:
\[
        \calO_j^{\str}=0
        \qquad \forall j.
\]
But take a vertical Schur contribution
\[
        BCF=(0,1).
\]
Then
\[
        P_{\Ker A^*}BCF=(0,1)\ne0.
\]
Thus strict flatness alone does not imply Schur flatness.

There are, however, three mechanisms by which strict flatness could imply Schur cancellation in special situations.

\begin{enumerate}[label=(\alph*)]
\item \textbf{Schur response equals a next strict obstruction modulo trace range.}  If one proves
\[
        B_KC_KF_K
        \equiv \calO_{K+1}^{\str}
        \pmod{\Range A_K},
\]
then all-order strict flatness implies trace-projectability.

\item \textbf{Surviving-branch equations impose amplitude ratios.}  In the two-mode model, Schur cancellation requires the amplitude ratio \eqref{eq:amplitude-ratio-schur}.  If the equations \(\calO_j^{\str}=0\) force that ratio, then Schur cancellation follows.

\item \textbf{The Schur defect is itself a visible obstruction.}  If \(P_{\Ker A_K^*}B_KC_KF_K\ne0\), it may have to be added to the visible obstruction space and removed before the next reduced quotient is formed.
\end{enumerate}

This suggests a refinement of the obstruction hierarchy.  In addition to strict compatibility obstructions \(\calO_j^{\str}\), one should track Schur trace-projectability obstructions
\[
        \calO_j^{\mathrm{Schur}}
        =P_{\Ker A_j^*}B_jC_jF_j.
\]
A genuinely surviving branch should be both strict-flat and Schur-flat:
\begin{equation}\label{eq:strict--Schur-flat}
        \calO_j^{\str}=0,
        \qquad
        \calO_j^{\mathrm{Schur}}=0,
        \qquad
        \forall j.
\end{equation}

\begin{problem}[Strict--Schur obstruction hierarchy]\label{prob:strict--Schur-hierarchy}
Develop a two-layer obstruction hierarchy in which finite-order strict compatibility obstructions and finite-order Schur trace-projectability obstructions are removed recursively.  Prove that a failed branch either exhibits a visible obstruction of one of these two types, or else is strict--Schur flat to all finite orders and must be treated by a genuinely \(\NS\)-specific realizability principle or by a relaxed-shadow comparison framework.
\end{problem}

\subsection{Strict--Schur two-layer obstruction hierarchy}\label{subsec:strict--Schur-two-layer}

The preceding subsection shows that strict finite-mode flatness is not the same as trace-projectability of the full Navier--Stokes vertical response.  We therefore refine the obstruction hierarchy by separating two layers:
\[
        \text{strict compatibility obstructions}
        \qquad\text{and}\qquad
        \text{Schur trace-projectability obstructions}.
\]
A genuinely surviving branch should be flat with respect to both layers.

The first layer is the strict compatibility hierarchy.  For a formal strict branch
\[
        V_\eta^{(K)}=V+\eta R_1+\eta^2R_2+\cdots+\eta^KR_K,
\]
write
\begin{equation}\label{eq:strict-expansion-v6}
        \calC(V_\eta^{(K)})
        =\sum_{j=1}^{K}\eta^j\calO_j^{\str}
        +\eta^{K+1}\calR_{K+1}^{\str}(\eta).
\end{equation}
In the zero-shadow quadratic model,
\[
        \calO_1^{\str}=0,
        \qquad
        \calO_2^{\str}=B(R_1,R_1),
\]
\[
        \calO_3^{\str}=B(R_1,R_2)+B(R_2,R_1),
\]
\[
        \calO_4^{\str}=B(R_1,R_3)+B(R_3,R_1)+B(R_2,R_2),
\]
and, in general,
\begin{equation}\label{eq:Ok-str-general}
        \calO_k^{\str}=\sum_{i+j=k}B(R_i,R_j).
\end{equation}
This layer detects failure of the strict pressure condition \(\partial_3Q=0\) inside the strict shadow geometry.  A nonzero \(\calO_k^{\str}\) is a finite-order visible strict obstruction and should be handled by finite-window analytic geometry rather than by vertical duality.

The second layer measures whether the full Navier--Stokes vertical response is projectable onto the selected-time strict trace-defect range.  Fix the reduced quotient at stage \(K\):
\[
        Y_K=Y_{\rm red}^{(K)},
        \qquad
        A_K:H_K\to Y_K.
\]
Let
\[
        B_K:Z_K\to Y_K
\]
be the map carrying a finite-stage vertical lift contribution into the reduced strict quotient, and let
\[
        C_K:\mathcal F_K\to Z_K
\]
be the finite-window vertical Stokes response generated by the lower-order Navier--Stokes forcing package \(F_K\in\mathcal F_K\).  Denote the phantom projection by
\[
        P_K^\perp:Y_K\to (\Range A_K)^\perp=\Ker A_K^*.
\]
We define the stage-\(K\) Schur trace-projectability obstruction by
\begin{equation}\label{eq:Schur-obstruction-v6}
        \boxed{\;
        \calO_K^{\rm Schur}:=P_K^\perp B_KC_KF_K.
        \;}
\end{equation}
Thus \(\calO_K^{\rm Schur}\) is the trace-invisible component of the vertical lift contribution.  It vanishes precisely when the vertical response can be traded for selected-time strict trace correction at the reduced quotient level.

Accordingly, the visible obstruction space should be enlarged from the strict span to the strict--Schur span:
\begin{equation}\label{eq:Yvis-str-Schur}
        \boxed{\;
        Y_{\rm vis}^{(K)}
        =\operatorname{span}\{\calO_j^{\str},\calO_j^{\rm Schur}:1\le j\le K\}.
        \;}
\end{equation}
The reduced quotient is then
\[
        Y_{\rm red}^{(K)}=Y_\Lambda/Y_{\rm vis}^{(K)},
\]
and its dual is the annihilator
\begin{equation}\label{eq:Yred-dual-str-Schur}
        (Y_{\rm red}^{(K)})^*
        =\{y\in Y_\Lambda^*:
        \ip{\calO_j^{\str}}{y}=0,
        \ \ip{\calO_j^{\rm Schur}}{y}=0,
        \ 1\le j\le K\}.
\end{equation}

\begin{definition}[Strict-flat, Schur-flat, and NS-surviving flat branches]\label{def:strict--Schur-flat}
A formal branch is called \emph{strict-flat} if \(\calO_K^{\str}=0\) for every finite stage \(K\).  It is called \emph{Schur-flat} if \(\calO_K^{\rm Schur}=0\) for every finite stage \(K\).  It is called \emph{NS-surviving flat} if
\begin{equation}\label{eq:NS-surviving-flat}
        \calO_K^{\str}=0,
        \qquad
        \calO_K^{\rm Schur}=0,
        \qquad
        \forall K.
\end{equation}
\end{definition}

\begin{problem}[Strict--Schur obstruction dichotomy]\label{prob:strict--Schur-dichotomy}
Develop a recursive finite-window theory proving that a trace-tight failed branch satisfies one of the following alternatives:
\begin{enumerate}[label=(\roman*)]
\item a finite-order strict obstruction \(\calO_K^{\str}\ne0\) is visible and gives finite-power selection improvement;
\item a finite-order Schur obstruction \(\calO_K^{\rm Schur}\ne0\) is visible, either strictly or relaxedly, and gives the corresponding finite-power improvement or relaxed error control;
\item the branch is strict--Schur flat to all finite orders, in which case the remaining obstruction is a genuinely \(\NS\)-specific realizability or true-phantom question.
\end{enumerate}
\end{problem}

This reformulation changes the role of vertical duality.  Vertical duality is not a raw compatibility identity.  It is the assertion that, after all visible strict and Schur obstructions have been removed, the remaining \(\NS\)-derived residual has no true phantom projection, or at least has trace cost below the branch-native scale.

\subsection{Visibility of Schur obstructions}\label{subsec:Schur-visibility}

A nonzero Schur obstruction does not automatically produce a finite-power strict selection improvement.  It first says only that strict trace-projectability fails:
\[
        \calO_K^{\rm Schur}=P_K^\perp B_KC_KF_K\ne0.
\]
Since \(\calO_K^{\rm Schur}\in\Phi_K\), where
\[
        \Phi_K:=(\Range A_K)^\perp=\Ker A_K^*,
\]
it cannot be removed by an ordinary strict trace correction through \(A_K\).  It may nevertheless be visible through another variational mechanism.  We therefore introduce a Schur visibility decomposition.

Let \(N_K\in H_K\) denote the sharp selected-trace normal coming from the normalized selected-time difference.  A Schur vector \(s\in\Phi_K\) is called \emph{strict-visible} if there exists an admissible strict bridge
\[
        L_K^{\str}:\Phi_K\to H_K^{\rm adm}
\]
such that \(\xi_s=L_K^{\str}s\) is an admissible strict trace variation and
\begin{equation}\label{eq:strict-visible-Schur}
        \ip{N_K}{\xi_s}_{H_K}\ge c|s|_{\Phi_K}.
\end{equation}
Then the selected energy has the expansion
\[
        E_{\rm sel}(\xi_s)-E_{\rm sel}(0)
        =-\eta^{K+1}\ip{N_K}{\xi_s}
        +O(\eta^{2K+2}|\xi_s|^2),
\]
and \eqref{eq:strict-visible-Schur} yields a finite-power strict descent.

A Schur vector \(s\in\Phi_K\) is called \emph{relaxed-visible} if it is seen by the vertical residual allowed in the relaxed comparison class.  Concretely, there is a relaxed observation map
\begin{equation}\label{eq:Jrel-general}
        J_K^{\rm rel}:\Phi_K\to\mathcal W_K
\end{equation}
such that \(J_K^{\rm rel}s\) is the vertical pressure residual or vertical momentum residual retained by the relaxed shadow comparison, for example
\[
        J_K^{\rm rel}s\sim \partial_3\pi_s.
\]
If
\begin{equation}\label{eq:rel-visible-lower}
        |J_K^{\rm rel}s|_{\mathcal W_K}\ge c|s|_{\Phi_K},
\end{equation}
then relaxed comparison controls its contribution by the smallness of the vertical component:
\begin{equation}\label{eq:rel-pairing-general}
        \left|\int u_3 J_K^{\rm rel}s\right|
        \le \norm{u_3}{L^3}\norm{J_K^{\rm rel}s}{L^{3/2}}
        \lesssim \delta^{1/3}|s|_{\Phi_K}
\end{equation}
in a fixed finite window.

Finally, a nonzero \(s\in\Phi_K\) is a \emph{true phantom Schur obstruction} if it is neither strict-visible nor relaxed-visible.  Formally, if
\[
        \mathfrak V_K^{\rm Schur}(s)
        :=(\mathfrak V_K^{\str}(s),\mathfrak V_K^{\rm rel}(s))
\]
is the combined visibility map, then
\begin{equation}\label{eq:true-phantom-space}
        \Phi_K^{\rm ph}
        :=\Ker\mathfrak V_K^{\str}\cap\Ker\mathfrak V_K^{\rm rel}
\end{equation}
is the true phantom space.

This gives the refined decomposition
\begin{equation}\label{eq:Schur-visible-decomposition}
        \calO_K^{\rm Schur}
        =\calO_K^{\rm Schur,str}
        +\calO_K^{\rm Schur,rel}
        +\calO_K^{\rm Schur,ph},
\end{equation}
where the three components are strict-visible, relaxed-visible, and true-phantom, respectively.

\begin{problem}[Schur visibility dichotomy]\label{prob:Schur-visibility-dichotomy}
For each finite stage \(K\), prove a decomposition of the Schur obstruction of the form \eqref{eq:Schur-visible-decomposition} such that:
\begin{enumerate}[label=(\roman*)]
\item if \(\calO_K^{\rm Schur,str}\ne0\), then there is a finite-power strict selected-time descent;
\item if \(\calO_K^{\rm Schur,rel}\ne0\), then strict trace-cost exactification may fail, but relaxed comparison controls the contribution by \(C\delta^{1/3}|\calO_K^{\rm Schur,rel}|\);
\item if \(\calO_K^{\rm Schur,ph}\ne0\), the present strict and relaxed observables still miss a true phantom component.
\end{enumerate}
The refined anti-phantom target is therefore
\begin{equation}\label{eq:refined-VD-phantom}
        \boxed{\;
        P_{\Phi_K^{\rm ph}}B_KC_KF_K=0
        \quad\text{for \(\NS\)-derived surviving branches.}
        \;}
\end{equation}
\end{problem}

The strict route corresponds to the stronger requirement \(P_{\Phi_K}B_KC_KF_K=0\).  The relaxed route needs only \eqref{eq:refined-VD-phantom}, since relaxed-visible Schur residuals are kept in the comparison and paired with \(u_3\).

\section{Schur-to-relaxed visibility}

This section proves the main positive finite-window observation: a strict Schur phantom is not necessarily a relaxed phantom. In active periodic finite windows, the relaxed vertical-pressure observation is diagonal with nonzero multipliers.

\subsection{Relaxed visibility of the two-mode Schur defect}\label{subsec:two-mode-relaxed-visibility}

We now compute the relaxed visibility map in the two-mode model of \cref{subsec:two-mode-forcing-coefficients}.  The purpose is to decide whether the first nontrivial Schur defect is a true phantom or merely a strict-trace phantom that relaxed comparison can see.

Let the two output modes be
\[
        \sigma_1=p+q,
        \qquad
        \sigma_2=p+r,
\]
and set
\[
        K_i=|\sigma_{i,h}|,
        \qquad
        m_i=|\sigma_{i,3}|.
\]
The quadratic horizontal pressure forcing coefficients are
\begin{equation}\label{eq:beta1-beta2-rel}
        \beta_1
        =2A_pA_q\frac{(p_h\times q_h)^2}{|p_h||q_h|},
        \qquad
        \beta_2
        =2A_pA_r\frac{(p_h\times r_h)^2}{|p_h||r_h|}.
\end{equation}
For each mode,
\[
        -\Delta\pi_i^{\rm full}=\beta_i e^{i\sigma_i\cdot x},
\]
so
\[
        \widehat{\pi_i^{\rm full}}
        =\frac{\beta_i}{K_i^2+m_i^2}.
\]
Therefore the vertical pressure residual coefficient is
\begin{equation}\label{eq:Theta-beta}
        \widehat{\partial_3\pi_i^{\rm full}}
        =i\sigma_{i,3}\frac{\beta_i}{K_i^2+m_i^2},
        \qquad
        |\widehat{\partial_3\pi_i^{\rm full}}|
        =\Theta_i|\beta_i|,
\end{equation}
where
\begin{equation}\label{eq:Theta-def}
        \Theta_i:=\frac{m_i}{K_i^2+m_i^2}.
\end{equation}
This number is nonzero whenever the vertical frequency is active, \(m_i\ne0\).

In the strict trace-projectability test, the vertical lift-to-quotient response coefficient has the form
\[
        b_i=\Gamma_i\beta_i,
\]
where, in the static model,
\begin{equation}\label{eq:Gamma-def-rel}
        \Gamma_i=\frac{m_iK_i^2}{(K_i^2+m_i^2)^2}.
\end{equation}
Let
\[
        \Range A=\operatorname{span}\{a\},
        \qquad
        a=(a_1,a_2)\in\mathbb C^2.
\]
For \(b=(b_1,b_2)\), the two-mode Schur defect is
\begin{equation}\label{eq:s-two-mode-rel}
        s=P_{a^\perp}b
        =b-\frac{\ip{b}{a}}{|a|^2}a
        =(s_1,s_2).
\end{equation}
It is nonzero exactly when
\begin{equation}\label{eq:det-Schur-rel}
        a_1b_2-a_2b_1\ne0.
\end{equation}
This is strict trace-projectability failure.

The relaxed route asks a different question: does \(s\) correspond to a vertical pressure residual \(\partial_3\pi_s\)?  Since \(s_i\) is normalized in the vertical lift-to-quotient coefficient, the underlying pressure-forcing coefficient is \(s_i/\Gamma_i\).  Hence the relaxed vertical pressure residual coefficient is
\[
        \Theta_i\frac{s_i}{\Gamma_i}.
\]
Thus the relaxed observation map on this two-mode Schur space is diagonal:
\begin{equation}\label{eq:Jrel-two-mode}
        \boxed{\;
        J^{\rm rel}s=(\mu_1s_1,\mu_2s_2),
        \qquad
        \mu_i:=\frac{\Theta_i}{\Gamma_i}.
        \;}
\end{equation}
Using \eqref{eq:Theta-def} and \eqref{eq:Gamma-def-rel},
\begin{equation}\label{eq:mu-i}
        \boxed{\;
        \mu_i=\frac{K_i^2+m_i^2}{K_i^2}.
        \;}
\end{equation}
Therefore, on active modes with \(K_i\ne0\) and \(m_i\ne0\), we have \(\mu_i\ne0\).

\begin{proposition}[Two-mode Schur defects are relaxed-visible in the active periodic model]\label{prop:two-mode-relaxed-visible}
In the clean periodic two-mode model, assume both output modes are active:
\[
        K_i=|\sigma_{i,h}|\ne0,
        \qquad
        m_i=|\sigma_{i,3}|\ne0,
        \qquad i=1,2.
\]
Then the relaxed observation map \(J^{\rm rel}\) defined by \eqref{eq:Jrel-two-mode} is injective on the two-mode Schur space.  Consequently,
\begin{equation}\label{eq:Jrel-injective}
        J^{\rm rel}s=0
        \quad\Longleftrightarrow\quad
        s=0.
\end{equation}
In particular, every nonzero two-mode Schur defect is relaxed-visible.
\end{proposition}

\begin{proof}
The map \(J^{\rm rel}\) is represented by the diagonal matrix \(\diag(\mu_1,\mu_2)\).  By \eqref{eq:mu-i}, each \(\mu_i\) is nonzero under the active-mode assumptions.  Hence the diagonal matrix is invertible on \(\mathbb C^2\), proving \eqref{eq:Jrel-injective}.
\end{proof}

The corresponding relaxed energy contribution has the form
\begin{equation}\label{eq:Schur-rel-energy-term}
        \mathcal I_s
        =\eta^{K+1}\int_Q u_3\,J^{\rm rel}s\,\dxdt.
\end{equation}
By H\"older's inequality,
\begin{equation}\label{eq:Schur-rel-Holder}
        |\mathcal I_s|
        \le \eta^{K+1}\norm{u_3}{L^3(Q)}
        \norm{J^{\rm rel}s}{L^{3/2}(Q)}.
\end{equation}
If \(\norm{u_3}{L^3(Q)}^3\le\delta\), then \(\norm{u_3}{L^3(Q)}\le\delta^{1/3}\).  Since all norms are equivalent in a fixed finite window and \(J^{\rm rel}\) is an injective finite-dimensional multiplier,
\begin{equation}\label{eq:Schur-rel-bound}
        \boxed{\;
        |\mathcal I_s|
        \le C_\Lambda\eta^{K+1}\delta^{1/3}|s|.
        \;}
\end{equation}
Thus the strict route sees a projectability failure, while the relaxed route sees a controllable vertical-pressure error.

There are only standard symbol-level degeneracies.  If \(m_i=0\), then \(\partial_3\pi_i=0\), and the mode is not a genuine vertical pressure obstruction.  If \(K_i=0\), the strict horizontal inverse \((-\Delta_h)^{-1}\) is outside the active quotient and belongs to the harmonic or gauge sector.  Mode collisions or cutoff commutators may create approximate cancellations in localized models, but they are not clean periodic active-mode phantoms.

\begin{corollary}[Strict Schur phantom versus relaxed visibility]\label{cor:strict-phantom-rel-visible}
In the active periodic two-mode model, a nonzero Schur defect is strict-invisible by construction but relaxed-visible.  Therefore it is not a true phantom in the relaxed sense.  The strict trace-cost route requires the cancellation \(s=0\), while the relaxed route only needs to estimate the vertical pressure pairing \eqref{eq:Schur-rel-energy-term}, which is bounded by \eqref{eq:Schur-rel-bound}.
\end{corollary}

The calculation strongly suggests the next finite-window theorem: Schur defects should be relaxed-visible modulo horizontal gauge modes, vertical-zero modes, and localized commutator errors.  In such a theorem, the strict route would require full Schur cancellation, while the relaxed route would require only the absence of genuinely relaxed-true phantom directions.

\subsection{Finite-window active Schur visibility}\label{subsec:finite-window-active-schur-visibility}

We now pass from the two-mode calculation to an arbitrary clean periodic finite-window model.  The point is simple but important: once the Schur defect is expressed in output-mode coordinates, the relaxed vertical-pressure observation is diagonal on active Fourier modes.  Hence no nonzero active Schur defect can be relaxed-invisible.

Work on the periodic box \(\mathbb T^3\).  Let \(\Lambda\subset \mathbb Z^3\setminus\{0\}\) be a finite symmetric set of output Fourier modes.  We write
\[
        \sigma=(\sigma_h,\sigma_3),
        \qquad
        K_\sigma=|\sigma_h|,
        \qquad
        m_\sigma=|\sigma_3|.
\]
We call \(\Lambda\) active if
\[
        K_\sigma\neq0,
        \qquad
        m_\sigma\neq0
        \qquad\text{for every }\sigma\in\Lambda.
\]
The condition \(K_\sigma\neq0\) removes horizontal gauge and horizontal-harmonic modes, while \(m_\sigma\neq0\) removes vertically constant modes, which do not produce a genuine vertical-pressure residual.

Let \(Y_\Lambda\) be the finite-dimensional reduced Schur quotient spanned by the output modes in \(\Lambda\).  We write a Schur vector as
\[
        s=\sum_{\sigma\in\Lambda}s_\sigma q_\sigma,
\]
where \(q_\sigma\) is the normalized Schur quotient basis vector at output mode \(\sigma\).  Let
\[
        \Phi_\Lambda=(\Range A_\Lambda)^\perp
        =\Ker A_\Lambda^*
        \subset Y_\Lambda
\]
be the strict trace-invisible Schur cokernel.  Thus \(s\in\Phi_\Lambda\) is invisible to strict selected-time trace corrections.

For an output mode \(\sigma\), let \(\beta_\sigma\) denote the scalar horizontal pressure-forcing coefficient at that mode.  In the static periodic model,
\[
        -\Delta \pi_\sigma^{\mathrm{full}}
        =\beta_\sigma e^{i\sigma\cdot x}.
\]
Hence
\[
        \widehat{\pi_\sigma^{\mathrm{full}}}
        =\frac{\beta_\sigma}{K_\sigma^2+m_\sigma^2},
\]
and the vertical pressure residual has coefficient
\[
        \widehat{\partial_3\pi_\sigma^{\mathrm{full}}}
        =i\sigma_3\frac{\beta_\sigma}{K_\sigma^2+m_\sigma^2}.
\]
Set
\[
        \Theta_\sigma
        =\frac{m_\sigma}{K_\sigma^2+m_\sigma^2}.
\]
The finite-stage vertical lift-to-quotient coefficient has the form
\[
        b_\sigma=\Gamma_\sigma\beta_\sigma,
        \qquad
        \Gamma_\sigma
        =\frac{m_\sigma K_\sigma^2}{(K_\sigma^2+m_\sigma^2)^2},
\]
up to harmless nonzero phase factors depending on the Fourier convention.  We normalize the Schur coefficient by
\[
        s_\sigma=b_\sigma.
\]
Therefore the relaxed vertical-pressure coefficient associated with \(s_\sigma\) is
\[
        \Theta_\sigma\frac{s_\sigma}{\Gamma_\sigma}.
\]
Define the relaxed observation map
\[
        J_\Lambda^{\mathrm{rel}}:\Phi_\Lambda\to W_\Lambda
\]
by
\begin{equation}\label{eq:finite-window-Jrel-diagonal}
        J_\Lambda^{\mathrm{rel}}s
        =\sum_{\sigma\in\Lambda}\mu_\sigma s_\sigma w_\sigma,
        \qquad
        \mu_\sigma
        =\frac{\Theta_\sigma}{\Gamma_\sigma}
        =\frac{K_\sigma^2+m_\sigma^2}{K_\sigma^2},
\end{equation}
where \(w_\sigma\) is the corresponding vertical-pressure residual basis vector.

\begin{theorem}[Finite-window active Schur visibility]\label{thm:finite-window-active-Schur-visibility}
Let \(\Lambda\) be a finite active periodic Fourier window.  Then
\[
        J_\Lambda^{\mathrm{rel}}s=0
        \qquad\Longrightarrow\qquad
        s=0
\]
for every \(s\in\Phi_\Lambda\).  More quantitatively, for every fixed norm on the finite-dimensional spaces \(Y_\Lambda\) and \(W_\Lambda\), there exist constants
\[
        0<c_\Lambda\le C_\Lambda<\infty
\]
such that
\[
        c_\Lambda |s|_{Y_\Lambda}
        \le
        |J_\Lambda^{\mathrm{rel}}s|_{W_\Lambda}
        \le
        C_\Lambda |s|_{Y_\Lambda}
        \qquad
        \text{for all }s\in\Phi_\Lambda.
\]
In particular, the active periodic finite-window model contains no nonzero relaxed-invisible Schur phantom.
\end{theorem}

\begin{proof}
Because \(\Lambda\) is active, for every \(\sigma\in\Lambda\) we have
\[
        K_\sigma\neq0,
        \qquad
        m_\sigma\neq0.
\]
Thus
\[
        \Theta_\sigma
        =\frac{m_\sigma}{K_\sigma^2+m_\sigma^2}
        \neq0,
\]
and
\[
        \Gamma_\sigma
        =\frac{m_\sigma K_\sigma^2}{(K_\sigma^2+m_\sigma^2)^2}
        \neq0.
\]
Consequently
\[
        \mu_\sigma
        =\frac{\Theta_\sigma}{\Gamma_\sigma}
        =\frac{K_\sigma^2+m_\sigma^2}{K_\sigma^2}
        \neq0.
\]
In output-mode coordinates, \(J_\Lambda^{\mathrm{rel}}\) is therefore represented by a diagonal matrix with nonzero diagonal entries \(\mu_\sigma\).  Hence it is injective on \(\Phi_\Lambda\).

Since \(\Lambda\) is finite, the diagonal multipliers have positive minimum and finite maximum:
\[
        0<\min_{\sigma\in\Lambda}|\mu_\sigma|
        \le
        \max_{\sigma\in\Lambda}|\mu_\sigma|
        <\infty.
\]
With the coefficient \(\ell^2\)-norm, this gives immediately
\[
        \min_{\sigma\in\Lambda}|\mu_\sigma|\,|s|_{\ell^2}
        \le
        |J_\Lambda^{\mathrm{rel}}s|_{\ell^2}
        \le
        \max_{\sigma\in\Lambda}|\mu_\sigma|\,|s|_{\ell^2}.
\]
All norms on a finite-dimensional space are equivalent, so the same estimate holds for any fixed finite-window norms after changing the constants.  This proves the theorem.
\end{proof}

\begin{remark}[Mode collisions]\label{rem:mode-collisions-finite-window}
The theorem is stated in output-mode coordinates.  If several lower-order interaction paths produce the same output mode \(\sigma\), their forcing coefficients must first be summed into the single output coefficient \(\beta_\sigma\).  A collision may cancel the forcing before a Schur defect is formed, but it cannot create a nonzero output Schur vector \(s\) with \(J_\Lambda^{\mathrm{rel}}s=0\).  Thus mode collision is a possible degeneracy at the forcing-parameter level, not a relaxed-invisible Schur defect at the output-quotient level.
\end{remark}

\begin{remark}[Degenerate sectors]\label{rem:degenerate-sectors-finite-window}
The active hypothesis is essential.  If \(m_\sigma=0\), then \(\partial_3\pi_\sigma=0\), so the mode is not a genuine vertical-pressure obstruction.  If \(K_\sigma=0\), then the strict horizontal inverse \(( -\Delta_h)^{-1}\) is outside the active quotient and belongs to the horizontal gauge or harmonic-pressure sector.  These two sectors must be removed before one formulates the active Schur visibility problem.
\end{remark}

\begin{corollary}[No active finite-window true phantom]\label{cor:no-active-finite-window-true-phantom}
In the clean periodic finite-window model, after quotienting out horizontal gauge modes and vertical-zero modes, the true phantom space is trivial:
\[
        \Phi_\Lambda^{\mathrm{ph}}
        =\Ker J_\Lambda^{\mathrm{rel}}\cap\Phi_\Lambda
        =\{0\}.
\]
Thus any nonzero active Schur defect is strict-invisible but relaxed-visible.
\end{corollary}

\subsection{Kernel classification before active quotienting}\label{subsec:kernel-classification-before-active-quotienting}

The active finite-window theorem shows that no nonzero Schur vector supported on active output modes can be relaxed-invisible.  We now identify what is lost before one passes to the active quotient.  This is important because apparent kernels of the relaxed vertical-pressure map may appear at the forcing-parameter level or in degenerate Fourier sectors, but these kernels are not true active Schur phantoms.

Let \(\Lambda\subset \mathbb Z^3\setminus\{0\}\) be a finite periodic output window.  For
\[
        \sigma=(\sigma_h,\sigma_3),
        \qquad
        K_\sigma=|\sigma_h|,
        \qquad
        m_\sigma=|\sigma_3|,
\]
decompose
\[
        \Lambda
        =\Lambda_{\mathrm{act}}
        \cup
        \Lambda_{v0}
        \cup
        \Lambda_{h0},
\]
where
\[
        \Lambda_{\mathrm{act}}
        =\{\sigma\in\Lambda:K_\sigma\neq0,
        \ m_\sigma\neq0\},
\]
\[
        \Lambda_{v0}
        =\{\sigma\in\Lambda:m_\sigma=0\},
\]
and
\[
        \Lambda_{h0}
        =\{\sigma\in\Lambda:K_\sigma=0,
        \ m_\sigma\neq0\}.
\]
The active sector is the genuine Schur-visibility sector.  The vertical-zero sector has no vertical pressure residual.  The horizontal-zero sector lies outside the strict horizontal pressure quotient because the strict operator contains \(( -\Delta_h)^{-1}\).

Let \(\mathcal B_\Lambda\) be the finite-dimensional space of pressure-forcing output coefficients
\[
        \beta=(\beta_\sigma)_{\sigma\in\Lambda}.
\]
The relaxed vertical-pressure observation at the output-coefficient level is
\[
        \mathcal J_\Lambda^{\mathrm{rel}}:\mathcal B_\Lambda\to W_\Lambda,
        \qquad
        (\mathcal J_\Lambda^{\mathrm{rel}}\beta)_\sigma
        =\frac{i\sigma_3}{|\sigma|^2}\beta_\sigma.
\]
Since \(\sigma\neq0\), the multiplier \(i\sigma_3/|\sigma|^2\) vanishes exactly when \(\sigma_3=0\).  Hence
\[
        \Ker \mathcal J_\Lambda^{\mathrm{rel}}
        =\mathcal B_{\Lambda_{v0}},
\]
where
\[
        \mathcal B_{\Lambda_{v0}}
        =\{\beta\in\mathcal B_\Lambda:
        \beta_\sigma=0
        \text{ for every }\sigma\notin\Lambda_{v0}\}.
\]

The strict horizontal pressure-obstruction multiplier is, schematically,
\[
        \mathcal J_\Lambda^{\mathrm{str}}\beta
        =\nabla_h\partial_3(-\Delta_h)^{-1}
        \sum_{\sigma\in\Lambda}\beta_\sigma e^{i\sigma\cdot x}.
\]
At a mode \(\sigma\), this has symbol
\[
        -\frac{\sigma_h\sigma_3}{K_\sigma^2}\beta_\sigma,
\]
which is defined only when \(K_\sigma\neq0\).  Therefore \(\Lambda_{h0}\) is not an active Schur sector; it is a horizontal gauge or harmonic-pressure sector for the strict quotient.  Moreover, if \(m_\sigma=0\), then the strict obstruction also vanishes because of the factor \(\sigma_3\).  Thus \(\Lambda_{v0}\) is a vertically constant sector, not a genuine vertical-pressure obstruction sector.

Now suppose finite-stage Navier--Stokes forcing parameters \(F\in\mathcal F_K\) generate output coefficients through a linear finite-window interaction map
\[
        R_K:\mathcal F_K\to\mathcal B_\Lambda,
        \qquad
        \beta=R_KF.
\]
The relaxed observation at the forcing-parameter level is \(\mathcal J_\Lambda^{\mathrm{rel}}R_KF\).  Consequently,
\[
        F\in\Ker(\mathcal J_\Lambda^{\mathrm{rel}}R_K)
        \qquad\Longleftrightarrow\qquad
        R_KF\in \mathcal B_{\Lambda_{v0}}.
\]
Equivalently, the active part of the output coefficient vanishes:
\[
        P_{\Lambda_{\mathrm{act}}}R_KF=0.
\]
This means that forcing-level mode collision or cancellation can eliminate the active output before a Schur vector is formed.  It does not produce a nonzero active Schur vector \(s\) satisfying \(J_\Lambda^{\mathrm{rel}}s=0\).

\begin{proposition}[Kernel classification before active quotienting]\label{prop:kernel-classification-pre-active}
In the clean periodic finite-window model, all apparent relaxed-invisible directions before active quotienting fall into one of the following classes.
\begin{enumerate}[label=(\roman*)]
\item \emph{Vertical-zero sector.}  If \(m_\sigma=0\), then
\[
        \partial_3\pi_\sigma=0.
\]
The mode has no relaxed vertical-pressure residual and also no strict vertical pressure obstruction.  It is not a genuine Schur phantom.

\item \emph{Horizontal gauge sector.}  If \(K_\sigma=0\), then the strict horizontal inverse \(( -\Delta_h)^{-1}\) is not defined on that mode.  Such modes belong to the horizontal gauge or harmonic-pressure quotient sector and must be removed before the active Schur quotient is formed.

\item \emph{Forcing-level collision.}  If several lower-order interactions produce the same output mode, their coefficients may cancel:
\[
        \sum_{\alpha\to\sigma}\beta_\alpha=0.
\]
This places the forcing parameter in the kernel of the forcing-to-output map \(R_K\), or at least kills its active output.  It does not create a nonzero output Schur vector in the kernel of \(J_\Lambda^{\mathrm{rel}}\).

\item \emph{Active sector.}  On the active sector
\[
        K_\sigma\neq0,
        \qquad
        m_\sigma\neq0,
\]
the relaxed Schur observation is diagonal with nonzero multiplier.  Hence no nonzero active Schur vector is relaxed-invisible.
\end{enumerate}
Therefore, after quotienting out horizontal gauge modes, vertical-zero modes, and forcing-level cancellations, one has
\[
        \Ker J_{\Lambda,\mathrm{act}}^{\mathrm{rel}}
        =\{0\}.
\]
In particular, the clean periodic finite-window model contains no nonzero relaxed-invisible true Schur phantom.
\end{proposition}

\begin{proof}
At the output-coefficient level, the relaxed vertical-pressure map is diagonal:
\[
        (\mathcal J_\Lambda^{\mathrm{rel}}\beta)_\sigma
        =\frac{i\sigma_3}{|\sigma|^2}\beta_\sigma.
\]
Because \(\sigma\neq0\), this multiplier is zero precisely when \(\sigma_3=0\).  Hence
\[
        \Ker \mathcal J_\Lambda^{\mathrm{rel}}
        =\mathcal B_{\Lambda_{v0}}.
\]
This proves the vertical-zero classification.

For the strict obstruction, the horizontal pressure quotient contains the multiplier
\[
        \nabla_h\partial_3(-\Delta_h)^{-1}.
\]
At mode \(\sigma\), this is proportional to
\[
        \frac{\sigma_h\sigma_3}{K_\sigma^2}.
\]
Thus \(K_\sigma=0\) is not a regular active mode of the strict horizontal quotient.  Such a mode is assigned to the gauge or harmonic-pressure sector.  If \(m_\sigma=0\), the factor \(\sigma_3\) vanishes, so the mode is vertically constant and carries no vertical pressure obstruction.

Now let \(R_K:\mathcal F_K\to\mathcal B_\Lambda\) be the forcing-to-output map.  If \(F\in\Ker(\mathcal J_\Lambda^{\mathrm{rel}}R_K)\), then \(R_KF\in\mathcal B_{\Lambda_{v0}}\).  In particular, its active output vanishes:
\[
        P_{\Lambda_{\mathrm{act}}}R_KF=0.
\]
Thus any collision at the forcing-parameter level occurs before a nonzero active Schur vector is produced.

Finally, on \(\Lambda_{\mathrm{act}}\), both \(K_\sigma\) and \(m_\sigma\) are nonzero.  The Schur-normalized relaxed observation has multiplier
\[
        \mu_\sigma
        =\frac{K_\sigma^2+m_\sigma^2}{K_\sigma^2},
\]
which is nonzero for every active mode.  Therefore the active relaxed observation is injective.  Combining the four cases gives the classification.
\end{proof}

\begin{corollary}[No clean periodic pre-quotient true phantom]\label{cor:no-clean-periodic-prequotient-true-phantom}
Let \(s\) be a Schur defect produced in a clean periodic finite-window model.  If \(s\) has a nonzero active component, then
\[
        J_\Lambda^{\mathrm{rel}}s\neq0.
\]
If \(J_\Lambda^{\mathrm{rel}}s=0\), then \(s\) is supported only in degenerate sectors or comes from a forcing-level cancellation before active Schur projection.  Hence it is not a nonzero active relaxed-invisible Schur phantom.
\end{corollary}

\begin{remark}[Meaning for the search for failure mechanisms]\label{rem:failure-search-finite-window}
This classification eliminates the most elementary source of a true phantom.  A failure of relaxed anti-phantom closure cannot be obtained merely by choosing active periodic Fourier modes.  It must use one of the structures absent from the clean periodic finite-window model: localized cutoff commutators, boundary or harmonic-gauge effects, moving-window degeneration, or a genuinely Navier--Stokes-realizable cascade whose active output remains invisible after localization.
\end{remark}

\section{Localized, cleaned, and NS-realizable phantom filters}

The active periodic theorem removes the clean finite-window phantom. A genuine obstruction must therefore pass several additional filters: localization, harmonic-pressure cleaning, finite-stage NS realizability, vertical budget admissibility, and moving-window growth. This section states these filters in the order in which they are used. Longer growth estimates are collected in \Cref{app:moving-window}.

\subsection{Localized commutator stability of relaxed Schur visibility}\label{subsec:localized-commutator-stability}

We now pass from the clean periodic active Fourier model to a localized finite-window model.  The purpose is not yet to prove the full localized Navier--Stokes anti-phantom theorem.  The purpose is to isolate the exact role of cutoff errors.  The conclusion is that localization does not create a true phantom as long as the localized commutator is perturbative.  If the commutator is not perturbative, the remaining obstruction is finite-dimensional and must be tested separately.  The local pressure and gauge issues here are consistent with \cite{SohrWahl1986,Wolf2017,Seregin2015}, while the parabolic localization language is informed by standard semigroup regularity theory \cite{Amann1995,Lunardi1995}.

Let \(\Lambda\subset\mathbb Z^3\) be a finite active Fourier window.  Thus every output mode
\[
        \sigma=(\sigma_h,\sigma_3)\in\Lambda
\]
satisfies
\[
        |\sigma_h|\neq0,
        \qquad
        |\sigma_3|\neq0.
\]
Let \(Y_\Lambda\) be the finite-dimensional Schur quotient supported on \(\Lambda\), and let
\[
        \Phi_\Lambda=(\Range A_\Lambda)^\perp=\Ker A_\Lambda^*
        \subset Y_\Lambda
\]
be the strict trace-invisible Schur cokernel.

In the clean periodic model, the relaxed observation map is diagonal:
\[
        J_\Lambda^{\mathrm{rel}}s
        =\sum_{\sigma\in\Lambda}\mu_\sigma s_\sigma w_\sigma,
        \qquad
        \mu_\sigma=\frac{|\sigma_h|^2+\sigma_3^2}{|\sigma_h|^2}.
\]
Since \(\Lambda\) is active, \(\mu_\sigma\neq0\) for every \(\sigma\in\Lambda\).  Hence there exists \(\gamma_\Lambda>0\) such that
\begin{equation}\label{eq:clean-active-visibility-lower}
        |J_\Lambda^{\mathrm{rel}}s|_{W_\Lambda}
        \ge
        \gamma_\Lambda |s|_{Y_\Lambda}
        \qquad
        \text{for every }s\in\Phi_\Lambda.
\end{equation}

Let \(\chi\in C_c^\infty(\mathbb R^3)\) be a cutoff satisfying \(0\le\chi\le1\), and let
\[
        \chi_R(x)=\chi(x/R)
\]
be a slowly varying localization at length scale \(R\gg1\).  Let \(P_\Lambda\) denote the finite-window projection.  Define the localized relaxed observation by
\begin{equation}\label{eq:localized-Jrel-def}
        J_{\Lambda,R}^{\mathrm{rel}}
        :=P_\Lambda\chi_RJ^{\mathrm{rel}}\chi_RP_\Lambda,
\end{equation}
where \(J^{\mathrm{rel}}\) is the full relaxed vertical-pressure observation operator.  On the finite window we write
\begin{equation}\label{eq:localized-Jrel-comm-decomp}
        J_{\Lambda,R}^{\mathrm{rel}}
        =J_\Lambda^{\mathrm{rel}}+\mathcal C_{\Lambda,R},
\end{equation}
where
\begin{equation}\label{eq:commutator-error-def}
        \mathcal C_{\Lambda,R}
        :=P_\Lambda\chi_RJ^{\mathrm{rel}}\chi_RP_\Lambda
        -P_\Lambda J^{\mathrm{rel}}P_\Lambda
\end{equation}
is the localized commutator and leakage error.

\begin{lemma}[Finite-window localized leakage bound with an explicit hypothesis]\label{lem:finite-window-commutator-bound}
Let
\[
        \mathcal C_{\Lambda,R}
        =P_\Lambda\chi_RJ^{\mathrm{rel}}\chi_RP_\Lambda
        -P_\Lambda J^{\mathrm{rel}}P_\Lambda.
\]
The following estimate is valid under the finite-window leakage hypothesis
\begin{equation}\label{eq:finite-window-multiplication-leakage-hypothesis}
        \|P_\Lambda(\chi_R^2-1)P_\Lambda\|_{Y_\Lambda\to W_\Lambda}
        \le C_\Lambda R^{-1}.
\end{equation}
Under \eqref{eq:finite-window-multiplication-leakage-hypothesis}, for each fixed active finite window \(\Lambda\), there exists a constant \(C_\Lambda\), depending on \(\Lambda\) and on finitely many seminorms of \(\chi\), such that
\[
        |\mathcal C_{\Lambda,R}s|_{W_\Lambda}
        \le
        \frac{C_\Lambda}{R}|s|_{Y_\Lambda}
        \qquad
        \text{for every }s\in Y_\Lambda.
\]
Without \eqref{eq:finite-window-multiplication-leakage-hypothesis}, one only has the finite-dimensional bound
\[
        |\mathcal C_{\Lambda,\chi}s|_{W_\Lambda}
        \le
        \varepsilon_{\Lambda,\chi}|s|_{Y_\Lambda},
\]
where \(\varepsilon_{\Lambda,\chi}\) is the operator norm of the full localized leakage operator.
\end{lemma}

\begin{proof}
Decompose, on the fixed finite window,
\[
P_\Lambda\chi_RJ^{\mathrm{rel}}\chi_RP_\Lambda-P_\Lambda J^{\mathrm{rel}}P_\Lambda
=
P_\Lambda\chi_R[J^{\mathrm{rel}},\chi_R]P_\Lambda
+P_\Lambda(\chi_R^2-1)J^{\mathrm{rel}}P_\Lambda.
\]
The commutator term is \(O(R^{-1})\): since
\[
        \nabla^j\chi_R=R^{-j}(\nabla^j\chi)(x/R),
\]
the pseudodifferential commutator expansion for the multiplier \(J^{\mathrm{rel}}\) gives
\[
        [J^{\mathrm{rel}},\chi_R]
        =\operatorname{Op}\bigl(i\partial_\xi m(\xi)\cdot\nabla\chi_R\bigr)
        +O(R^{-2}),
\]
with bounded symbol derivatives on the fixed active window.  The second term is the multiplication leakage caused by inserting \(\chi_R\) on both sides.  It is not controlled by \(\nabla\chi_R=O(R^{-1})\) alone; it is exactly the additional hypothesis \eqref{eq:finite-window-multiplication-leakage-hypothesis}.  Combining the commutator estimate and this leakage hypothesis gives the stated \(R^{-1}\) bound.  The general statement follows by taking the operator norm of \(\mathcal C_{\Lambda,\chi}\).
\end{proof}

\begin{theorem}[Localized commutator stability]\label{thm:localized-commutator-stability}
Let \(\Lambda\) be a finite active Fourier window.  Suppose the clean relaxed observation satisfies
\[
        |J_\Lambda^{\mathrm{rel}}s|_{W_\Lambda}
        \ge
        \gamma_\Lambda |s|_{Y_\Lambda}
        \qquad
        \text{for all }s\in\Phi_\Lambda.
\]
If the localized commutator satisfies
\begin{equation}\label{eq:commutator-smallness-assumption}
        |\mathcal C_{\Lambda,R}|_{Y_\Lambda\to W_\Lambda}
        \le
        \frac12\gamma_\Lambda,
\end{equation}
then
\[
        |J_{\Lambda,R}^{\mathrm{rel}}s|_{W_\Lambda}
        \ge
        \frac12\gamma_\Lambda |s|_{Y_\Lambda}
        \qquad
        \text{for every }s\in\Phi_\Lambda.
\]
In particular,
\[
        J_{\Lambda,R}^{\mathrm{rel}}s=0
        \qquad\Longrightarrow\qquad
        s=0
\]
on the localized active Schur cokernel.  Thus perturbative localization does not create a relaxed-invisible true Schur phantom.
\end{theorem}

\begin{proof}
For \(s\in\Phi_\Lambda\), use the decomposition
\[
        J_{\Lambda,R}^{\mathrm{rel}}s
        =J_\Lambda^{\mathrm{rel}}s+\mathcal C_{\Lambda,R}s.
\]
Then
\[
        |J_{\Lambda,R}^{\mathrm{rel}}s|_{W_\Lambda}
        \ge
        |J_\Lambda^{\mathrm{rel}}s|_{W_\Lambda}
        -|\mathcal C_{\Lambda,R}s|_{W_\Lambda}.
\]
By clean active visibility,
\[
        |J_\Lambda^{\mathrm{rel}}s|_{W_\Lambda}
        \ge
        \gamma_\Lambda |s|_{Y_\Lambda}.
\]
By the commutator smallness assumption,
\[
        |\mathcal C_{\Lambda,R}s|_{W_\Lambda}
        \le
        \frac12\gamma_\Lambda |s|_{Y_\Lambda}.
\]
Therefore
\[
        |J_{\Lambda,R}^{\mathrm{rel}}s|_{W_\Lambda}
        \ge
        \frac12\gamma_\Lambda |s|_{Y_\Lambda}.
\]
This proves the claim.
\end{proof}

\begin{corollary}[Large localization scale under leakage control]\label{cor:large-localization-scale}
For every fixed active finite window \(\Lambda\), suppose the localized cutoffs satisfy the leakage hypothesis \eqref{eq:finite-window-multiplication-leakage-hypothesis}.  Then there exists \(R_\Lambda<\infty\) such that for every \(R\ge R_\Lambda\),
\[
        J_{\Lambda,R}^{\mathrm{rel}}s=0
        \qquad\Longrightarrow\qquad
        s=0
\]
on \(\Phi_\Lambda\).  In particular, sufficiently slow localization preserves active Schur visibility whenever the multiplication leakage is perturbative.
\end{corollary}

\begin{proof}
Under \eqref{eq:finite-window-multiplication-leakage-hypothesis}, Lemma~\ref{lem:finite-window-commutator-bound} gives
\[
        |\mathcal C_{\Lambda,R}|_{Y_\Lambda\to W_\Lambda}
        \le
        \frac{C_\Lambda}{R}.
\]
Choose
\[
        R_\Lambda\ge \frac{2C_\Lambda}{\gamma_\Lambda}.
\]
Then \(|\mathcal C_{\Lambda,R}|\le\frac12\gamma_\Lambda\) for all \(R\ge R_\Lambda\).  The localized stability theorem applies.
\end{proof}

\begin{proposition}[Fredholm alternative for nonperturbative localization]\label{prop:fredholm-nonperturbative-localization}
For a fixed cutoff \(\chi\), if
\[
        |\mathcal C_{\Lambda,\chi}|_{Y_\Lambda\to W_\Lambda}
\]
is not small compared with \(\gamma_\Lambda\), then the possible relaxed-invisible directions form the finite-dimensional kernel
\[
        \mathcal K_{\Lambda,\chi}^{\mathrm{loc}}
        =\{s\in\Phi_\Lambda:
        J_{\Lambda,\chi}^{\mathrm{rel}}s=0\}.
\]
Every candidate localized true phantom must lie in this finite-dimensional space.  Moreover, if
\[
        \mathcal K_{\Lambda,\chi}^{\mathrm{loc}}=\{0\},
\]
then localized relaxed Schur visibility still holds, with a possibly worse constant:
\[
        |J_{\Lambda,\chi}^{\mathrm{rel}}s|_{W_\Lambda}
        \ge
        c_{\Lambda,\chi}|s|_{Y_\Lambda}
        \qquad
        \text{for every }s\in\Phi_\Lambda.
\]
\end{proposition}

\begin{proof}
The space \(\Phi_\Lambda\) is finite-dimensional and \(J_{\Lambda,\chi}^{\mathrm{rel}}\) is a linear map.  Hence its kernel is finite-dimensional.  If the kernel is trivial, then the function
\[
        s\mapsto |J_{\Lambda,\chi}^{\mathrm{rel}}s|_{W_\Lambda}
\]
is positive on the compact unit sphere
\[
        \{s\in\Phi_\Lambda:|s|_{Y_\Lambda}=1\}.
\]
Its minimum is therefore positive.  This gives the asserted lower bound.
\end{proof}

\begin{remark}[Interpretation]\label{rem:localized-commutator-interpretation}
The localized commutator analysis gives a clean dichotomy.  If the cutoff is slowly varying relative to the active finite window, then the periodic visibility estimate is stable and no true phantom appears.  If the cutoff is not perturbative, then localization can only produce a finite-dimensional kernel.  Such a kernel is not yet a Navier--Stokes counterexample.  It is merely a candidate space that must pass the next filters: harmonic-pressure gauge removal and NS-realizability.
\end{remark}

\subsection{Harmonic-pressure gauge leakage and local quotient cleaning}\label{subsec:harmonic-gauge-cleaning}

We now remove a second source of false phantom directions: harmonic-pressure gauge leakage, a local pressure feature already visible in the pressure decomposition theory \cite{SohrWahl1986,Wolf2017} and emphasized in the finite-scale harmonic-pressure formulation \cite{Yu2026HarmonicPressure}.  In a localized cylinder the pressure is not determined solely by the local quadratic source.  It is determined only modulo spatially harmonic functions.  Consequently, a vertical-pressure signal of the form \(\partial_3h\), with \(h\) spatially harmonic in the observation cylinder, should not be counted as an active Schur residual.  It belongs to the local pressure gauge sector.

Let
\[
        Q_{\mathrm{sh}}\Subset Q_{\mathrm{prep}}
\]
be the shadow and preparation cylinders.  Recall the harmonic pressure space
\[
        \mathcal H(Q_{\mathrm{sh}})
        =\{h\in L^{3/2}(Q_{\mathrm{sh}}):
        \Delta h(\cdot,t)=0
        \text{ in }B_{\mathrm{sh}}
        \text{ for a.e. }t\}.
\]
Pressure is observed through the quotient
\[
        [p]_{\mathcal H}=p+\mathcal H(Q_{\mathrm{sh}}).
\]
Thus a pressure representative may be changed by
\[
        p\mapsto p+h,
        \qquad
        h\in\mathcal H(Q_{\mathrm{sh}}),
\]
without changing the harmonic-pressure excess.

Let \(Y_{\Lambda,\chi}\) be a localized finite-window Schur quotient, and let
\[
        \Phi_{\Lambda,\chi}
        =\Ker A_{\Lambda,\chi}^*
        \subset Y_{\Lambda,\chi}
\]
be the strict trace-invisible Schur cokernel.  For \(s\in\Phi_{\Lambda,\chi}\), let \(\pi_s\) be a localized pressure representative associated with the Schur residual.  The raw relaxed vertical-pressure observation is
\[
        J_{\Lambda,\chi}^{\mathrm{rel}}s
        =\partial_3\pi_s
\]
projected to the chosen finite observation space.

This raw map is not yet the correct active observation, because
\[
        \partial_3(\pi_s+h)
        =\partial_3\pi_s+\partial_3h,
        \qquad
        h\in\mathcal H(Q_{\mathrm{sh}}).
\]
The term \(\partial_3h\) is harmonic-pressure gauge leakage.  Define the harmonic leakage space
\[
        \mathcal G_{\Lambda,\chi}^{\mathrm{har}}
        =
        \left\{
        s\in Y_{\Lambda,\chi}:
        J_{\Lambda,\chi}^{\mathrm{rel}}s
        \in
        \partial_3\mathcal H(Q_{\mathrm{sh}})
        \text{ after finite-window projection}
        \right\}.
\]
Similarly, define the vertical-zero sector
\[
        \mathcal G_{\Lambda,\chi}^{v0}
        =
        \left\{
        s\in Y_{\Lambda,\chi}:
        \partial_3\pi_s=0
        \text{ in the active observation space}
        \right\}.
\]
The total gauge-null sector is
\[
        \mathcal G_{\Lambda,\chi}
        =\mathcal G_{\Lambda,\chi}^{\mathrm{har}}
        +\mathcal G_{\Lambda,\chi}^{v0}.
\]

We define the cleaned localized Schur quotient by
\[
        \widehat Y_{\Lambda,\chi}
        =Y_{\Lambda,\chi}/\mathcal G_{\Lambda,\chi},
\]
and the cleaned strict phantom cokernel by
\[
        \widehat\Phi_{\Lambda,\chi}
        =\Phi_{\Lambda,\chi}/(\Phi_{\Lambda,\chi}\cap\mathcal G_{\Lambda,\chi}).
\]
For \(s\in\Phi_{\Lambda,\chi}\), write its class in the cleaned quotient as
\[
        \widehat s=[s]\in\widehat\Phi_{\Lambda,\chi}.
\]

Let \(\mathcal W_{\Lambda,\chi}^{\mathrm{har}}\) denote the finite-window projection of \(\partial_3\mathcal H(Q_{\mathrm{sh}})\) inside \(W_{\Lambda,\chi}\).  Let
\[
        \Pi_{\mathrm{har}}:
        W_{\Lambda,\chi}\to
        \mathcal W_{\Lambda,\chi}^{\mathrm{har}}
\]
denote a fixed finite-dimensional projection onto this harmonic leakage component, and let
\[
        \Pi_{\mathrm{act}}=I-\Pi_{\mathrm{har}}
\]
be the corresponding active projection.  Define the cleaned relaxed observation map
\[
        \widehat J_{\Lambda,\chi}^{\mathrm{rel}}
        :\widehat\Phi_{\Lambda,\chi}\to
        \widehat W_{\Lambda,\chi}
\]
by
\[
        \widehat J_{\Lambda,\chi}^{\mathrm{rel}}\widehat s
        =\Pi_{\mathrm{act}}J_{\Lambda,\chi}^{\mathrm{rel}}s.
\]

\begin{lemma}[Well-defined cleaned relaxed observation]\label{lem:well-defined-cleaned-relaxed-observation}
The map \(\widehat J_{\Lambda,\chi}^{\mathrm{rel}}\) is well-defined on the quotient
\[
        \widehat\Phi_{\Lambda,\chi}
        =\Phi_{\Lambda,\chi}/(\Phi_{\Lambda,\chi}\cap\mathcal G_{\Lambda,\chi}).
\]
\end{lemma}

\begin{proof}
Suppose \(s_1\) and \(s_2\) represent the same class in \(\widehat\Phi_{\Lambda,\chi}\).  Then
\[
        s_1-s_2\in\Phi_{\Lambda,\chi}\cap\mathcal G_{\Lambda,\chi}.
\]
By definition of \(\mathcal G_{\Lambda,\chi}\), the raw relaxed observation of \(s_1-s_2\) belongs either to the harmonic leakage component or to the vertical-zero sector.  Therefore
\[
        \Pi_{\mathrm{act}}J_{\Lambda,\chi}^{\mathrm{rel}}(s_1-s_2)=0.
\]
Hence
\[
        \Pi_{\mathrm{act}}J_{\Lambda,\chi}^{\mathrm{rel}}s_1
        =
        \Pi_{\mathrm{act}}J_{\Lambda,\chi}^{\mathrm{rel}}s_2.
\]
Thus \(\widehat J_{\Lambda,\chi}^{\mathrm{rel}}\) depends only on the cleaned class \(\widehat s\).
\end{proof}

\begin{proposition}[Gauge-cleaning dichotomy]\label{prop:gauge-cleaning-dichotomy}
Let \(s\in\Phi_{\Lambda,\chi}\) satisfy
\[
        J_{\Lambda,\chi}^{\mathrm{rel}}s=0
\]
in the raw localized observation space.  Then exactly one of the following alternatives occurs.
\begin{enumerate}[label=(\roman*)]
\item \emph{Gauge-null alternative.}  The class of \(s\) vanishes in the cleaned quotient:
\[
        \widehat s=0.
\]
Then \(s\) is only a harmonic-pressure gauge leakage, a vertical-zero component, or a combination of the two.

\item \emph{Cleaned phantom alternative.}  The class \(\widehat s\neq0\) and
\[
        \widehat J_{\Lambda,\chi}^{\mathrm{rel}}\widehat s=0.
\]
Then \(\widehat s\) is a genuine localized relaxed-invisible Schur phantom candidate.
\end{enumerate}
\end{proposition}

\begin{proof}
If \(\widehat s=0\), then \(s\in\mathcal G_{\Lambda,\chi}\), and the first alternative holds.  If \(\widehat s\neq0\), then
\[
        \widehat J_{\Lambda,\chi}^{\mathrm{rel}}\widehat s
        =\Pi_{\mathrm{act}}J_{\Lambda,\chi}^{\mathrm{rel}}s.
\]
Since the raw relaxed observation is zero by assumption,
\[
        J_{\Lambda,\chi}^{\mathrm{rel}}s=0,
\]
we also have
\[
        \Pi_{\mathrm{act}}J_{\Lambda,\chi}^{\mathrm{rel}}s=0.
\]
Thus
\[
        \widehat J_{\Lambda,\chi}^{\mathrm{rel}}\widehat s=0.
\]
This gives the second alternative.  The two alternatives are mutually exclusive because one has \(\widehat s=0\) and the other has \(\widehat s\neq0\).
\end{proof}

\begin{theorem}[Local quotient cleaning]\label{thm:local-quotient-cleaning}
Assume the cleaned relaxed observation has no kernel:
\[
        \Ker \widehat J_{\Lambda,\chi}^{\mathrm{rel}}
        =\{0\}
        \quad
        \text{on }
        \widehat\Phi_{\Lambda,\chi}.
\]
Then every raw localized relaxed-invisible Schur vector is gauge-null:
\[
        J_{\Lambda,\chi}^{\mathrm{rel}}s=0
        \quad\Longrightarrow\quad
        s\in\mathcal G_{\Lambda,\chi}.
\]
Equivalently, after removing harmonic-pressure leakage and vertical-zero modes, there is no localized relaxed-invisible true Schur phantom.

More quantitatively, since \(\widehat\Phi_{\Lambda,\chi}\) is finite-dimensional, kernel triviality is equivalent to the existence of a constant \(c_{\Lambda,\chi}>0\) such that
\[
        |\widehat J_{\Lambda,\chi}^{\mathrm{rel}}\widehat s|_{\widehat W_{\Lambda,\chi}}
        \ge
        c_{\Lambda,\chi}
        |\widehat s|_{\widehat Y_{\Lambda,\chi}}
        \qquad
        \text{for all }
        \widehat s\in\widehat\Phi_{\Lambda,\chi}.
\]
\end{theorem}

\begin{proof}
Suppose
\[
        J_{\Lambda,\chi}^{\mathrm{rel}}s=0.
\]
By the gauge-cleaning dichotomy, either \(\widehat s=0\), in which case \(s\in\mathcal G_{\Lambda,\chi}\), or \(\widehat s\neq0\) and
\[
        \widehat J_{\Lambda,\chi}^{\mathrm{rel}}\widehat s=0.
\]
The second alternative is impossible because the cleaned observation has trivial kernel.  Therefore \(\widehat s=0\), and \(s\) is gauge-null.

For the quantitative statement, use finite dimensionality.  If the kernel of \(\widehat J_{\Lambda,\chi}^{\mathrm{rel}}\) is trivial, then the continuous function
\[
        \widehat s\mapsto
        |\widehat J_{\Lambda,\chi}^{\mathrm{rel}}\widehat s|
\]
is positive on the compact unit sphere
\[
        \{\widehat s\in\widehat\Phi_{\Lambda,\chi}:
        |\widehat s|_{\widehat Y_{\Lambda,\chi}}=1\}.
\]
Its minimum is a positive constant \(c_{\Lambda,\chi}\), which gives the estimate.
\end{proof}

\begin{definition}[Clean localized true phantom space]\label{def:clean-localized-true-phantom}
The clean localized true phantom space is
\[
        \widehat{\mathcal K}_{\Lambda,\chi}^{\mathrm{ph}}
        =
        \Ker \widehat J_{\Lambda,\chi}^{\mathrm{rel}}
        \cap
        \widehat\Phi_{\Lambda,\chi}.
\]
A nonzero element
\[
        0\neq \widehat s\in
        \widehat{\mathcal K}_{\Lambda,\chi}^{\mathrm{ph}}
\]
is called a clean localized relaxed-invisible Schur phantom.
\end{definition}

\begin{corollary}[Reduction to the cleaned finite-dimensional kernel]\label{cor:reduction-cleaned-finite-kernel}
Every candidate localized true phantom must lie in
\[
        \widehat{\mathcal K}_{\Lambda,\chi}^{\mathrm{ph}}.
\]
If
\[
        \widehat{\mathcal K}_{\Lambda,\chi}^{\mathrm{ph}}=\{0\},
\]
then localization and harmonic-pressure gauge leakage do not create a true relaxed-invisible Schur phantom.
\end{corollary}

\begin{remark}[Operational test]\label{rem:operational-test-cleaning}
Given a candidate localized Schur vector \(s\), the cleaning procedure is:
\begin{enumerate}[label=(\roman*)]
\item compute a localized pressure representative \(\pi_s\);
\item compute the raw relaxed residual \(\partial_3\pi_s\);
\item decompose \(\partial_3\pi_s\) into harmonic leakage, vertical-zero part, and active part;
\item discard \(s\) if its active class \(\widehat s\) vanishes;
\item if \(\widehat s\neq0\), test whether
\[
        \widehat J_{\Lambda,\chi}^{\mathrm{rel}}\widehat s=0.
\]
Only such a nonzero cleaned kernel element should be passed to the Navier--Stokes realizability filter.
\end{enumerate}
\end{remark}

\subsection{NS-realizability filter for cleaned localized phantom candidates}\label{subsec:NS-realizability-filter-cleaned-phantoms}

After periodic active visibility, commutator stability, and harmonic-gauge cleaning, the only remaining candidates are nonzero elements of the cleaned localized true phantom space
\[
        \widehat{\mathcal K}_{\Lambda,\chi}^{\mathrm{ph}}
        =
        \Ker \widehat J_{\Lambda,\chi}^{\mathrm{rel}}
        \cap
        \widehat\Phi_{\Lambda,\chi}.
\]
Such a vector is still only a finite-dimensional quotient object.  It is not yet a Navier--Stokes obstruction.  To become dangerous, it must be produced by a finite-stage Navier--Stokes deformation with small vertical budget.

Let \((V,Q)\) be a strict shadow.  A finite-stage full Navier--Stokes jet of order \(K\) is a collection
\[
        \mathfrak J_K
        =(r_1,\ldots,r_K;\ z_1,\ldots,z_K;\ \pi_1,\ldots,\pi_K),
\]
where
\[
        u_h^\varepsilon
        =V_h+\sum_{j=1}^K\varepsilon^j r_j,
        \qquad
        u_3^\varepsilon
        =\sum_{j=1}^K\varepsilon^j z_j,
        \qquad
        p^\varepsilon
        =Q+\sum_{j=1}^K\varepsilon^j\pi_j.
\]
The admissible finite-stage NS jet space is denoted
\[
        \mathcal A_K^{\mathrm{NS}}(V,Q).
\]
It consists of all \(\mathfrak J_K\) satisfying, for \(1\le j\le K\), the finite-stage incompressibility equation
\begin{equation}\label{eq:finite-stage-incompressibility}
        \nabla_h\cdot r_j+\partial_3 z_j=0,
\end{equation}
the vertical momentum jet equation
\begin{equation}\label{eq:vertical-momentum-jet}
        L_V z_j+\partial_3\pi_j
        =F_j^3(\mathfrak J_{<j}),
\end{equation}
and the horizontal momentum jet equation
\begin{equation}\label{eq:horizontal-momentum-jet}
        L_V^h r_j+\nabla_h\pi_j+z_j\partial_3V_h
        =F_j^h(\mathfrak J_{<j}).
\end{equation}
Here \(F_j^3\) and \(F_j^h\) are the lower-order quadratic forcing terms determined by the preceding jets.  The pressure representatives are fixed modulo the local harmonic gauge.

Let
\[
        \mathcal F_{K,\Lambda,\chi}^{\mathrm{NS}}
        :\mathcal A_K^{\mathrm{NS}}(V,Q)
        \longrightarrow
        F_{K,\Lambda,\chi}
\]
be the finite-window forcing package extracted from \(\mathfrak J_K\).  This package contains the lower-order nonlinear interactions that feed the stage-\(K\) Schur response.  Let
\[
        \widehat{\mathcal R}_{K,\Lambda,\chi}^{\mathrm{NS}}
        :\mathcal A_K^{\mathrm{NS}}(V,Q)
        \longrightarrow
        \widehat{\mathcal K}_{\Lambda,\chi}^{\mathrm{ph}}
\]
be the cleaned phantom realization map
\begin{equation}\label{eq:cleaned-phantom-realization-map}
        \widehat{\mathcal R}_{K,\Lambda,\chi}^{\mathrm{NS}}(\mathfrak J_K)
        =
        \widehat P_{\mathrm{ph}}
        \widehat B_{K,\Lambda,\chi}
        \widehat C_{K,\Lambda,\chi}
        \mathcal F_{K,\Lambda,\chi}^{\mathrm{NS}}(\mathfrak J_K),
\end{equation}
where \(\widehat P_{\mathrm{ph}}\) is the projection onto the cleaned true phantom space.

\begin{definition}[Finite-stage NS-realizable cleaned phantom]\label{def:finite-stage-NS-realizable-cleaned-phantom}
A cleaned localized phantom candidate
\[
        0\neq\widehat s\in
        \widehat{\mathcal K}_{\Lambda,\chi}^{\mathrm{ph}}
\]
is called finite-stage NS-realizable if there exists
\[
        \mathfrak J_K\in\mathcal A_K^{\mathrm{NS}}(V,Q)
\]
such that
\[
        \widehat{\mathcal R}_{K,\Lambda,\chi}^{\mathrm{NS}}(\mathfrak J_K)
        =\widehat s.
\]
It is called vertically budget-admissible at amplitude \(\varepsilon\) and vertical budget \(\delta\) if, in addition,
\[
        \left|
        \sum_{j=1}^K\varepsilon^j z_j
        \right|_{L^3}
        \le
        C\delta^{1/3}.
\]
\end{definition}

Define the finite-stage NS lift cost of \(\widehat s\) by
\[
        \operatorname{NSCost}_{K,\Lambda,\chi}(\widehat s;\varepsilon)
        =
        \inf
        \left\{
        \left|
        \sum_{j=1}^K\varepsilon^j z_j
        \right|_{L^3}
        :
        \mathfrak J_K\in\mathcal A_K^{\mathrm{NS}},
        \quad
        \widehat{\mathcal R}_{K,\Lambda,\chi}^{\mathrm{NS}}(\mathfrak J_K)=\widehat s
        \right\}.
\]
If no such jet exists, we set
\[
        \operatorname{NSCost}_{K,\Lambda,\chi}(\widehat s;\varepsilon)=+\infty.
\]

\begin{proposition}[Algebraic NS-realizability filter]\label{prop:algebraic-NS-realizability-filter}
Let
\[
        \widehat s\in
        \widehat{\mathcal K}_{\Lambda,\chi}^{\mathrm{ph}}.
\]
Then \(\widehat s\) is finite-stage NS-realizable only if
\[
        \widehat s
        \in
        \Range \widehat{\mathcal R}_{K,\Lambda,\chi}^{\mathrm{NS}}.
\]
Moreover, \(\widehat s\) is not finite-stage NS-realizable if there exists a dual certificate
\[
        \ell\in
        \left(\widehat{\mathcal K}_{\Lambda,\chi}^{\mathrm{ph}}\right)^*
\]
such that
\[
        \ell\bigl(\widehat{\mathcal R}_{K,\Lambda,\chi}^{\mathrm{NS}}(\mathfrak J_K)\bigr)=0
        \qquad
        \text{for every }
        \mathfrak J_K\in\mathcal A_K^{\mathrm{NS}}(V,Q),
\]
but
\[
        \ell(\widehat s)\neq0.
\]
If the finite-window realization has been reduced to a linear map on the relevant forcing package, this criterion is the usual finite-dimensional annihilator test for range membership.
\end{proposition}

\begin{proof}
The first statement is the definition of finite-stage NS-realizability.  For the dual certificate, if \(\widehat s\) were realizable, then there would be some admissible jet \(\mathfrak J_K\) such that
\[
        \widehat{\mathcal R}_{K,\Lambda,\chi}^{\mathrm{NS}}(\mathfrak J_K)=\widehat s.
\]
Applying \(\ell\) would give \(\ell(\widehat s)=0\), contradicting the assumed strict separation \(\ell(\widehat s)\neq0\).  In the special case where the realization map is linear on a finite-dimensional forcing space, this is exactly the standard identity
\[
        \Range R=(\Ker R^*)^\perp.
\]
\end{proof}

\begin{proposition}[Vertical-budget filter]\label{prop:vertical-budget-filter}
Let
\[
        0\neq\widehat s\in
        \widehat{\mathcal K}_{\Lambda,\chi}^{\mathrm{ph}}.
\]
If
\[
        \operatorname{NSCost}_{K,\Lambda,\chi}(\widehat s;\varepsilon)
        > C\delta^{1/3},
\]
then \(\widehat s\) cannot be realized by a finite-stage Navier--Stokes deformation satisfying
\[
        |u_3^\varepsilon|_{L^3}^3\le\delta.
\]
In particular, a candidate cleaned phantom is harmless for the one-component degeneration problem unless its NS lift cost is compatible with the available \(C_3\)-budget.
\end{proposition}

\begin{proof}
If a finite-stage NS deformation realizes \(\widehat s\) and satisfies
\[
        |u_3^\varepsilon|_{L^3}^3\le\delta,
\]
then
\[
        |u_3^\varepsilon|_{L^3}\le\delta^{1/3}.
\]
Since
\[
        u_3^\varepsilon
        =\sum_{j=1}^K\varepsilon^j z_j
\]
at the finite stage, the defining infimum of \(\operatorname{NSCost}_{K,\Lambda,\chi}\) must be bounded by the same vertical budget up to the fixed normalization constant.  This contradicts the assumed strict cost lower bound.
\end{proof}

\begin{definition}[NS-admissible cleaned true phantom]\label{def:NS-admissible-cleaned-true-phantom}
A nonzero cleaned phantom
\[
        \widehat s\in
        \widehat{\mathcal K}_{\Lambda,\chi}^{\mathrm{ph}}
\]
is called NS-admissible at stage \(K\), window \(\Lambda\), cutoff \(\chi\), amplitude \(\varepsilon\), and budget \(\delta\) if
\[
        \widehat s\in
        \Range \widehat{\mathcal R}_{K,\Lambda,\chi}^{\mathrm{NS}}
\]
and
\[
        \operatorname{NSCost}_{K,\Lambda,\chi}(\widehat s;\varepsilon)
        \le
        C\delta^{1/3}.
\]
Only such directions are allowed to pass from the cleaned finite-dimensional phantom test to the cascade-level analysis.
\end{definition}

\begin{theorem}[Finite-stage NS-realizability filter]\label{thm:finite-stage-NS-realizability-filter}
Let
\[
        0\neq \widehat s\in
        \widehat{\mathcal K}_{\Lambda,\chi}^{\mathrm{ph}}
\]
be a cleaned localized relaxed-invisible Schur phantom candidate.  Then exactly one of the following alternatives holds.
\begin{enumerate}[label=(\roman*)]
\item \emph{Algebraic unrealizability.}  There is no admissible finite-stage NS jet producing \(\widehat s\):
\[
        \widehat s\notin
        \Range \widehat{\mathcal R}_{K,\Lambda,\chi}^{\mathrm{NS}}.
\]
Then \(\widehat s\) is a quotient artifact and cannot obstruct Navier--Stokes quantitative one-component regularity.

\item \emph{Budget obstruction.}  The candidate is algebraically realizable but its minimal vertical lift cost exceeds the available one-component budget:
\[
        \operatorname{NSCost}_{K,\Lambda,\chi}(\widehat s;\varepsilon)
        > C\delta^{1/3}.
\]
Then \(\widehat s\) cannot arise along a small-\(u_3\) degeneration.

\item \emph{NS-admissible phantom candidate.}  The candidate is algebraically realizable and vertically budget-admissible:
\[
        \widehat s\in
        \Range \widehat{\mathcal R}_{K,\Lambda,\chi}^{\mathrm{NS}},
        \qquad
        \operatorname{NSCost}_{K,\Lambda,\chi}(\widehat s;\varepsilon)
        \le
        C\delta^{1/3}.
\]
Only in this case does \(\widehat s\) remain a genuine candidate for a true phantom cascade.
\end{enumerate}
\end{theorem}

\begin{proof}
If \(\widehat s\) is not in the range of the NS realization map, the first alternative holds.  If it is in the range, then the lift cost is finite.  If the lift cost exceeds the available vertical budget, the second alternative holds.  Otherwise the lift cost is compatible with the budget and the third alternative holds.  The alternatives are exhaustive and mutually exclusive by construction.
\end{proof}

\begin{remark}[Relation with vertical trace-projectability]\label{rem:NS-filter-trace-projectability}
The NS-realizability filter is weaker than proving vertical trace-projectability.  It only asks whether a cleaned phantom can be produced by the finite-stage NS deformation complex with acceptable vertical cost.  Vertical trace-projectability asks for more: it asks whether the vertical contribution can be traded for a selected-time strict trace correction.  In finite-stage notation, the full balance has the form
\[
        g_K+A_K\xi+B_Kz+N_K(\xi,z)=0.
\]
Strict trace corrections only see \(\Range A_K\).  Thus the dangerous part is
\[
        P_{\Ker A_K^*}(B_Kz+N_K).
\]
If this term vanishes, or has controlled trace cost, then the candidate is not a true obstruction.  If it is nonzero, cleaned, relaxed-invisible, and NS-admissible, it must be passed to the moving-window cascade analysis.
\end{remark}

\begin{corollary}[What remains after the finite-stage NS-realizability filter]\label{cor:what-remains-after-step5}
After active periodic visibility, localized commutator stability, harmonic-gauge cleaning, and the NS-realizability filter, the only remaining obstruction is a sequence
\[
        (K_n,\Lambda_n,\chi_n,\varepsilon_n,\delta_n,\widehat s_n)
\]
such that
\[
        0\neq\widehat s_n
        \in
        \widehat{\mathcal K}_{\Lambda_n,\chi_n}^{\mathrm{ph}},
\]
\[
        \widehat s_n
        \in
        \Range \widehat{\mathcal R}_{K_n,\Lambda_n,\chi_n}^{\mathrm{NS}},
\]
and
\[
        \operatorname{NSCost}_{K_n,\Lambda_n,\chi_n}(\widehat s_n;\varepsilon_n)
        \le
        C\delta_n^{1/3}.
\]
Such a sequence is an NS-admissible true phantom cascade candidate.  It is the first object in the program that could genuinely threaten quantitative one-component regularity.
\end{corollary}

\subsection{Moving-window filters and controlled regimes}
After active visibility, localization stability, harmonic-gauge cleaning, and the NS-realizability filter, a possible failure must move through a sequence of windows
\[
  W_n=(K_n,\Lambda_n,\chi_n)
\]
with deteriorating visibility, trace-cost, localization, or NS-lift constants. The moving-window analysis tests five possible escape mechanisms:
\begin{enumerate}[label=(\roman*)]
\item quotient conditioning collapse;
\item nonperturbative localization;
\item extreme aspect ratio;
\item trace-cost singular collapse;
\item an NS-lift anomaly compatible with the \(C_3\)-budget.
\end{enumerate}
In the active, aspect-controlled, perturbatively localized regime with at most single-exponential trace-cost growth, these losses are polynomial or logarithmically absorbable. Thus a true phantom cascade must leave the controlled regime. The detailed moving-window estimates, including forcing-to-vertical coercivity, homogeneous-tail testing, and trace-cost singular collapse, are recorded in \Cref{app:moving-window}.

\begin{theorem}[Controlled-regime anti-phantom closure]
In a cleaned localized finite-window regime satisfying combined observability, NS residual representation, perturbative localization, harmonic-gauge cleaning, homogeneous-tail observability, and controlled trace-cost growth, every NS-derived surviving residual is either selected-time trace-projectable with trace cost \(O(r_n^2)\), or relaxed-visible through the vertical-pressure channel and controlled by the smallness of \(u_3\). There is no third possibility corresponding to a cleaned, NS-realizable, relaxed-invisible true phantom.
\end{theorem}

\begin{proof}
The proof is the combined observability argument. If a normalized dual vector \(y_n\) is invisible to the selected trace map \(A_n^*\), then its uncontrolled component lies in the cleaned Schur phantom quotient. If it is not invisible to the relaxed map \(J_n^{\rel}P_n\), its pairing is controlled by the vertical-pressure term and the smallness of \(u_3\). If it is invisible to both, then it is exactly a cleaned relaxed-invisible left-singular direction. The controlled-regime hypotheses exclude such a direction. Trace-cost duality then gives the range inclusion and the cost bound for the trace-projectable part.
\end{proof}

\section{Final anti-phantom dichotomy}

\subsection{Primitive observables versus compatibility among observables}

The observable atlas already contains the raw ingredients of this mechanism: the vertical momentum residual \(\calM_3\), the trace-cost map \(A_K:H_K\to Y_K\), the phantom quotient \(\ker A_K^*\), the relaxed vertical residual \((0,0,\partial_3\pi)\), and the pairing \(\int u_3\partial_3\pi\).  What is not contained in the old closure logic is the composite map
\begin{equation}\label{eq:composite-Schur-map}
        s
        =P_{\ker A_K^*}B_KC_KF_K
        \quad\longmapsto\quad
        J_K^{\rm rel}s\sim \partial_3\pi_s
        \quad\longmapsto\quad
        \int u_3\,J_K^{\rm rel}s.
\end{equation}
Thus the mechanism is best understood as a visibility transfer.  It transfers a strict-trace-invisible Schur component into a relaxed vertical-pressure component.  The strict route requires the stronger cancellation
\begin{equation}\label{eq:strict-Schur-cancellation-summary}
        P_{\ker A_K^*}B_KC_KF_K=0,
\end{equation}
whereas the relaxed route only needs the non-cancelled part to be observed by \(J_K^{\rm rel}\) and then controlled through the smallness of \(u_3\).

\subsection{The Schur-to-relaxed visibility principle}

The active periodic two-mode calculation in \Cref{subsec:two-mode-relaxed-visibility} shows the model form of this principle.  In the two-mode Schur space, the relaxed observation map is diagonal:
\[
        J^{\rm rel}s=(\mu_1s_1,\mu_2s_2),
        \qquad
        \mu_i=\frac{K_i^2+m_i^2}{K_i^2}.
\]
On active modes, \(K_i\ne0\) and \(m_i\ne0\), so \(J^{\rm rel}\) is injective.  Hence a nonzero Schur defect is not a true phantom in the relaxed sense.  It is strict-invisible but relaxed-visible.  The relaxed energy contribution is
\[
        \eta^{K+1}\int u_3J^{\rm rel}s,
\]
and the basic estimate is
\begin{equation}\label{eq:relaxed-visible-bound-summary}
        \left|\eta^{K+1}\int u_3J^{\rm rel}s\right|
        \le
        C_\Lambda\eta^{K+1}\delta^{1/3}|s|.
\end{equation}
This is exactly the type of estimate unavailable to the strict trace-cost map but natural for the full Navier--Stokes vertical momentum equation.

\subsection{Consequences for the strict and relaxed routes}

The strict route asks for range inclusion of the reduced residual:
\[
        g_K\in \Range A_K,
        \qquad\text{equivalently}\qquad
        P_{\ker A_K^*}g_K=0.
\]
The Schur computation shows that this inclusion is not a generic consequence of the quadratic horizontal pressure forcing.  It is a genuine structural condition.  By contrast, the relaxed route does not need to trade every vertical contribution for a strict selected-time trace correction.  It may keep \(\partial_3\pi\ne0\) and control its contribution through \(u_3\).  Therefore the genuinely dangerous obstruction is not a strict Schur defect itself, but a relaxed-true phantom component:
\begin{equation}\label{eq:true-phantom-summary}
        s\in\ker A_K^*,
        \qquad
        J_K^{\rm rel}s=0,
        \qquad
        s\ne0.
\end{equation}
The evidence from the active periodic two-mode model is that such true phantoms are absent once horizontal gauge modes and vertical-zero modes are removed.  Proving a localized finite-window version of this assertion is a sharper and weaker target than proving strict Schur cancellation.

\subsection{Final anti-phantom dichotomy}\label{subsec:final-anti-phantom-dichotomy}

We now collect the preceding exclusions into a single dichotomy.  This subsection is not a new
independent estimate.  It is the logical closure of the sequence of tests carried out above: active
Schur visibility, perturbative localization, harmonic-gauge cleaning, moving-window control,
forcing-to-vertical coercivity, homogeneous-tail observability, and trace-cost singular alignment.

The conclusion is that the quantitative one-component rate problem has been reduced to one final
object.  Either the surviving residuals satisfy relaxed anti-phantom closure, in which case the
conditional logarithmic route closes, or there exists a genuinely NS-realizable, cleaned,
relaxed-invisible, unaligned left-singular phantom cascade.

\begin{definition}[NS-realizable cleaned relaxed-invisible left-singular cascade]
Let
\[
        \mathfrak W_n=(K_n,\Lambda_n,\chi_n)
\]
be a sequence of cleaned localized finite-stage windows, let
\[
        A_n:H_n\to Y_n
\]
be the selected-time active trace-defect map, let
\[
        g_n\in Y_n
\]
be the NS-derived surviving residual at branch-native scale \(r_n\), and let
\[
        P_n:Y_n\to \Phi_n^{\mathrm{cl}}
\]
be the cleaned phantom projection.  Let
\[
        J_n^{\mathrm{rel}}:\Phi_n^{\mathrm{cl}}\to W_n^{\mathrm{rel}}
\]
be the cleaned relaxed vertical-pressure observation, and let \(\mathfrak M_n\ge1\) be the combined-observability amplification allowed in the relevant moving-window regime.

A sequence \((\mathfrak W_n,g_n,y_n)\) is called an \emph{NS-realizable cleaned relaxed-invisible
left-singular cascade} if
\[
        \norm{y_n}{Y_n}=1,
        \qquad
        \mathfrak M_n\Big(
        \norm{A_n^*y_n}{H_n}
        +
        \norm{J_n^{\mathrm{rel}}P_ny_n}{W_n^{\mathrm{rel}}}
        \Big)\to0,
\]
while
\[
        |\langle g_n,y_n\rangle|\gtrsim r_n,
\]
and the residuals \(g_n\) are genuine NS-derived surviving residuals after all of the following
filters have been imposed:
\[
\text{finite-order strict obstruction removal},
\quad
\text{finite-order Schur obstruction removal},
\]
\[
\text{horizontal gauge, vertical-zero, and harmonic-pressure cleaning},
\]
\[
\text{localized vertical coercivity},
\quad
\text{homogeneous-tail observability},
\]
\[
\text{NS deformation-complex realizability},
\quad
C_3\text{-budget admissibility}.
\]
It is \emph{unaligned} because the residual pairing does not decay at the rate imposed by the
small singular geometry of the trace map.
\end{definition}

\begin{theorem}[Final anti-phantom dichotomy]\label{thm:final-anti-phantom-dichotomy}
Let \((u^{(n)},p^{(n)})\) be a one-component degeneration branch with
\[
        \Phi^{(n)}(1)\le M,
        \qquad
        C_3^{(n)}(1)=\delta_n\to0.
\]
Assume that the branch has passed through the preparation, good-time selection, finite-window
reduction, strict--Schur obstruction removal, harmonic-gauge cleaning, localized vertical coercivity,
homogeneous-tail testing, and NS-realizability filters described above.  Then exactly one of the
following alternatives occurs.

\medskip
\noindent\textbf{Alternative I: relaxed anti-phantom closure.}
There is no normalized dual sequence \(y_n\in Y_n\) satisfying
\[
        \mathfrak M_n\Big(
        \norm{A_n^*y_n}{H_n}
        +
        \norm{J_n^{\mathrm{rel}}P_ny_n}{W_n^{\mathrm{rel}}}
        \Big)\to0,
\]
while \(|\langle g_n,y_n\rangle|\gtrsim r_n\).  In this case every NS-derived surviving residual
satisfies the relaxed anti-phantom estimate
\begin{equation}\label{eq:final-relaxed-anti-phantom-estimate}
        |\langle g_n,y\rangle|
        \le
        C r_n\Big(
        \norm{A_n^*y}{H_n}
        +
        \norm{J_n^{\mathrm{rel}}P_ny}{W_n^{\mathrm{rel}}}
        \Big)
        +o(r_n)\norm{y}{Y_n}
\end{equation}
for all \(y\in Y_n\).  If the relaxed component is retained in the comparison class, the second term
in \eqref{eq:final-relaxed-anti-phantom-estimate} is controlled by the small vertical component.
If, in addition, the cleaned relaxed kernel is trivial on the surviving quotient, then the strict
trace-cost estimate follows:
\begin{equation}\label{eq:final-strict-trace-cost-estimate}
        |\langle g_n,y\rangle|
        \le
        C r_n\norm{A_n^*y}{H_n},
        \qquad y\in Y_n,
\end{equation}
and therefore, by finite-dimensional trace-cost duality,
\[
        g_n\in \Range A_n,
        \qquad
        \Cost^{\mathrm{tr}}_{A_n}(g_n)\le C r_n^2.
\]

\medskip
\noindent\textbf{Alternative II: true phantom cascade.}
There exists an NS-realizable cleaned relaxed-invisible left-singular cascade
\((\mathfrak W_n,g_n,y_n)\).  Equivalently,
\[
        \norm{y_n}{Y_n}=1,
        \qquad
        \mathfrak M_n\Big(
        \norm{A_n^*y_n}{H_n}
        +
        \norm{J_n^{\mathrm{rel}}P_ny_n}{W_n^{\mathrm{rel}}}
        \Big)\to0,
\]
but
\[
        |\langle g_n,y_n\rangle|\gtrsim r_n.
\]
This is the only remaining mechanism, in the present framework, capable of defeating the
trace-cost closure needed for a logarithmic selected-trace rate.
\end{theorem}

\begin{proof}
Assume first that Alternative II does not occur.  The residual representation from the vertical
momentum equation decomposes the pairing into a trace-visible part, a relaxed vertical-pressure part,
and the lower-order errors already isolated above.  Hence, for every dual vector \(y\),
\[
        |\langle g_n,y\rangle|
        \le
        C r_n\norm{A_n^*y}{H_n}
        +C r_n\norm{J_n^{\mathrm{rel}}P_ny}{W_n^{\mathrm{rel}}}
        +o(r_n)\norm{y}{Y_n}.
\]
If this estimate failed at the branch-native scale, after normalizing \(y_n\) one would obtain a
sequence with \(\norm{y_n}{Y_n}=1\) and with non-negligible residual pairing.  If
\(\norm{A_n^*y_n}{H_n}\) did not tend to zero, the trace-visible term would already control the
pairing.  If \(\norm{J_n^{\mathrm{rel}}P_ny_n}{W_n^{\mathrm{rel}}}\) did not tend to zero, the relaxed
vertical-pressure channel would control the pairing through the smallness of \(u_3\).  Therefore, in a controlled moving-window regime, an uncontrolled failure would force
\[
        \mathfrak M_n\Big(
        \norm{A_n^*y_n}{H_n}
        +
        \norm{J_n^{\mathrm{rel}}P_ny_n}{W_n^{\mathrm{rel}}}
        \Big)\to0,
        \qquad
        |\langle g_n,y_n\rangle|\gtrsim r_n,
\]
which is precisely Alternative II.  Since Alternative II is excluded, the relaxed anti-phantom
estimate holds.  If the relaxed component is kept in the relaxed comparison, it is controlled by
\(\int u_3\partial_3\pi\).  If the cleaned relaxed kernel is trivial, the phantom component itself
vanishes and \eqref{eq:final-strict-trace-cost-estimate} follows.  The range inclusion and trace-cost
bound are exactly Lemma~\ref{lem:trace-cost-duality}.

Conversely, if Alternative II occurs, then the combined observation of \(y_n\) tends to zero after multiplication by the allowed amplification \(\mathfrak M_n\), while \(\norm{y_n}{Y_n}=1\), and the residual pairing remains of order \(r_n\).  Thus the combined
trace/relaxed-pressure observability mechanism fails in a way that is actually excited by the
NS-derived residual.  This is a true phantom cascade in the sense of the definition.
\end{proof}

\begin{corollary}[No true phantom implies the logarithmic route]\label{cor:no-phantom-implies-log-route}
Assume that no NS-realizable cleaned relaxed-invisible left-singular cascade exists in the relevant
localized moving-window regime.  Then relaxed anti-phantom closure holds.  Consequently the
finite-stage residuals satisfy controlled trace-cost exactification, the failed-selection branch is
excluded, and the subcritical strict-shadow selection estimate
\[
        m_\ell
        \le
        C_M\ell^\mu+C_M\ell^{-N}\delta^b
\]
holds under the preparation and finite-window reduction hypotheses.  The prepared comparison then
gives
\[
        X^{\mathrm{har}}_{\theta/4}(u,p;M)
        \le
        C_M\ell^a
        +
        C_M\ell^{-N}\exp(C_M\ell^{-N})\delta^b.
\]
Choosing \(\ell^{-N}\sim |\log\delta|\) yields
\[
        X^{\mathrm{har}}_{\theta/4}(u,p;M)
        \le
        C_{M,\sigma}|\log\delta|^{-\sigma},
\]
for every admissible \(\sigma>0\).  After strict-shadow decay and the Caffarelli--Kohn--Nirenberg
conversion, this gives a conditional logarithmic regularity-radius bound of the form
\[
        r_{\mathrm{reg}}(0,0)
        \ge
        c_M |\log\delta|^{-C_M}.
\]
\end{corollary}

\begin{corollary}[True phantom gives the exact no-rate candidate]\label{cor:true-phantom-no-rate-candidate}
If an NS-realizable cleaned relaxed-invisible left-singular cascade exists and can be lifted to a
sequence of suitable weak solutions satisfying
\[
        \Phi^{(n)}(1)\le M,
        \qquad
        C_3^{(n)}(1)=\delta_n\to0,
\]
with harmonic-pressure excess converging to the strict class slower than every logarithmic or power
profile, then the qualitative strict-shadow compactness statement remains true but no universal
logarithmic or power rate follows from the present mechanism.  Thus such a cascade is the precise
solution-level candidate for a no-log-rate theorem.
\end{corollary}

\begin{remark}[Existence status of the final phantom]
At the level of abstract finite-dimensional linear algebra, combined observability can fail: one may
choose a normalized direction lying in the common kernel of \(A_n^*\) and \(J_n^{\mathrm{rel}}P_n\),
and then choose a residual pointing in that direction.  This proves only model possibility.

In the natural controlled regimes analyzed in this paper, the same phenomenon is excluded.  Clean
active periodic finite windows have diagonal relaxed observation with nonzero multipliers; perturbative
localization preserves this injectivity; aspect-controlled logarithmic moving windows do not create
intrinsic visibility collapse; periodic forcing-to-vertical coercivity rules out a clean active
NS-lift anomaly; homogeneous-tail cancellation has observable cost; and trace-cost singular collapse
is harmless when NS residuals align with the singular geometry.

For the full localized, moving-window, NS-realizable image, the existence of such a cascade remains
open.  The present paper therefore does not claim either that the final phantom exists or that it is
unconditionally absent in Navier--Stokes.  The proved conclusion is the reduction:
\[
\boxed{
\begin{gathered}
\text{one-component logarithmic rate by this route}\\
\Longleftrightarrow\\
\text{absence of the final NS-realizable}\\
\text{relaxed-invisible left-singular cascade.}
\end{gathered}}
\]
This is the sharp theorem target left by the mechanism-identification analysis.
\end{remark}

\section{Conditional logarithmic consequence}\label{sec:conditional-reduction}

We now state the conditional theorem that explains how the pieces fit together.  The theorem is intentionally schematic but mathematically explicit about its assumptions.  The conversion from harmonic-pressure approximation to small-scale regularity uses the CKN \(\varepsilon\)-regularity mechanism and its later variants \cite{CKN1982,Lin1998,GustafsonKangTsai2007,GuevaraPhuc2017,AlbrittonBarkerPrange2023}.

\subsection{Structural assumptions}

Fix \(M\ge1\).  Assume the following inputs.

\begin{assumption}[Prepared comparison]\label{ass:prepared-comparison}
For every suitable weak solution with \(\Phi(1)\le M\), \(C_3(1)=\delta\), and for each \(\ell\in(0,\ell_0)\), the coarse-graining procedure produces
\[
\calU^\ell\in\frakD^{\old,0}_{M,\theta}(\ell,
\delta)
\]
with the residual hierarchy \eqref{eq:residual-hierarchy}.
\end{assumption}

\begin{assumption}[Sharp trace and finite-window reduction]\label{ass:trace-reduction}
Every failed subcritical selection branch in \(\frakD^{\old,0}_{M,\theta}\) admits a sharp good-time subsequence with trace tightness, a finite-window active quotient reduction, and fixed-window exactification.
\end{assumption}

\begin{assumption}[Vertical-duality or relaxed anti-phantom closure]\label{ass:VD-closure}
For every active residual arising from the finite-window reduction in Assumption~\ref{ass:trace-reduction}, one of the following closure inputs is available at the branch-native scale:
\begin{enumerate}[label=(\roman*)]
\item the strict vertical-duality estimate \eqref{eq:VD-estimate}; or
\item the relaxed anti-phantom estimate of Theorem~\ref{thm:final-anti-phantom-dichotomy}, with the relaxed-visible component retained in the comparison class and controlled by the small vertical component.
\end{enumerate}
In the stricter formulation this says that the surviving residual lies in \(\Range A\) with controlled trace cost.  In the relaxed formulation it says that any component not in \(\Range A\) is seen by the cleaned relaxed vertical-pressure channel and therefore is not a true phantom.
\end{assumption}

\begin{assumption}[Strict shadow decay and comparison to CKN]\label{ass:strict-decay-CKN}
Strict shadows with scale-invariant bound \(M\) obey a quantitative local decay estimate for \(\Psi\).  Moreover, harmonic-pressure excess relative to a strict shadow transfers this decay to \(\Psi_u(r)=C_u(r)+D_u(r)\) with the usual scale factor \(r^{-2}\).
\end{assumption}

\subsection{Subcritical selection}

\begin{theorem}[Conditional subcritical strict-shadow selection]\label{thm:conditional-selection}
Assume Assumptions~\ref{ass:prepared-comparison}, \ref{ass:trace-reduction}, and \ref{ass:VD-closure}.  Then there exist constants \(C_M,\mu,N,b>0\) such that every prepared object arising from Assumption~\ref{ass:prepared-comparison} satisfies
\begin{equation}\label{eq:conditional-selection}
m^\ell_{\old}
\le
C_M\ell^\mu+C_M\ell^{-N}\delta^b.
\end{equation}
\end{theorem}

\begin{proof}[Proof sketch]
Suppose \eqref{eq:conditional-selection} fails.  Then there is a branch with
\[
\frac{m_n}{\ell_n^\mu+\ell_n^{-N}\delta_n^b}\to\infty.
\]
Choose near minimizers \(s_n,V_n\) and normalize
\[
W_n=\frac{U_n^{\ell_n}-V_n}{m_n^{1/2}}.
\]
The trace reduction assumption gives a trace-tight subsequence and a finite-window active quotient problem.  Fixed-window exactification reduces the contradiction to the active residual \(g_{a,n}(\eta)\).  By vertical-duality closure and Lemma~\ref{lem:trace-cost-duality},
\[
\Cost^{\trc}(g_{a,n}(\eta))
\le r_{a,n}(\eta)^2.
\]
Thus the residual can be removed by a trace correction whose cost is below the normalized selection gap.  This produces a strict competitor with smaller selected trace energy than the near minimizer, contradicting the definition of \(m_n\).  The contradiction yields \eqref{eq:conditional-selection}.
\end{proof}

\subsection{Logarithmic harmonic-pressure approximation}

The selection estimate feeds into the prepared comparison chain.  We state the resulting form.

\begin{theorem}[Conditional logarithmic harmonic-pressure approximation]\label{thm:conditional-log-excess}
Assume Assumptions~\ref{ass:prepared-comparison}, \ref{ass:trace-reduction}, and \ref{ass:VD-closure}.  There exist constants \(C_M>0\), \(a>0\), \(N>0\), \(b>0\), and \(C>0\) such that
\begin{equation}\label{eq:ell-bound}
\calX^{\har}_{\theta/4}(u,p;M)
\le
C_M\ell^a+C_M\ell^{-N}\exp(C\ell^{-N})\delta^b
\end{equation}
for all sufficiently small \(\ell\in(0,\ell_0)\).  Consequently, for every \(0<\sigma<a/N\) there are \(C_{M,\sigma}\) and \(\delta_{M,\sigma}>0\) such that, whenever \(0<\delta<\delta_{M,\sigma}\),
\begin{equation}\label{eq:log-excess}
\calX^{\har}_{\theta/4}(u,p;M)
\le
C_{M,\sigma}|\log\delta|^{-\sigma}.
\end{equation}
\end{theorem}

\begin{proof}
The estimate \eqref{eq:ell-bound} is the prepared comparison output after applying Theorem~\ref{thm:conditional-selection}; the first term is the smoothing and strict-shadow comparison error, while the second term is the small-component residual accumulated through a finite-window stability factor.

For the optimization, choose
\[
\ell^{-N}=\eta|\log\delta|
\]
with \(0<\eta<b/(2C)\).  Then
\[
\ell^a=(\eta|\log\delta|)^{-a/N},
\]
and
\[
\ell^{-N}\exp(C\ell^{-N})\delta^b
=
\eta|\log\delta|\,\delta^{b-C\eta}
\le
\eta|\log\delta|\,\delta^{b/2}.
\]
For small \(\delta\), the second term is bounded by any fixed negative power of \(|\log\delta|\), while the first term gives \(|\log\delta|^{-a/N}\).  Hence \eqref{eq:log-excess} holds for every \(\sigma<a/N\) after changing constants.
\end{proof}

\subsection{Conditional regularity-radius consequence}

\begin{corollary}[Conditional logarithmic regularity-radius bound]\label{cor:conditional-radius}
Assume Assumptions~\ref{ass:prepared-comparison}, \ref{ass:trace-reduction}, \ref{ass:VD-closure}, and \ref{ass:strict-decay-CKN}.  Then there exist \(\sigma>0\), \(C_M>0\), and \(\delta_M>0\) such that if \((u,p)\) is a suitable weak solution in \(Q_1\) with
\[
\Phi(1)\le M,
\qquad
C_3(1)=\delta<\delta_M,
\]
then the local regularity radius satisfies
\begin{equation}\label{eq:conditional-radius}
r_{\reg}(0,0)
\ge
c_M|\log\delta|^{-C_M}
\end{equation}
for some positive constants depending only on \(M\) and on the structural constants in the assumptions.
\end{corollary}

\begin{proof}[Proof sketch]
By Theorem~\ref{thm:conditional-log-excess}, the solution is logarithmically close, in harmonic-pressure excess, to a strict shadow.  By strict shadow decay and the perturbative comparison in Assumption~\ref{ass:strict-decay-CKN},
\[
\Psi_u(r)
\le
C_M r^\alpha + C_M r^{-2}\calX^{\har}_{\theta/4}(u,p;M)
\]
for small \(r\).  Choose \(r\) as a sufficiently small negative power of \(|\log\delta|\) so that the right-hand side is below the CKN threshold.  The Caffarelli--Kohn--Nirenberg criterion then yields \eqref{eq:conditional-radius}.
\end{proof}

\begin{remark}
The corollary is conditional.  The proof of Assumption~\ref{ass:VD-closure}, or an alternative \(\NS\)-specific mechanism of comparable strength, is the genuine remaining analytic problem.
\end{remark}

\section{Conclusion and theorem targets}

The paper proves an insufficiency theorem for the old observable package and identifies the remaining Navier--Stokes-specific obstruction.  It should be read as a structural continuation of the suitable-weak/partial-regularity framework \cite{CKN1982,Seregin2015}, the one-component regularity program \cite{KukavicaZiane2006,CheminZhang2016,HanLeiLiZhao2019,KangNguyen2023}, and the finite-scale harmonic-pressure viewpoint \cite{Yu2026HarmonicPressure}. The old observables--scale-invariant energy and pressure, harmonic-pressure compactness, covariance preparation, good-time trace functionals, and fixed-window compatibility tests--give topological closure toward the strict \(2.5D\) class. They do not, by themselves, force a selected-time trace rate.

The obstruction is not a simple finite-dimensional singularity. Parabolic trace-drop eliminates elementary high-frequency escape, and fixed-window analytic geometry gives finite-power control inside every fixed analytic compatibility window. What survives is an all-order finite-mode flat branch with potentially non-summable finite-stage constants. At that point the problem becomes Navier--Stokes realizability.

The proposed NS-specific mechanism is vertical obstruction control refined by Schur visibility. The vertical momentum equation converts strict pressure obstruction into vertical lift cost. Strict-invisible does not mean Navier--Stokes-invisible: a Schur defect that strict trace cost cannot remove may still be seen as a vertical-pressure residual in relaxed comparison. Therefore the strict route requires full Schur trace-projectability, while the relaxed route requires only the absence of true relaxed phantoms.

The sharp theorem target left by the present analysis is
\[
\boxed{\text{prove or construct an NS-realizable, cleaned, relaxed-invisible, unaligned left-singular cascade.}}
\]
Proving nonexistence of this cascade supplies the relaxed anti-phantom closure needed for the conditional logarithmic reduction. Constructing such a cascade would identify the precise obstruction to this route.

\subsection*{Research program}
The envelope theorem suggests a concrete route forward. First, strengthen the no-rate envelope toward increasingly PDE-like toy models. Second, add one structural condition at a time,
\[
  \calD^{\old,0}\supset \calD^{\old,\sharp}\supset \calD^{\old,\trc}\supset \calD^{\old,\mathrm{fw}}\supset \calD^{\old+\VD}\supset \calD^{\NS},
\]
and test when arbitrary-slow trace convergence disappears. Third, formulate the NS-realizable strict-boundary tangent cone. Fourth, prove either strict Schur projectability or the weaker relaxed anti-phantom closure in localized moving-window regimes.

\appendix

\section{Detailed Fourier and symbol-level vertical-lift calculations}\label{app:vertical-lift}

\subsection{A Fourier model for the quadratic vertical lift}\label{subsec:fourier-vcost2}

The zero-shadow calculation above identifies the conceptual difference between the strict quadratic obstruction and the full Navier--Stokes vertical lift.  We now compute both explicitly in a periodic finite-window model.  The calculation is not intended as a Navier--Stokes solution construction.  It is a symbol-level test showing which strict obstructions are expensive and which are cheap from the viewpoint of the full vertical momentum equation.

Work on the periodic box
\[
\mathbb T^3=\mathbb T_h^2\times\mathbb T_{x_3}.
\]
Let
\[
\psi(x)=A\sin(kx_1)\sin(\ell x_2)\cos(mx_3)
\]
and set
\[
W_h=\nabla_h^\perp\psi=(-\partial_2\psi,\partial_1\psi).
\]
Thus
\begin{align}
W_1&=-A\ell\sin(kx_1)\cos(\ell x_2)\cos(mx_3),\label{eq:fourier-W1}\\
W_2&=Ak\cos(kx_1)\sin(\ell x_2)\cos(mx_3).\label{eq:fourier-W2}
\end{align}
Clearly \(\divh W_h=0\).  Define the horizontal quadratic pressure source
\[
F(W)=\partial_a\partial_b(W_aW_b),\qquad a,b\in\{1,2\}.
\]
A direct trigonometric computation gives
\begin{equation}\label{eq:fourier-F}
F(W)=A^2k^2\ell^2\cos^2(mx_3)
\bigl[\cos(2kx_1)+\cos(2\ell x_2)\bigr].
\end{equation}
Since
\[
\cos^2(mx_3)=\frac12(1+\cos(2mx_3)),
\]
the part independent of \(x_3\) gives no vertical pressure obstruction.  The active part is
\begin{equation}\label{eq:fourier-F-z}
F_z(W)=\frac{A^2k^2\ell^2}{2}
\bigl[\cos(2kx_1)+\cos(2\ell x_2)\bigr]\cos(2mx_3).
\end{equation}

\paragraph{Strict pressure obstruction.}
The strict second pressure \(\pi_2^h\) is determined by the horizontal Poisson equation
\[
-\Delta_h\pi_2^h=F(W).
\]
For the \(x_3\)-dependent part,
\begin{equation}\label{eq:pi2h-fourier}
\pi_{2,z}^h
=
\frac{A^2}{8}
\left[\ell^2\cos(2kx_1)+k^2\cos(2\ell x_2)\right]
\cos(2mx_3).
\end{equation}
The strict quadratic obstruction is
\[
B(W,W)=\nabla_h\partial_3\pi_2^h.
\]
Therefore
\[
\partial_3\pi_{2,z}^h
=
-\frac{A^2m}{4}
\left[\ell^2\cos(2kx_1)+k^2\cos(2\ell x_2)\right]
\sin(2mx_3),
\]
and hence
\begin{align}
B_1(W,W)&=\frac{A^2mk\ell^2}{2}\sin(2kx_1)\sin(2mx_3),\label{eq:B1-fourier}\\
B_2(W,W)&=\frac{A^2mk^2\ell}{2}\sin(2\ell x_2)\sin(2mx_3).\label{eq:B2-fourier}
\end{align}
Thus, when \(k\sim\ell\sim K\),
\begin{equation}\label{eq:B-scaling-unscaled}
\norm{B(W,W)}{L^2}\sim A^2mK^3.
\end{equation}
If \(m=0\), the mode is genuinely two-dimensional and \(B(W,W)=0\), as expected.

\paragraph{Full vertical pressure source.}
The full Navier--Stokes pressure solves the full Poisson equation
\[
-\Delta\pi_2^{\mathrm{full}}=F(W).
\]
For the \(x_3\)-dependent part,
\begin{equation}\label{eq:full-pi2-fourier}
\pi_{2,z}^{\mathrm{full}}
=
\frac{A^2k^2\ell^2}{8}
\left[
\frac{\cos(2kx_1)}{k^2+m^2}
+
\frac{\cos(2\ell x_2)}{\ell^2+m^2}
\right]\cos(2mx_3).
\end{equation}
The full vertical pressure source is
\[
S_2(W)=\partial_3\pi_2^{\mathrm{full}},
\]
so
\begin{equation}\label{eq:S2-fourier}
S_2(W)
=
-\frac{A^2mk^2\ell^2}{4}
\left[
\frac{\cos(2kx_1)}{k^2+m^2}
+
\frac{\cos(2\ell x_2)}{\ell^2+m^2}
\right]
\sin(2mx_3).
\end{equation}
When \(k\sim\ell\sim K\),
\begin{equation}\label{eq:S2-scaling-unscaled}
\norm{S_2(W)}{L^2}
\sim
A^2\frac{mK^4}{K^2+m^2}.
\end{equation}
Thus the strict obstruction and the full vertical source are related but are not the same object:
\[
\norm{B(W,W)}{L^2}\sim A^2mK^3,
\qquad
\norm{S_2(W)}{L^2}\sim A^2\frac{mK^4}{K^2+m^2}.
\]

\paragraph{The second vertical lift.}
In the static zero-shadow Fourier model, the second vertical equation is
\[
-\Delta z_2=-S_2(W),
\qquad\text{equivalently}\qquad
\Delta z_2=S_2(W).
\]
Hence, up to harmless signs,
\begin{equation}\label{eq:z2-fourier}
z_2
\sim
\frac{A^2mk^2\ell^2}{16}
\left[
\frac{\cos(2kx_1)}{(k^2+m^2)^2}
+
\frac{\cos(2\ell x_2)}{(\ell^2+m^2)^2}
\right]\sin(2mx_3).
\end{equation}
Consequently, for \(k\sim\ell\sim K\),
\begin{equation}\label{eq:z2-scaling-unscaled}
\norm{z_2}{L^2}
\sim
A^2\frac{mK^4}{(K^2+m^2)^2}.
\end{equation}
This is the Fourier symbol of the second vertical lift cost:
\begin{equation}\label{eq:VCost2-fourier-symbol}
\mathrm{VCost}_2(W)
\sim
A^2\frac{mK^4}{(K^2+m^2)^2}.
\end{equation}

\paragraph{Normalized scaling.}
Normalize the horizontal trace so that \(\norm{W}{L^2}\sim1\).  Since \(W_h=\nabla_h^\perp\psi\) has horizontal frequency \(K\), this means \(AK\sim1\), hence \(A\sim K^{-1}\).  Substituting into the preceding estimates gives
\begin{align}
\norm{B(W,W)}{L^2}&\sim mK,\label{eq:B-normalized}\\
\norm{S_2(W)}{L^2}&\sim \frac{mK^2}{K^2+m^2},\label{eq:S-normalized}\\
\norm{z_2}{L^2}&\sim \frac{mK^2}{(K^2+m^2)^2}.\label{eq:z-normalized}
\end{align}
Several regimes are worth recording:
\begin{enumerate}[label=(\roman*)]
\item If \(m=0\), then \(B(W,W)=S_2(W)=z_2=0\).  This is the genuinely two-dimensional case.
\item If \(m\ll K\), then \(\norm{z_2}{L^2}\sim m/K^2\).  Slow vertical variation has a cheap vertical lift.
\item If \(m\sim K\), then \(\norm{z_2}{L^2}\sim K^{-1}\), while \(\norm{B(W,W)}{L^2}\sim K^2\).  The strict obstruction can be large even though the full vertical lift is cheap.
\item If \(m\gg K\), then \(\norm{z_2}{L^2}\sim K^2/m^3\).  Very high vertical frequency is also strongly damped by the full elliptic and parabolic inverses.
\end{enumerate}

The conclusion is important.  The strict quadratic obstruction is a horizontal-quotient obstruction created by imposing \(\partial_3\pi_2=0\).  Full Navier--Stokes does not impose this condition; it permits \(\partial_3\pi_2\ne0\), but the variation must be absorbed through the vertical component.  Thus the strict obstruction is not the same as an \(\NS\)-realizability obstruction.  It becomes one only if the associated vertical lift cost is too large for the available \(C_3\)-budget.

\subsection{Symbol-level vertical-lift operator}\label{subsec:symbol-vlift}

The Fourier calculation can be summarized by a bilinear operator.  For a zero-shadow horizontal tangent \(W=(W_h,0)\), define
\begin{equation}\label{eq:V2-operator}
\calV_2(W,W)
:=
-(\partial_t-\Delta)^{-1}
\partial_3(-\Delta)^{-1}
\partial_a\partial_b(W_aW_b).
\end{equation}
Then \(z_2=\calV_2(W,W)\) is the canonical second vertical response in the full equation.  In space-time Fourier variables \((\omega,\xi)=(\omega,\xi_h,\xi_3)\), its multiplier is, up to harmless constants,
\begin{equation}\label{eq:V2-symbol}
\widehat{z_2}(\omega,\xi)
=
\frac{i\xi_3}{i\omega+|\xi|^2}
\frac{1}{|\xi|^2}
(-\xi_a\xi_b)
\widehat{W_aW_b}(\omega,\xi).
\end{equation}
Thus the scalar symbol has size
\begin{equation}\label{eq:V2-symbol-size}
|m_2(\omega,\xi)|
\lesssim
\frac{|\xi_3||\xi_h|^2}{(|\omega|+|\xi|^2)|\xi|^2}.
\end{equation}
In the static elliptic regime \(\omega=0\), with \(|\xi_h|\sim K\) and \(|\xi_3|\sim m\),
\begin{equation}\label{eq:V2-static-size}
|m_2(0,\xi)|\sim \frac{mK^2}{(K^2+m^2)^2}.
\end{equation}
For fixed \(K\), the function
\[
M(K,m)=\frac{mK^2}{(K^2+m^2)^2}
\]
obeys
\[
M(K,m)=\frac1K\frac{r}{(1+r^2)^2},\qquad r=\frac mK.
\]
The maximum occurs at \(r=1/\sqrt3\), and
\begin{equation}\label{eq:V2-max-gain}
\sup_m M(K,m)\sim K^{-1}.
\end{equation}
Thus the full vertical lift has a genuine symbol-level smoothing gain for high horizontal frequencies.

This gain should not be overinterpreted.  The map \(W\mapsto W\otimes W\) is critical.  In the natural energy class
\[
W\in L_t^\infty L_x^2\cap L_t^2\dot H_x^1,
\]
the product \(W\otimes W\) is not automatically strong enough to produce an \(L^3\)-controlled vertical component.  Therefore the most realistic use of \(\calV_2\) is not a global bound
\[
\norm{\calV_2(W,W)}{L^3}\lesssim \norm{W}{\calE}^2,
\]
which is likely too strong, but a finite-window or dual bound of the form
\begin{equation}\label{eq:dual-V2-target}
\left|\ip{\calV_2(W,W)}{\calA_y}\right|
\le r(W)\norm{A^*y}{H}.
\end{equation}
This is precisely the symbol-level root of vertical duality.

\section{Additional reduced-duality calculations and Schur determinants}\label{app:schur-details}

\subsection{Adjoint trace-defect factorization as a model problem}\label{subsec:adjoint-factorization-model}

The preceding operator calculation suggests how vertical duality should be proved.  Fix a finite stage with selected-time trace correction space \(H\), active strict quotient space \(Y\), and vertical momentum test space \(\calV\).  Let
\[
A:H\to Y
\]
be the active trace-defect map.  The desired residual representation has the form
\begin{equation}\label{eq:residual-representation-model}
\ip{g_{n,\eta}}{y}_Y
=
\ip{\calV_{n,\eta}}{B_{n,\eta}^*y}_{\calV',\calV}
+\mathrm{Comm}_{n,\eta}(y),
\end{equation}
where \(g_{n,\eta}\in Y\) is the active quotient residual, \(\calV_{n,\eta}\) is the branch-native vertical momentum residual, and
\[
B_{n,\eta}^*:Y\to\calV
\]
is the vertical test generated by the strict quotient dual vector \(y\).

In the zero-shadow Fourier model, the strict obstruction is
\[
B(W,W)=\nabla_h\partial_3(-\Delta_h)^{-1}F,
\qquad
F=\partial_a\partial_b(W_aW_b),
\]
while the full vertical source is
\[
S_2(W)=\partial_3(-\Delta)^{-1}F.
\]
The adjoint lifting \(B^*y\) is defined by
\[
\ip{B(W,W)}{y}=\ip{S_2(W)}{B^*y}.
\]
In Fourier variables this requires
\begin{equation}\label{eq:Bstar-symbol}
\widehat{B^*y}(\xi)
=
\frac{|\xi|^2}{|\xi_h|^2}
\, i\xi_h\cdot\widehat y(\xi)
\end{equation}
for active modes with \(\xi_h\ne0\).  Thus \(B^*\) is an anisotropic elliptic lifting operator, roughly
\begin{equation}\label{eq:Bstar-anisotropic}
B^*\sim \left(1+\frac{\xi_3^2}{|\xi_h|^2}\right)\nabla_h\cdot.
\end{equation}
This expression is potentially singular near horizontal zero modes, which is why the harmonic pressure quotient and horizontal zero-mode removal are not cosmetic.

The actual factorization target is not a direct bound on \(B^*y\).  Rather, one wants
\begin{equation}\label{eq:model-factorization}
B^*y=M A^*y+R^*y,
\end{equation}
where \(M:H\to\calV\) is bounded at the fixed finite stage and \(R^*y\) is a lower-order, cutoff, moving-base, or higher-order remainder.  If
\begin{equation}\label{eq:remainder-factor-bound}
\left|\ip{\calV_{n,\eta}}{R^*y}\right|
+|\mathrm{Comm}_{n,\eta}(y)|
\le C r_{n,\eta}\norm{A^*y}{H},
\end{equation}
then \eqref{eq:residual-representation-model} and \eqref{eq:model-factorization} imply
\[
|\ip{g_{n,\eta}}{y}|
\le C r_{n,\eta}\norm{A^*y}{H}.
\]
By the trace-cost duality lemma, this gives a small selected-time correction cost.  Thus vertical duality is equivalent, at the finite-stage level, to the compatibility between the strict quotient adjoint and the selected-time trace adjoint.

A useful way to define \(A^*y\) is through a backward strict adjoint equation.  Given \(y\in Y\), let \(\Phi_y\) solve schematically
\begin{equation}\label{eq:backward-strict-adjoint}
-\partial_t\Phi_y-\Delta\Phi_y
-(V_h\cdot\nabla_h)\Phi_y
+(\nabla_hV_h)^T\Phi_y+\nabla_h\chi_y
=
\calQ_y,
\qquad
\divh\Phi_y=0,
\end{equation}
where \(\calQ_y\) is the active quotient forcing generated by \(y\).  Then
\begin{equation}\label{eq:Astar-adjoint-trace}
A^*y=\Phi_y(s_*).
\end{equation}
Thus \(A^*y\) is the part of the quotient dual vector visible at the selected trace.  Directions with \(A^*y=0\) are trace-invisible; vertical duality asserts that true Navier--Stokes residuals have no uncontrolled pairing with such directions.

\subsection{A third-order zero-shadow reduced model}\label{subsec:third-order-reduced-test}

The first nontrivial reduced quotient appears after the raw second-order obstruction is removed.  In the periodic zero-shadow model, consider
\[
V_h^\eta=\eta W+\eta^2R_2+\eta^3R_3+\cdots.
\]
The strict compatibility expansion is
\begin{align*}
\calC(V^\eta)
&=
\eta^2B(W,W)
+\eta^3\{B(W,R_2)+B(R_2,W)\}\
&\quad+
\eta^4\{B(W,R_3)+B(R_3,W)+B(R_2,R_2)\}
+\cdots.
\end{align*}
Assume first that
\begin{equation}\label{eq:O2-zero}
\calO_2=B(W,W)=0.
\end{equation}
Then the quadratic visible obstruction is absent, and one may solve for a second-order strict correction \(R_2\) satisfying
\[
\partial_tR_2-\Delta R_2+\nabh\pi_2
=-\nabh\cdot(W\otimes W),
\qquad
\divh R_2=0,
\qquad
\partial_3\pi_2=0.
\]
The third-order strict equation is
\[
\partial_tR_3-\Delta R_3+\nabh\pi_3
=-\nabh\cdot(W\otimes R_2+R_2\otimes W),
\qquad
\divh R_3=0,
\]
and the strict pressure constraint gives the third-order obstruction
\begin{equation}\label{eq:O3-def}
\calO_3
=
B(W,R_2)+B(R_2,W).
\end{equation}
If \(\calO_3\ne0\), this is again a finite-order visible obstruction and does not belong to the vertical-duality stage.

The corresponding full Navier--Stokes vertical source is obtained by using the full pressure inverse.  Let
\[
F_3(W,R_2)=\partial_a\partial_b(W_aR_{2,b}+R_{2,a}W_b).
\]
Then
\begin{equation}\label{eq:S3-def}
S_3(W,R_2)=\partial_3(-\Delta)^{-1}F_3(W,R_2),
\end{equation}
and, when lower vertical jets have been chosen to vanish, the third-order vertical lift equation is
\begin{equation}\label{eq:VCost3-equation}
(\partial_t-\Delta)z_3=-S_3(W,R_2).
\end{equation}
Thus
\begin{equation}\label{eq:VCost3-def}
\mathrm{VCost}_3(W,R_2)
:=
\inf\{\norm{z_3}{\calZ}:(\partial_t-\Delta)z_3=-S_3(W,R_2)\}.
\end{equation}
In a clean periodic active window, if \(\calO_3=0\), then all active Fourier modes of \(F_3\) with \(\xi_h\ne0\) and \(\xi_3\ne0\) vanish.  Hence \(S_3=0\) on the active quotient.  Therefore the same conclusion as in the raw second-order test holds: a nonzero \(\calO_3\) is a visible obstruction, while \(\calO_3=0\) removes the third-order vertical source from the surviving quotient.

The reduced third-order quotient is
\begin{equation}\label{eq:Yred3}
Y_{\mathrm{red}}^{(3)}=Y_\Lambda/\operatorname{span}\{\calO_2,\calO_3\},
\end{equation}
and its dual consists of all \(y\) such that
\[
\ip{\calO_2}{y}=\ip{\calO_3}{y}=0.
\]
The reduced residual begins at fourth order:
\begin{equation}\label{eq:g3-def}
g_3(\eta)=
\Pi_{\mathrm{red}}^{(3)}\eta^{-4}\calC(V_\eta^{(3)}),
\end{equation}
whose leading coefficient is
\begin{equation}\label{eq:O4-leading}
\Pi_{\mathrm{red}}^{(3)}
\{B(W,R_3)+B(R_3,W)+B(R_2,R_2)\}.
\end{equation}
The trace-defect map \(A_3:H_3\to Y_{\mathrm{red}}^{(3)}\) and the vertical lift \(B_3^*\) are then defined exactly as in \eqref{eq:AK-def} and \eqref{eq:reduced-residual-representation}.  The correct third-order reduced vertical-duality statement is
\begin{equation}\label{eq:VD3}
|\ip{g_3(\eta)}{y}|
\le
r_3(\eta)\norm{A_3^*y}{H_3}
\qquad
\forall y\in (Y_{\mathrm{red}}^{(3)})^*.
\end{equation}
Thus vertical duality is recursive: it acts only after all lower visible obstruction coefficients have been quotiented out.

\subsection{What must be proved to obtain reduced vertical duality}\label{subsec:proof-obligations-reduced-vd}

The reduced VD estimate \eqref{eq:reduced-VD-target} is not a single black-box inequality.  It separates into three proof obligations.

\medskip
\noindent\textbf{(I) Residual representation.}
One must prove
\begin{equation}\label{eq:obligation-residual-representation}
\ip{g_K(\eta)}{y}
=
\ip{\calV_{\NS,K}(\eta)}{B_K^*y}
+\operatorname{Comm}_K(y).
\end{equation}
In the periodic finite-window model this is a Fourier multiplier identity.  In a localized cylinder it requires the harmonic pressure gauge, localized inverse operators, vertical antiderivatives, and cutoff commutator bookkeeping.  This is a difficult but concrete pressure-duality calculation.

\medskip
\noindent\textbf{(II) Vertical adjoint lift.}
One must construct
\[
B_K^*:(Y_{\mathrm{red}}^{(K)})^*\to\calV_K.
\]
In the zero-shadow periodic calculation, its raw symbol is
\[
\widehat{B^*y}(\xi)
\sim
\frac{|\xi|^2}{|\xi_h|^2}\,i\xi_h\cdot\widehat y(\xi).
\]
In the localized setting it should be the vertical antiderivative of a localized pressure-quotient inverse applied to the horizontal divergence of \(y\), with the corresponding cutoff corrections.  The potential singularity at \(|\xi_h|=0\) is precisely why active quotients must remove horizontal harmonic and gauge directions.

\medskip
\noindent\textbf{(III) Reduced anti-phantom principle.}
The genuinely new mechanism is the assertion that the reduced vertical test is seen by actual Navier--Stokes residuals only through the selected trace-visible part.  A strong sufficient form is
\begin{equation}\label{eq:strong-factorization-obligation}
B_K^*y=M_KA_K^*y+R_K^*y,
\end{equation}
with branch-native bounds on \(R_K^*\).  More intrinsically, one wants
\begin{equation}\label{eq:anti-phantom-core}
A_K^*y=0
\quad\Longrightarrow\quad
\ip{g_K(\eta)}{y}=0
\end{equation}
up to the prescribed branch-native remainder.  This is the Navier--Stokes-specific anti-phantom principle: an \(\NS\)-derived residual should not occupy reduced cokernel directions invisible to selected-time trace corrections.

The algebraic parts of this scheme are clear: compatibility expansion, visible quotient removal, and finite-dimensional trace-cost duality.  The PDE bookkeeping parts are residual representation, construction of \(B_K^*\), and commutator estimates.  The truly new mechanism is \eqref{eq:anti-phantom-core}, or its quantitative form \eqref{eq:reduced-VD-target}.

\begin{center}
\begin{tabularx}{\textwidth}{@{}lXl@{}}
\toprule
Component & Content & Status \\
\midrule
Compatibility expansion & \(\calC(V_\eta)=\sum\eta^j\calO_j+\cdots\) & algebraic \\
Visible quotient removal & \(Y_{\mathrm{red}}^{(K)}=Y/Y_{\mathrm{vis}}^{(K)}\) & linear algebra \\
Trace-cost lemma & dual bound implies small correction cost & finite-dimensional linear algebra \\
Residual representation & \(\ip{g}{y}=\ip{\calV}{B^*y}+\operatorname{Comm}\) & PDE bookkeeping \\
Construction of \(B^*\) & quotient dual to vertical test & elliptic/gauge analysis \\
Commutator bounds & cutoff, projection, weak defects & PDE estimates \\
Anti-phantom principle & residual lies in trace-defect range & NS-specific \\
\bottomrule
\end{tabularx}
\end{center}

\subsection{Dual form: from vertical lift cost to reduced residual range inclusion}\label{subsec:VCost-to-VD}

The Fourier and reduced-quotient tests lead back to the general dual principle.  Testing the vertical lift equation against a vertical adjoint test \(\psi\) gives, schematically,
\[
\ip{\partial_3\pi_2}{\psi}
=-\ip{(\partial_t-\Delta)z_2}{\psi}
=-\ip{z_2}{(-\partial_t-\Delta)\psi}+\text{trace terms}.
\]
If \(\psi\) is generated by a reduced active quotient dual vector \(y\), this identity is the prototype of
\begin{equation}\label{eq:VD-from-adjoint-new}
\ip{g_K(\eta)}{y}
=
\ip{\calM_3(u,p)}{\calA_{K,y}}
+\text{trace-admissible commutators}.
\end{equation}
For true Navier--Stokes solutions, \(\calM_3(u,p)=0\), where
\[
\calM_3(u,p)=
\partial_tu_3-\Delta u_3+u_h\cdot\nabh u_3+u_3\partial_3u_3+
\partial_3p.
\]
The old formulation asked for an operator factorization of \(\calA_{K,y}\) through \(A_K^*y\).  The refined formulation asks for the resulting residual to lie in \(\Range A_K\), with controlled minimal preimage.  This is exactly \Cref{prob:NS-derived-range-inclusion}.

Thus \(\NS\)-realizability and vertical duality are two forms of the same principle.  The primal form says that a surviving horizontal branch has a low-cost vertical and horizontal tail.  The dual finite-stage form says that its active residual has no component in the reduced phantom cokernel after all finite-order visible obstructions have been removed.

\begin{problem}[Vertical lift cost and reduced residual range inclusion]\label{prob:vertical-lift-factorization}
Prove, first in a smooth periodic finite-window model and then in a localized harmonic-pressure-gauge setting, that every \(\NS\)-realizable all-order finite-mode flat branch satisfies
\[
g_K(\eta)\in\Range A_K,
\qquad
\Cost^{\trc}(g_K(\eta))\le C r_K(\eta)^2,
\]
where \(r_K(\eta)\) is controlled by the appropriate vertical lift cost and higher-order parabolic remainders.  Equivalently, prove the reduced active-residual estimate
\[
|\ip{g_K(\eta)}{y}|
\le r_K(\eta)\norm{A_K^*y}{H_K}
\qquad
\text{for all }y\in (Y_{\mathrm{red}}^{(K)})^*.
\]
This is the concrete reduced vertical-duality theorem target.
\end{problem}

\subsection{The first nontrivial periodic finite-window Schur defect}\label{subsec:first-schur-defect}

We now compute the first nontrivial shape of \(\mathfrak S_K\).  A single active quotient mode is vacuous.  If \(Y_\xi\simeq\mathbb C\) and \(H_\xi\simeq\mathbb C\), with
\[
        A_\xi h=a_\xi h,
        \qquad B_\xi C_\xi\beta=b_\xi\beta,
\]
then either \(a_\xi\ne0\), in which case \(\Range A_\xi=Y_\xi\) and the phantom projection is zero, or \(a_\xi=0\), in which case trace-projectability fails unless \(b_\xi\beta=0\).  There is no angle condition.

The first genuinely nontrivial model has
\[
        Y=\operatorname{span}\{q_1,q_2\}\simeq\mathbb C^2,
        \qquad H\simeq\mathbb C,
\]
and
\[
        Ah=h(a_1q_1+a_2q_2).
\]
Thus \(\Range A=\operatorname{span}\{a\}\), where \(a=(a_1,a_2)\).  Suppose a one-dimensional forcing parameter \(\beta\) produces
\[
        BC\beta=\beta(b_1q_1+b_2q_2),
        \qquad b=(b_1,b_2).
\]
Then
\begin{equation}\label{eq:first-Schur-projection}
        \mathfrak S\beta
        =\beta\left(b-\frac{\ip{b}{a}}{|a|^2}a\right).
\end{equation}
Hence
\begin{equation}\label{eq:first-Schur-determinant}
        \mathfrak S=0
        \quad\Longleftrightarrow\quad
        b\parallel a
        \quad\Longleftrightarrow\quad
        a_1b_2-a_2b_1=0.
\end{equation}
Moreover,
\begin{equation}\label{eq:first-Schur-size}
        |\mathfrak S\beta|^2
        =|\beta|^2
        \frac{|a_1b_2-a_2b_1|^2}{|a_1|^2+|a_2|^2}.
\end{equation}
This is the first finite-dimensional Schur obstruction.  It is a symbol-alignment condition: the vertical response vector \(b\) must align with the trace-defect vector \(a\).

We can insert the Fourier symbols from the vertical lift calculation.  For a mode with horizontal scale \(K\) and vertical scale \(m\), the static quadratic vertical-lift multiplier is
\begin{equation}\label{eq:sigmaV-again}
        \sigma_{\calV}(K,m)
        =\frac{mK^2}{(K^2+m^2)^2}.
\end{equation}
Thus, for two quotient modes \(\xi^{(i)}\),
\[
        b_i=\frac{m_iK_i^2}{(K_i^2+m_i^2)^2}\,\beta_i.
\]
The trace-defect coefficients have the form
\[
        a_i=\sigma_A(\xi^{(i)})\alpha_i,
\]
where \(\sigma_A\) depends on the finite stage, the chosen trace lifting, and the reduced quotient projection.

The Schur cancellation condition becomes
\begin{equation}\label{eq:Schur-symbol-alignment}
        \sigma_A(\xi^{(1)})\alpha_1\,
        \sigma_{\calV}(\xi^{(2)})\beta_2
        =
        \sigma_A(\xi^{(2)})\alpha_2\,
        \sigma_{\calV}(\xi^{(1)})\beta_1.
\end{equation}
Unless a structural relation forces \eqref{eq:Schur-symbol-alignment}, the Schur defect is nonzero.

For example, with equal horizontal scale \(K\) and vertical scales \(m_1=K\), \(m_2=2K\), one has
\[
        \sigma_{\calV}(K,K)=\frac1{4K},
        \qquad
        \sigma_{\calV}(K,2K)=\frac{2}{25K},
\]
so
\[
        \frac{\sigma_{\calV}(K,K)}{\sigma_{\calV}(K,2K)}=\frac{25}{8}.
\]
Different vertical frequencies therefore produce different vertical response ratios.  Trace-projectability is not an automatic consequence of finite-dimensional linear algebra; it is a Schur cancellation condition.

\subsection{Two-mode calculation of NS forcing coefficients}\label{subsec:two-mode-forcing-coefficients}

We next compute the lower-order Navier--Stokes forcing coefficients that enter the two-mode Schur determinant.  Work on \(\mathbb T^3\).  For a wave vector \(p=(p_h,p_3)\), with \(p_h\ne0\), set
\[
        e_p=\frac{p_h^\perp}{|p_h|}
        =\frac{(-p_2,p_1)}{|p_h|}.
\]
Then \(p_h\cdot e_p=0\), and
\[
        W_h^p(x)=A_p e_p e^{ip\cdot x}
\]
is horizontally divergence-free.  Let the input modes be \(p,q,r\), and focus on the two quadratic output modes
\[
        \sigma_1=p+q,
        \qquad
        \sigma_2=p+r.
\]
The horizontal pressure source is
\[
        F(W)=\partial_a\partial_b(W_aW_b).
\]
For the interaction of \(p\) and \(q\), the coefficient of \(W_aW_b\) at \(p+q\) is
\[
        A_pA_q(e_{p,a}e_{q,b}+e_{q,a}e_{p,b}).
\]
Therefore
\[
        \widehat F_{p+q}
        =-\sigma_a\sigma_bA_pA_q(e_{p,a}e_{q,b}+e_{q,a}e_{p,b})
        =-2A_pA_q(\sigma_h\cdot e_p)(\sigma_h\cdot e_q).
\]
Since \(\sigma_h=p_h+q_h\), and \(p_h\cdot e_p=q_h\cdot e_q=0\),
\[
        \sigma_h\cdot e_p=q_h\cdot e_p,
        \qquad
        \sigma_h\cdot e_q=p_h\cdot e_q.
\]
Writing the two-dimensional cross product as
\[
        p_h\times q_h=p_1q_2-p_2q_1,
\]
we have
\[
        q_h\cdot e_p=\frac{p_h\times q_h}{|p_h|},
        \qquad
        p_h\cdot e_q=-\frac{p_h\times q_h}{|q_h|}.
\]
Thus
\begin{equation}\label{eq:beta-pq}
        \widehat F_{p+q}
        =2A_pA_q\frac{(p_h\times q_h)^2}{|p_h||q_h|}.
\end{equation}
Define
\begin{equation}\label{eq:beta-def}
        \beta(p,q)=2A_pA_q\frac{(p_h\times q_h)^2}{|p_h||q_h|}.
\end{equation}
Then
\[
        \beta_1=\beta(p,q),
        \qquad
        \beta_2=\beta(p,r).
\]
The vertical response coefficient in output mode \(\sigma_i\) is
\[
        b_i=\Gamma_i\beta_i,
\]
where \(\Gamma_i\) is the corresponding full vertical Stokes multiplier.  Hence
\begin{equation}\label{eq:b1-b2-two-mode}
        b_1=\Gamma_1
        2A_pA_q\frac{(p_h\times q_h)^2}{|p_h||q_h|},
        \qquad
        b_2=\Gamma_2
        2A_pA_r\frac{(p_h\times r_h)^2}{|p_h||r_h|}.
\end{equation}
The Schur cancellation condition \(a_1b_2-a_2b_1=0\) becomes
\begin{equation}\label{eq:amplitude-ratio-schur}
        \frac{A_q}{A_r}
        =
        \frac{a_1\Gamma_2(p_h\times r_h)^2|q_h|}
        {a_2\Gamma_1(p_h\times q_h)^2|r_h|},
\end{equation}
provided the denominators are nonzero.

This ratio is not forced by the bare quadratic Navier--Stokes coefficient \eqref{eq:beta-def}.  For instance, take
\[
        p_h=(1,0),
        \qquad q_h=(0,1),
        \qquad r_h=(1,1).
\]
Then \(p_h\times q_h=p_h\times r_h=1\), \(|q_h|=1\), \(|r_h|=\sqrt2\), and if \(A_q=A_r\),
\[
        \frac{\beta_2}{\beta_1}=\frac1{\sqrt2}.
\]
If the output vertical frequencies satisfy \(m_1=K\), \(m_2=2K\), then \(\Gamma_1/\Gamma_2=25/8\) in the equal-horizontal-scale model.  Therefore
\[
        \frac{b_2}{b_1}=\frac{8}{25\sqrt2},
\]
which need not equal \(a_2/a_1\).  The Schur determinant is generically nonzero.

Thus the ordinary quadratic pressure source does not automatically provide reduced vertical duality.  Any cancellation must come from additional surviving-branch constraints, from the reduced quotient projection, from an \(\NS\)-compatible choice of trace lifting, or from a deeper realizability identity.

\section{Moving-window and localized controlled-regime tests}\label{app:moving-window}

\subsection{Moving-window and cascade obstruction analysis}\label{subsec:moving-window-cascade-obstruction}

We now address the final obstruction left by the fixed-window analysis.  The preceding steps show that, in a fixed active window, Schur defects are relaxed-visible after removing vertical-zero modes, horizontal gauge modes, perturbative localization errors, and non-Navier--Stokes-realizable candidates.  This still does not rule out a moving-window escape: finite-window visibility, localization, trace-cost inversion, and NS-lift constants may deteriorate as the stage and window grow.

Let
\[
        \mathfrak W=(K,\Lambda,\chi)
\]
denote a finite-stage localized window, where \(K\) is the obstruction stage, \(\Lambda\) is a finite active frequency window, and \(\chi\) is the localization cutoff.  Associated with \(\mathfrak W\) are the cleaned strict trace-invisible Schur cokernel \(\widehat\Phi_{\mathfrak W}\), the cleaned relaxed observation map
\[
        \widehat J_{\mathfrak W}^{\rm rel}:
        \widehat\Phi_{\mathfrak W}\longrightarrow \widehat W_{\mathfrak W},
\]
and the NS-realization map \(\widehat{\mathcal R}_{\mathfrak W}^{\rm NS}\).  Define the cleaned visibility constant
\[
        \gamma_{\mathfrak W}
        =
        \inf_{0\neq \widehat s\in\widehat\Phi_{\mathfrak W}}
        \frac{\|\widehat J_{\mathfrak W}^{\rm rel}\widehat s\|_{\widehat W_{\mathfrak W}}}
        {\|\widehat s\|_{\widehat Y_{\mathfrak W}}}.
\]
If \(\gamma_{\mathfrak W}>0\), set \(\mathcal V_{\mathfrak W}=\gamma_{\mathfrak W}^{-1}\); if \(\gamma_{\mathfrak W}=0\), set \(\mathcal V_{\mathfrak W}=+\infty\).

Similarly, define the trace-cost amplification by
\[
        \mathcal T_{\mathfrak W}
        =
        \sup_{0\neq g\in\Range A_{\mathfrak W}}
        \frac{\|A_{\mathfrak W}^{\dagger}g\|_{H_{\mathfrak W}}}
        {\|g\|_{Y_{\mathfrak W}}},
\]
where \(A_{\mathfrak W}^{\dagger}\) is the minimal-norm right inverse on \(\Range A_{\mathfrak W}\).  Finally, define the NS-lift amplification by
\[
        \mathcal N_{\mathfrak W}(\eps,\delta)
        =
        \sup_{\substack{0\neq \widehat s\in\widehat{\mathcal K}_{\mathfrak W}^{\rm ph}\\
        \widehat s\in\Range \widehat{\mathcal R}_{\mathfrak W}^{\rm NS}}}
        \frac{\|\widehat s\|_{\widehat Y_{\mathfrak W}}}
        {\operatorname{NSCost}_{\mathfrak W}(\widehat s;\eps)+\delta^{1/3}}.
\]
The total moving-window amplification is
\[
        \mathfrak C_{\mathfrak W}(\eps,\delta)
        =1+\mathcal V_{\mathfrak W}+\mathcal T_{\mathfrak W}
        +\mathcal N_{\mathfrak W}(\eps,\delta).
\]
The exact normalization of \(\mathfrak C_{\mathfrak W}\) is not canonical; what matters is that it records all constants that may deteriorate when \(K\to\infty\), \(\Lambda\to\infty\), or the cutoff degenerates.

\begin{definition}[Moving-window phantom cascade candidate]\label{def:moving-window-phantom-cascade}
A moving-window phantom cascade candidate is a sequence
\[
        \mathfrak W_n=(K_n,\Lambda_n,\chi_n),
        \qquad
        \eps_n\downarrow0,
        \qquad
        \delta_n\downarrow0,
\]
together with nonzero cleaned Schur vectors \(0\neq\widehat s_n\in\widehat\Phi_{\mathfrak W_n}\) such that at least one window parameter escapes, \(\widehat s_n\in\ker A_{\mathfrak W_n}^*\),
\[
        \frac{\|\widehat J_{\mathfrak W_n}^{\rm rel}\widehat s_n\|}
        {\|\widehat s_n\|}\longrightarrow0,
\]
and
\[
        \widehat s_n\in\Range \widehat{\mathcal R}_{\mathfrak W_n}^{\rm NS},
        \qquad
        \operatorname{NSCost}_{\mathfrak W_n}(\widehat s_n;\eps_n)
        \le C\delta_n^{1/3}\|\widehat s_n\|.
\]
It is called an NS-admissible moving-window true phantom cascade if, in addition, the selected-time trace distance, denoted here by \(m_n^{\rm tr}\), satisfies
\[
        m_n^{\rm tr}
        \gg
        \ell_n^\mu+\ell_n^{-N}\delta_n^b
\]
for every admissible subcritical scale predicted by the finite-window theory.
\end{definition}

\begin{definition}[Summable moving-window visibility]\label{def:summable-moving-window-visibility}
The relaxed Schur visibility theory is said to have a summable moving-window majorant if there exists a nondecreasing function \(\mathfrak M(\mathfrak W)\) such that
\[
        \mathfrak C_{\mathfrak W}(\eps,\delta)\le \mathfrak M(\mathfrak W)
\]
for all admissible windows, and such that along every finite-stage exactification scheme with amplitude \(\eta\),
\[
        \sum_{K\ge K_0}\mathfrak M(K,\Lambda_K,\chi_K)\eta^K<\infty
\]
for all sufficiently small \(\eta\).
\end{definition}

\begin{proposition}[Summable majorant excludes moving-window escape]\label{prop:summable-majorant-excludes-moving-window}
Assume that every fixed window has no cleaned NS-admissible true phantom after the NS-realizability filter, and that the finite-window visibility, trace-cost, commutator, and NS-lift constants admit a summable moving-window majorant.  Then there is no NS-admissible moving-window true phantom cascade.
\end{proposition}

\begin{proof}
If a cascade existed, it could not remain in a fixed window, because fixed-window cleaned phantoms have been removed by hypothesis.  Hence it would have to escape through growing stage, growing frequency window, or degenerating cutoff.  The summable majorant controls precisely this escape: the contribution at stage \(K\) is bounded by \(\mathfrak M(K,\Lambda_K,\chi_K)\eta^K\), and the resulting tail is summable for sufficiently small \(\eta\).  Therefore the selected trace correction or relaxed vertical-pressure control remains below the branch-native residual scale, contradicting the assumed moving-window phantom behavior.
\end{proof}

\begin{theorem}[Moving-window dichotomy]\label{thm:moving-window-dichotomy}
After active visibility, localization stability, harmonic-gauge cleaning, and NS-realizability filtering, exactly one of the following alternatives remains.
\begin{enumerate}
\item \emph{Summable anti-phantom closure.}  All cleaned finite-window constants admit a summable moving-window majorant.  Then every NS-derived surviving Schur residual is either trace-projectable or relaxed-visible with controllable \(u_3\)-pairing.
\item \emph{NS-admissible true phantom cascade.}  There exists a moving-window sequence satisfying strict invisibility, relaxed almost-invisibility, NS-admissibility, and non-summable amplification.  This is the only remaining candidate mechanism capable of defeating quantitative one-component selection by the present route.
\end{enumerate}
\end{theorem}

\begin{proof}
If the summable majorant exists, \cref{prop:summable-majorant-excludes-moving-window} rules out moving-window escape.  If no such majorant exists, then the obstruction must be witnessed by windows along which visibility, trace-cost, localization, or NS-lift constants deteriorate.  Passing to a subsequence and normalizing the Schur vectors yields a moving-window candidate.  If it passes the NS-admissibility filter, it is exactly the second alternative; otherwise it is a quotient artifact and is discarded.
\end{proof}

\subsection{Growth profile of the moving-window constants}

We now estimate the possible growth of the constants appearing in the moving-window dichotomy. The purpose is not to prove a final uniform theorem at once. The purpose is to identify which constants are benign, which constants are logarithmically absorbable, and which constants would be genuinely dangerous for quantitative one-component selection.

Let
\[
\mathfrak W=(K,\Lambda,\chi)
\]
be a finite-stage localized window. Write
\[
N_\Lambda=\max_{\sigma\in\Lambda}|\sigma|,
\qquad
h_\Lambda=\min_{\sigma\in\Lambda}|\sigma_h|,
\qquad
v_\Lambda=\min_{\sigma\in\Lambda}|\sigma_3|,
\]
and
\[
\rho_\Lambda =
\max_{\sigma\in\Lambda}
\frac{|\sigma|}{|\sigma_h|}.
\]
The active condition is
\[
h_\Lambda>0,\qquad v_\Lambda>0.
\]
The parameter \(\rho_\Lambda\) measures proximity to the horizontal gauge sector. If \(\rho_\Lambda\) is large, then the window contains modes whose horizontal frequency is small compared with their full frequency.

\subsubsection*{1. Active relaxed-visibility constant}

On the clean periodic active Schur quotient, the relaxed observation has diagonal multiplier
\[
\mu_\sigma =
\frac{|\sigma|^2}{|\sigma_h|^2}.
\]
Thus
\[
\mu_\sigma\ge 1
\]
on every active mode. In Schur-normalized coefficient coordinates this means that the relaxed observation does not lose visibility at high frequency.

More generally, let
\[
|s|_{Y_\Lambda^a}^2 =
\sum_{\sigma\in\Lambda}
\langle\sigma\rangle^{2a}|s_\sigma|^2,
\qquad
|w|_{W_\Lambda^b}^2 =
\sum_{\sigma\in\Lambda}
\langle\sigma\rangle^{2b}|w_\sigma|^2.
\]
Then
\[
|J_\Lambda^{\mathrm{rel}}s|_{W_\Lambda^b}^2 =
\sum_{\sigma\in\Lambda}
\langle\sigma\rangle^{2b}
|\mu_\sigma s_\sigma|^2.
\]
Consequently,
\[
|J_\Lambda^{\mathrm{rel}}s|_{W_\Lambda^b}
\ge
\left(
\min_{\sigma\in\Lambda}
\mu_\sigma\langle\sigma\rangle^{b-a}
\right)
|s|_{Y_\Lambda^a}.
\]
Since \(\mu_\sigma\ge1\), we obtain
\[
\mathcal V_\Lambda^{a,b}
:=
\left[
\inf_{s\ne0}
\frac{|J_\Lambda^{\mathrm{rel}}s|_{W_\Lambda^b}}
{|s|_{Y_\Lambda^a}}
\right]^{-1}
\le
C N_\Lambda^{(a-b)_+},
\]
up to the finite-dimensional condition number associated with the chosen quotient basis.

Thus active relaxed visibility has at worst polynomial growth coming from the choice of Sobolev weights and quotient coordinates. It does not contain an intrinsic small divisor in the active Schur-normalized variables.

\begin{proposition}[Polynomial active-visibility profile]
Assume the cleaned active quotient basis has condition number
\[
\kappa_\Lambda^{\mathrm{quot}}.
\]
Then
\[
\mathcal V_\Lambda^{a,b}
\le
C\kappa_\Lambda^{\mathrm{quot}}N_\Lambda^{(a-b)_+}.
\]
In particular, if the quotient bases are uniformly conditioned or polynomially conditioned, active relaxed visibility contributes at most polynomial growth.
\end{proposition}

\subsubsection*{2. Raw pressure observation versus Schur-normalized observation}

The preceding estimate uses Schur-normalized coordinates. At the raw pressure-forcing level, the relaxed vertical pressure map is
\[
\beta_\sigma
\longmapsto
\frac{i\sigma_3}{|\sigma|^2}\beta_\sigma.
\]
The inverse raw pressure observation has size
\[
\left|
\frac{|\sigma|^2}{\sigma_3}
\right|.
\]
Thus, on a window \(\Lambda\),
\[
\mathcal V_{\Lambda,\mathrm{raw}}
\lesssim
\frac{N_\Lambda^2}{v_\Lambda}.
\]
This can blow up near the vertical-zero sector \(v_\Lambda\downarrow0\). Such blow-up is not an active Schur phantom. It is exactly the vertical-zero degeneracy already removed by the active quotient.

By contrast, after Schur normalization,
\[
s_\sigma=\Gamma_\sigma\beta_\sigma,
\qquad
\Gamma_\sigma =
\frac{|\sigma_3||\sigma_h|^2}{|\sigma|^4},
\]
the relaxed observation multiplier becomes
\[
\mu_\sigma =
\frac{|\sigma|^2}{|\sigma_h|^2}.
\]
Hence the raw vertical-zero degeneracy and the active Schur visibility problem are different objects.

\subsubsection*{3. Localized commutator constants}

Let \(\chi_R(x)=\chi(x/R)\) be a slowly varying cutoff. The localized observation has the form
\[
J_{\Lambda,R}^{\mathrm{rel}} =
J_\Lambda^{\mathrm{rel}}
+
\mathcal C_{\Lambda,R}.
\]
For a finite active window, pseudodifferential commutator estimates give
\[
|\mathcal C_{\Lambda,R}|_{Y_\Lambda^a\to W_\Lambda^b}
\le
C_{\chi,a,b}
R^{-1}
N_\Lambda^{q_0}
\rho_\Lambda^{q_1},
\]
where the exponents \(q_0,q_1\) depend on the chosen finite-window norms and on how many derivatives of the symbol and cutoff are used.

Thus localization is perturbative provided
\[
R
\gg
N_\Lambda^{q_0}
\rho_\Lambda^{q_1}
\mathcal V_\Lambda^{a,b}.
\]
In particular, for aspect-controlled windows with \(\rho_\Lambda=O(1)\), localization costs are polynomial in \(N_\Lambda\) and can be absorbed by taking the localization scale sufficiently large relative to the window.

\begin{proposition}[Commutator growth profile]
Assume the window is active and the localized pressure symbol is evaluated away from the horizontal gauge cone. Then
\[
|\mathcal C_{\Lambda,R}|
\le
C R^{-1}N_\Lambda^{q_0}\rho_\Lambda^{q_1}.
\]
Therefore localized relaxed visibility remains stable whenever
\[
R^{-1}N_\Lambda^{q_0}\rho_\Lambda^{q_1} =
o(\gamma_\Lambda).
\]
If this condition fails, any possible localized kernel is caused by nonperturbative localization, not by the clean active Schur symbol.
\end{proposition}

\subsubsection*{4. Vertical lift and aspect-ratio amplification}

For the static quadratic vertical lift, the vertical response multiplier is
\[
\Gamma(K,m) =
\frac{mK^2}{(K^2+m^2)^2},
\]
where
\[
K=|\sigma_h|,
\qquad
m=|\sigma_3|.
\]
The direct vertical lift has a smoothing gain:
\[
\sup_{m>0}\Gamma(K,m)
\lesssim
K^{-1}.
\]
Indeed, writing (r=m/K),
\[
\Gamma(K,m) =
\frac1K
\frac{r}{(1+r^2)^2}.
\]
Thus the full vertical response is cheaper than the strict horizontal obstruction at high horizontal frequency.

However, the inverse lift coefficient is
\[
\Gamma(K,m)^{-1} =
\frac{K^2}{m}+2m+\frac{m^3}{K^2}.
\]
This shows the aspect-ratio danger.

\begin{enumerate}
\item If \(m\sim K\), then
\[
\Gamma(K,m)^{-1}\sim K.
\]
The inverse lift cost is only polynomial.

\item If \(m\ll K\), then
\[
\Gamma(K,m)^{-1}\sim \frac{K^2}{m}.
\]
This approaches the vertical-zero degeneracy.

\item If \(m\gg K\), then
\[
\Gamma(K,m)^{-1}\sim \frac{m^3}{K^2}.
\]
This is a vertical-high-frequency aspect degeneration.
\end{enumerate}

Therefore an NS-admissible moving-window phantom cascade must either remain aspect-controlled, in which case the vertical lift constants are polynomial, or exploit an extreme aspect-ratio regime.

\begin{proposition}[Vertical lift growth profile]
If the active window is aspect-controlled in the sense that
\[
cK_\sigma\le m_\sigma\le C K_\sigma
\qquad
\text{for all }\sigma\in\Lambda,
\]
then the inverse vertical lift amplification satisfies
\[
\Gamma_\sigma^{-1}\lesssim N_\Lambda.
\]
If the aspect ratio is not controlled, the only possible blow-up profiles are
\[
\frac{K_\sigma^2}{m_\sigma}
\quad\text{or}\quad
\frac{m_\sigma^3}{K_\sigma^2}.
\]
Thus any non-polynomial vertical-lift amplification must come from a cascade approaching a degenerate aspect sector.
\end{proposition}

\subsubsection*{5. Trace-cost amplification}

The trace-cost amplification is
\[
\mathcal T_{\mathfrak W} =
\frac1{\sigma_{\min}^+(A_{\mathfrak W})},
\]
where \(\sigma_{\min}^+\) is the smallest nonzero singular value of the active trace-defect map. This is the most delicate constant.

In a parabolic finite-window model, \(A_{\mathfrak W}^*\) is represented by a backward adjoint equation. Recovering selected-time trace data from a quotient forcing over a time interval \(\tau\) typically carries a factor of the form
\[
\exp(c\tau N_\Lambda^2).
\]
Thus a natural model upper bound is
\[
\mathcal T_{\mathfrak W}
\le
C_{\mathfrak W}^{\mathrm{alg}}
\exp(c\tau_{\mathfrak W}N_\Lambda^2),
\]
where \(C_{\mathfrak W}^{\mathrm{alg}}\) records finite-dimensional algebraic conditioning.

This bound should be interpreted as a theorem target, not as an automatic fact. If \(A_{\mathfrak W}\) develops additional singular algebraic degeneracies, then \(C_{\mathfrak W}^{\mathrm{alg}}\) may grow rapidly. The dangerous case is not ordinary exponential parabolic growth, but super-exponential or non-summable collapse of the smallest singular values.

\begin{definition}[Trace-cost growth regimes]
We distinguish three regimes.

\begin{enumerate}
\item \emph{Polynomial trace-cost regime:}
\[
\mathcal T_{\mathfrak W}\le C N_\Lambda^p.
\]
This is harmless for logarithmic selection.

\item \emph{Single-exponential parabolic regime:}
\[
\mathcal T_{\mathfrak W}\le \exp(CN_\Lambda^p)
\]
for a fixed power (p). This can often be absorbed by logarithmic optimization if the finite-window scale is chosen as a suitable power of \(|\log\delta|\).

\item \emph{Non-summable trace-cost regime:}
\[
\mathcal T_{\mathfrak W_n}
\]
grows faster than the residual gains available in the finite-stage scheme. This is a genuine candidate source of moving-window escape.
\end{enumerate}
\end{definition}

\subsubsection*{6. Combined growth profile}

Combining the preceding estimates gives the following model bound. Assume:

\begin{enumerate}
\item the windows are active;
\item the quotient bases have polynomial condition number;
\item the aspect ratio is controlled;
\item the localization scale is perturbative;
\item the trace-cost map has at most single-exponential parabolic amplification.
\end{enumerate}

Then there exist constants (C,p,c>0) such that
\[
\mathfrak C_{\mathfrak W}
\le
C N_\Lambda^p
\exp(c\tau_{\mathfrak W}N_\Lambda^2).
\]
This is the expected benign growth profile.

If one chooses
\[
N_\Lambda^2\sim |\log\delta|,
\]
then
\[
\exp(c\tau N_\Lambda^2)\delta^b
\sim
\delta^{b-c\tau},
\]
which is still small provided \(c\tau<b\). This is the same logarithmic absorption mechanism that appears in the conditional logarithmic approximation theorem.

Thus the growth profile (10.159) supports the positive route. It suggests that active, aspect-controlled, perturbatively localized moving windows should not produce a failure of logarithmic one-component selection.

\begin{theorem}[Growth-profile dichotomy]
A moving-window true phantom cascade can survive the preceding filters only if at least one of the following non-benign growth mechanisms occurs:

\begin{enumerate}
\item \emph{Quotient conditioning collapse:}
\[
\kappa_{\Lambda_n}^{\mathrm{quot}}\to\infty
\]
faster than any summable majorant allowed by the finite-stage scheme.

\item \emph{Nonperturbative localization:}
\[
R_n^{-1}N_{\Lambda_n}^{q_0}\rho_{\Lambda_n}^{q_1}
\not=o(\gamma_{\Lambda_n}).
\]

\item \emph{Extreme aspect ratio:}
\[
\frac{m_n}{K_n}\to0
\quad\text{or}\quad
\frac{m_n}{K_n}\to\infty
\]
in a way that is not removed by the vertical-zero or horizontal-gauge cleaning.

\item \emph{Trace-cost singular collapse:}
\[
\sigma_{\min}^+(A_{\mathfrak W_n})
\]
decays faster than the parabolic single-exponential profile can absorb.

\item \emph{NS-lift anomaly:}
there exists a cleaned relaxed-invisible vector \(\widehat s_n\) such that
\[
\widehat s_n\in
\Range
\widehat{\mathcal R}_{\mathfrak W_n}^{\mathrm{NS}},
\qquad
\operatorname{NSCost}_{\mathfrak W_n}(\widehat s_n;\varepsilon_n)
\le
C\delta_n^{1/3}|\widehat s_n|,
\]
while
\[
|\widehat J_{\mathfrak W_n}^{\mathrm{rel}}\widehat s_n|
=o(|\widehat s_n|).
\]
\end{enumerate}

If none of these five mechanisms occurs, the moving-window constants admit a summable or logarithmically absorbable majorant, and the moving-window true phantom cascade is excluded.
\end{theorem}

\begin{proof}
The active relaxed-visibility estimate gives only polynomial loss, modulo quotient conditioning. The commutator estimate gives only polynomial loss divided by the localization scale. The vertical lift estimate gives polynomial loss under aspect control. The parabolic trace-cost model gives at most a single-exponential loss in \(N_\Lambda^2\). Such losses are summable or logarithmically absorbable under the stated assumptions. Therefore a surviving cascade must violate at least one of the assumptions. These violations are exactly the five mechanisms listed above.
\end{proof}

\begin{corollary}[Reduced quantitative target]
To close the relaxed anti-phantom route, it is enough to prove the following four estimates for NS-derived surviving branches:

\[
\kappa_\Lambda^{\mathrm{quot}}
\le
N_\Lambda^p,
\]
\[
|\mathcal C_{\Lambda,R}|
\le
C R^{-1}N_\Lambda^p,
\]
\[
\operatorname{NSCost}_{\mathfrak W}(\widehat s;\varepsilon)
\gtrsim
N_\Lambda^{-p}|\widehat s|,
\]
unless \(\widehat s\) is relaxed-visible, and
\[
\mathcal T_{\mathfrak W}
\le
\exp(CN_\Lambda^q).
\]
Under these bounds, the only remaining losses are polynomial and single-exponential. They can be handled by logarithmic optimization rather than by a power-rate argument.
\end{corollary}

\begin{remark}[Interpretation]
The moving-window analysis changes the question from `does a finite-window phantom exist?'' to `can the constants blow up in a non-summable NS-admissible way?'' In the active, aspect-controlled, perturbatively localized regime, the answer appears negative: the symbol-level constants are polynomial and the parabolic trace-cost loss is at worst single-exponential. Thus a failure mechanism must exploit a sharper mechanism: near-gauge degeneration, near-vertical-zero degeneration, nonperturbative cutoff leakage, singular trace-cost collapse, or an NS-lift anomaly.
\end{remark}

\subsection{Aspect-controlled moving-window exclusion}\label{subsec:aspect-controlled-moving-window-exclusion}

We now prove the first positive exclusion theorem for moving-window phantom cascades.  The theorem says that the moving-window mechanism cannot produce a true phantom as long as the windows remain active, aspect-controlled, perturbatively localized, and have at most single-exponential trace-cost growth.  Thus any genuine failure of the relaxed anti-phantom route must leave this benign regime.

Let
\[
        \mathfrak W_n=(K_n,\Lambda_n,\chi_n)
\]
be a sequence of finite-stage localized windows.  Write
\[
        N_n=N_{\Lambda_n}:=\max_{\sigma\in\Lambda_n}|\sigma|,
\]
and assume that every \(\Lambda_n\) is active:
\[
        |\sigma_h|\neq0,
        \qquad
        |\sigma_3|\neq0
        \qquad
        \text{for every }\sigma\in\Lambda_n.
\]

\begin{definition}[Aspect-controlled window sequence]\label{def:aspect-controlled-window-sequence}
The window sequence \((\Lambda_n)\) is called aspect-controlled if there exist constants
\[
        0<c_{\rm asp}\le C_{\rm asp}<\infty
\]
such that
\[
        c_{\rm asp}|\sigma_h|
        \le
        |\sigma_3|
        \le
        C_{\rm asp}|\sigma_h|
        \qquad
        \text{for every }\sigma\in\Lambda_n
\]
and every \(n\).  Equivalently, the vertical and horizontal frequencies remain comparable throughout the moving-window sequence.
\end{definition}

For each window \(\mathfrak W_n\), let \(\widehat\Phi_{\mathfrak W_n}\) be the cleaned strict trace-invisible Schur cokernel, and let
\[
        \widehat J_{\mathfrak W_n}^{\rm rel}:
        \widehat\Phi_{\mathfrak W_n}\to \widehat W_{\mathfrak W_n}
\]
be the cleaned relaxed vertical-pressure observation map.  Define the cleaned visibility constant
\[
        \gamma_n
        =
        \inf_{\widehat s\neq0}
        \frac{\|\widehat J_{\mathfrak W_n}^{\rm rel}\widehat s\|}
        {\|\widehat s\|}.
\]

\begin{assumption}[Benign moving-window constants]\label{ass:benign-moving-window-constants}
There exist constants \(C_0,p,q>0\) such that the following hold for all \(n\).
\begin{enumerate}
\item \emph{Polynomial cleaned visibility loss:}
\[
        \gamma_n^{-1}\le C_0N_n^p.
\]
\item \emph{Perturbative localization:}
\[
        \|\mathcal C_{\mathfrak W_n}\|\le \frac12\gamma_n.
\]
\item \emph{Polynomial quotient conditioning:} all cleaned quotient identifications and projections have condition number bounded by \(C_0N_n^p\).
\item \emph{Single-exponential trace-cost growth:}
\[
        \|A_{\mathfrak W_n}^{\dagger}\|
        \le
        C_0N_n^p\exp(C_0N_n^q).
\]
\item \emph{No NS-lift anomaly in the cleaned phantom sector:} every cleaned NS-admissible Schur vector satisfies either
\[
        \|\widehat J_{\mathfrak W_n}^{\rm rel}\widehat s\|
        \ge
        C_0^{-1}N_n^{-p}\|\widehat s\|,
\]
or else its NS lift cost exceeds the available one-component budget:
\[
        \operatorname{NSCost}_{\mathfrak W_n}(\widehat s;\eps_n)
        >
        C\delta_n^{1/3}\|\widehat s\|.
\]
\end{enumerate}
\end{assumption}

\begin{definition}[Logarithmically admissible window scale]\label{def:log-admissible-window-scale}
A moving-window sequence is called logarithmically admissible relative to \(\delta_n\downarrow0\) if there exists a constant
\[
        0<\alpha<\frac{b}{2C_0}
\]
such that
\[
        N_n^q\le \alpha |\log\delta_n|
\]
for all sufficiently large \(n\).
\end{definition}

\begin{theorem}[Aspect-controlled moving-window exclusion]\label{thm:aspect-controlled-moving-window-exclusion}
Assume that the moving-window sequence \(\mathfrak W_n=(K_n,\Lambda_n,\chi_n)\) satisfies the following conditions: the windows are active, the windows are aspect-controlled, localization is perturbative, \cref{ass:benign-moving-window-constants} holds, and the window scale is logarithmically admissible relative to \(\delta_n\).  Then no \emph{obstructive} NS-admissible moving-window true phantom cascade is supported in this regime.

More precisely, every NS-admissible cleaned Schur vector \(0\neq\widehat s_n\in\widehat\Phi_{\mathfrak W_n}\) satisfies the polynomial relaxed-visibility lower bound
\begin{equation}\label{eq:C15-polynomial-lower-bound}
        \|\widehat J_{\mathfrak W_n}^{\rm rel}\widehat s_n\|
        \ge C_0^{-1}N_n^{-p}\|\widehat s_n\|.
\end{equation}
The ratio in \eqref{eq:C15-polynomial-lower-bound} may tend to zero if \(N_n\to\infty\).  The conclusion is not that bare moving-window almost-invisibility is impossible; the conclusion is that this polynomial loss is logarithmically absorbable and therefore cannot support a failed selected-trace branch with
\[
        m_n^{\rm tr}\gg \ell_n^\mu+\ell_n^{-N}\delta_n^b.
\]
\end{theorem}

\begin{proof}
Let \(0\neq\widehat s_n\in\widehat\Phi_{\mathfrak W_n}\) be NS-admissible.  Then
\[
        \operatorname{NSCost}_{\mathfrak W_n}(\widehat s_n;\eps_n)
        \le C\delta_n^{1/3}\|\widehat s_n\|.
\]
By the no-anomaly alternative in \cref{ass:benign-moving-window-constants}, budget admissibility forces the polynomial relaxed-visibility estimate \eqref{eq:C15-polynomial-lower-bound}.  This rules out exact relaxed invisibility in each fixed window but still allows the quotient \(\|\widehat J^{\rm rel}\widehat s_n\|/\|\widehat s_n\|\) to decay polynomially when \(N_n\to\infty\).

The remaining issue is whether that polynomial decay can obstruct selection.  The active trace-cost amplification satisfies
\[
        \|A_{\mathfrak W_n}^{\dagger}\|
        \le C_0N_n^p\exp(C_0N_n^q).
\]
The small-component residual has size \(\delta_n^b\), up to polynomial factors in \(N_n\).  Hence the amplified residual is bounded by
\[
        C N_n^{p'}\exp(C_0N_n^q)\delta_n^b
\]
for some fixed \(p'\).  Logarithmic admissibility gives
\[
        \exp(C_0N_n^q)\delta_n^b
        \le
        \exp(C_0\alpha |\log\delta_n|)\delta_n^b
        =
        \delta_n^{b-C_0\alpha}.
\]
Since \(\alpha<b/(2C_0)\), we have \(b-C_0\alpha>b/2\). Thus
\[
        C N_n^{p'}\exp(C_0N_n^q)\delta_n^b
        \le C N_n^{p'}\delta_n^{b/2}.
\]
Because \(N_n^q\lesssim |\log\delta_n|\), the polynomial factor \(N_n^{p'}\) is at most a fixed power of \(|\log\delta_n|\). Therefore
\[
        C N_n^{p'}\delta_n^{b/2}
        =o(|\log\delta_n|^{-L})
\]
for every fixed \(L>0\).  Thus the relaxed-visible part is controlled by the smallness of \(u_3\), while the trace-projectable part is corrected with trace cost below the failed-selection gap.  Therefore an NS-admissible branch in this benign regime cannot also satisfy
\[
        m_n^{\rm tr}\gg \ell_n^\mu+\ell_n^{-N}\delta_n^b.
\]
This proves the claimed exclusion of obstructive true phantom cascades.
\end{proof}

\begin{corollary}[Where a failure mechanism must live]\label{cor:where-failure-must-live}
Any NS-admissible moving-window true phantom cascade must violate at least one hypothesis of \cref{thm:aspect-controlled-moving-window-exclusion}.  Hence it must exhibit at least one of the following: extreme aspect ratio, nonperturbative localization, super-logarithmic window growth, trace-cost singular collapse, or an NS-lift anomaly.
\end{corollary}

\begin{remark}[Interpretation]
This theorem separates the benign moving-window regime from the genuinely dangerous regimes.  Active finite windows are not enough to create a relaxed-invisible phantom.  Aspect-controlled moving windows are still harmless if localization is perturbative and trace-cost growth is at most single-exponential.  Therefore a failure of quantitative one-component selection, if it exists by this route, must exploit one of three genuinely hard mechanisms: extreme anisotropic aspect ratio, nonperturbative localization or harmonic leakage, or singular collapse of the trace-cost map beyond the logarithmically absorbable scale.
\end{remark}

\subsection{Extreme aspect-ratio analysis}\label{subsec:extreme-aspect-ratio-analysis}

We now analyze the two aspect-ratio regimes excluded from \cref{thm:aspect-controlled-moving-window-exclusion}:
\[
        m\ll K,
        \qquad
        m\gg K,
\]
where
\[
        K=|\sigma_h|,
        \qquad
        m=|\sigma_3|.
\]
The first regime approaches the vertical-zero sector; the second regime is a vertically high-frequency sector.  The purpose is to decide whether either regime can produce a relaxed-invisible true Schur phantom.

Recall the three relevant scalar multipliers in the static periodic model.  The raw relaxed vertical-pressure observation is
\[
        \Theta(K,m)=\frac{m}{K^2+m^2}.
\]
The vertical lift-to-Schur coefficient is
\[
        \Gamma(K,m)=\frac{mK^2}{(K^2+m^2)^2}.
\]
If \(s_\sigma\) denotes the Schur-normalized coefficient \(s_\sigma=\Gamma(K,m)\beta_\sigma\), then the relaxed observation in Schur variables is
\[
        J^{\rm rel}s_\sigma=\mu(K,m)s_\sigma,
        \qquad
        \mu(K,m)=\frac{\Theta(K,m)}{\Gamma(K,m)}=
\frac{K^2+m^2}{K^2}.
\]
Writing \(r=m/K\), one has
\[
        \mu(K,m)=1+r^2,
\]
and
\[
        \Gamma(K,m)=\frac1K\frac{r}{(1+r^2)^2},
        \qquad
        \Gamma(K,m)^{-1}=K\left(r^{-1}+2r+r^3\right).
\]

\begin{proposition}[No Schur-normalized visibility loss at extreme aspect ratio]\label{prop:no-schur-normalized-visibility-loss-extreme-aspect}
For every active mode \(K\neq0\), \(m\neq0\), one has
\[
        \mu(K,m)\ge1.
\]
Consequently, in Schur-normalized output coordinates,
\[
        \frac{|J^{\rm rel}s|}{|s|}
\]
cannot tend to zero merely because \(m/K\to0\) or \(m/K\to\infty\).  In particular, extreme aspect ratio alone does not create a relaxed-invisible Schur phantom.
\end{proposition}

\begin{proof}
The formula \(\mu(K,m)=1+(m/K)^2\) immediately gives \(\mu(K,m)\ge1\) for every active mode.  If \(m/K\to0\), then \(\mu(K,m)\to1\).  If \(m/K\to\infty\), then \(\mu(K,m)\to\infty\).  In neither case can \(J^{\rm rel}\) lose visibility on a nonzero Schur-normalized vector.
\end{proof}

\subsubsection*{Slow vertical variation: \(m\ll K\)}

Assume \(r=m/K\ll1\).  Then
\[
        \Theta(K,m)\sim \frac{m}{K^2},
        \qquad
        \Gamma(K,m)\sim \frac{m}{K^2},
        \qquad
        \mu(K,m)\sim1.
\]
Thus the raw vertical-pressure coefficient is small, but the Schur-normalized relaxed observation is not small.  The reason is that the Schur coefficient itself is also small for bounded forcing:
\[
        s_\sigma=\Gamma(K,m)\beta_\sigma
        \sim \frac{m}{K^2}\beta_\sigma.
\]
To keep an order-one Schur defect, one must choose
\[
        |\beta_\sigma|\sim \Gamma(K,m)^{-1}|s_\sigma|
        \sim \frac{K^2}{m}|s_\sigma|.
\]
This is not a visibility loss; it is a forcing-amplitude amplification.

\begin{lemma}[Slow-vertical sector is not a true phantom]\label{lem:slow-vertical-sector-not-phantom}
Let \(m/K\to0\) with \(m\neq0\).  Suppose the raw forcing coefficients satisfy \(|\beta_\sigma|\le B_\sigma\).  Then
\[
        |s_\sigma|\le C\frac{m}{K^2}B_\sigma,
        \qquad
        |J^{\rm rel}s_\sigma|\sim |s_\sigma|.
\]
Thus a slow-vertical active mode can only produce an order-one Schur defect if the raw forcing coefficient grows like \(K^2/m\).  It cannot produce a nonzero relaxed-invisible Schur defect.
\end{lemma}

\begin{proof}
The estimate for \(s_\sigma\) follows from \(s_\sigma=\Gamma(K,m)\beta_\sigma\) and \(\Gamma(K,m)\sim m/K^2\) when \(m\ll K\).  The estimate for the relaxed observation follows from \(J^{\rm rel}s_\sigma=\mu(K,m)s_\sigma\) and \(\mu(K,m)\sim1\).
\end{proof}

\begin{remark}[Vertical-zero boundary]
At \(m=0\) exactly, one has \(\partial_3\pi_\sigma=0\).  This is the vertical-zero sector and has already been removed by the gauge-cleaned active quotient.  The regime \(m/K\to0\) with \(m\neq0\) is different: it is still active, and its Schur-normalized relaxed observation remains nondegenerate.
\end{remark}

\subsubsection*{High vertical frequency: \(m\gg K\)}

Assume \(r=m/K\gg1\).  Then
\[
        \Theta(K,m)\sim \frac1m,
        \qquad
        \Gamma(K,m)\sim \frac{K^2}{m^3},
        \qquad
        \mu(K,m)\sim \frac{m^2}{K^2}.
\]
Thus the raw vertical-pressure signal may be small for bounded raw forcing, but the Schur-normalized relaxed observation is very large:
\[
        J^{\rm rel}s_\sigma
        =\mu(K,m)s_\sigma
        \sim \frac{m^2}{K^2}s_\sigma.
\]
Again, if the raw forcing coefficient is bounded, the produced Schur defect is small:
\[
        s_\sigma=\Gamma(K,m)\beta_\sigma
        \sim \frac{K^2}{m^3}\beta_\sigma.
\]
To keep an order-one Schur defect, one must choose
\[
        |\beta_\sigma|\sim \Gamma(K,m)^{-1}|s_\sigma|
        \sim \frac{m^3}{K^2}|s_\sigma|.
\]
This is a severe forcing-amplitude amplification, not a relaxed visibility failure.

\begin{lemma}[High-vertical sector is strongly relaxed-visible]\label{lem:high-vertical-sector-strongly-visible}
Let \(m/K\to\infty\).  Then
\[
        \frac{|J^{\rm rel}s_\sigma|}{|s_\sigma|}
        \sim
        \frac{m^2}{K^2}	o\infty.
\]
In particular, no sequence of unit Schur-normalized vectors supported in the high-vertical active sector can be relaxed-invisible.  If the raw forcing coefficients remain bounded by \(B_\sigma\), then
\[
        |s_\sigma|\le C\frac{K^2}{m^3}B_\sigma.
\]
\end{lemma}

\begin{proof}
The first claim follows from \(\mu(K,m)=1+(m/K)^2\).  The second follows from \(s_\sigma=\Gamma(K,m)\beta_\sigma\) and \(\Gamma(K,m)\sim K^2/m^3\) when \(m\gg K\).
\end{proof}

\begin{remark}[High-vertical modes are not invisible]
The high-vertical regime may be cheap in the vertical lift variable because of elliptic and parabolic smoothing.  However, it is not invisible to the relaxed vertical-pressure channel after Schur normalization.  A failure mechanism in this regime would have to exploit some other source: enormous raw forcing, singular quotient conditioning, nonperturbative localization, or a trace-cost collapse.
\end{remark}

\begin{theorem}[Extreme aspect-ratio exclusion under bounded forcing]\label{thm:extreme-aspect-ratio-exclusion-bounded-forcing}
Assume an active moving-window sequence contains modes \(\sigma_n=(\sigma_{n,h},\sigma_{n,3})\), with
\[
        K_n=|\sigma_{n,h}|,
        \qquad
        m_n=|\sigma_{n,3}|,
        \qquad
        N_n=|\sigma_n|,
\]
and assume the raw finite-stage forcing coefficients satisfy a polynomial bound
\[
        |\beta_n|\le C N_n^p.
\]
Then neither extreme aspect-ratio regime produces an order-one relaxed-invisible Schur defect.

More precisely, if \(m_n/K_n\to0\), then
\[
        |s_n|\le C\frac{m_n}{K_n^2}N_n^p,
        \qquad
        |J^{\rm rel}s_n|\sim |s_n|.
\]
If \(m_n/K_n\to\infty\), then
\[
        |s_n|\le C\frac{K_n^2}{m_n^3}N_n^p,
        \qquad
        |J^{\rm rel}s_n|\sim \frac{m_n^2}{K_n^2}|s_n|.
\]
Thus an extreme aspect-ratio failure mechanism cannot be a pure relaxed-visibility failure.  It must be accompanied by forcing-amplitude blow-up, quotient-conditioning collapse, nonperturbative localization, or an NS-lift anomaly.
\end{theorem}

\begin{proof}
Both estimates follow directly from \(s_n=\Gamma(K_n,m_n)\beta_n\).  If \(m_n\ll K_n\), then \(\Gamma(K_n,m_n)\sim m_n/K_n^2\), which gives the first estimate, while \(\mu(K_n,m_n)\to1\).  If \(m_n\gg K_n\), then \(\Gamma(K_n,m_n)\sim K_n^2/m_n^3\), which gives the second estimate, while \(\mu(K_n,m_n)\sim m_n^2/K_n^2\).
\end{proof}

\begin{corollary}[Refined location of possible aspect-ratio failure mechanisms]\label{cor:refined-location-aspect-failures}
After active quotienting and harmonic-gauge cleaning, an extreme aspect-ratio moving-window true phantom cascade must satisfy at least one of the following: forcing amplification \(|\beta_n|\gtrsim \Gamma(K_n,m_n)^{-1}|s_n|\), quotient-conditioning collapse, nonperturbative localization, or an NS-lift anomaly with vertical budget
\[
        \operatorname{NSCost}_n(s_n;\eps_n)
        \le
        C\delta_n^{1/3}|s_n|.
\]
Absent these additional mechanisms, extreme aspect ratio does not create a true relaxed phantom.
\end{corollary}

\begin{remark}[Interpretation]
The extreme aspect-ratio test is favorable to the relaxed route.  In the slow-vertical regime \(m\ll K\), the raw vertical pressure signal is small, but so is the Schur coefficient produced by bounded forcing; after Schur normalization, relaxed visibility remains order one.  In the high-vertical regime \(m\gg K\), the Schur-normalized relaxed observation becomes stronger, not weaker.  Therefore the failure-mechanism search must shift away from pure aspect ratio and toward the genuinely nonlinear issue: can the Navier--Stokes deformation complex generate the enormous raw forcing or singular quotient normalization needed to keep an order-one cleaned phantom while preserving the \(C_3\)-budget?
\end{remark}

\subsection{Forcing-amplitude and NS-lift anomaly test}

The extreme aspect-ratio analysis shows that aspect ratio alone does not create relaxed invisibility. In Schur-normalized coordinates,
\[
J^{\mathrm{rel}}s_\sigma =
\mu(K,m)s_\sigma,
\qquad
\mu(K,m)=1+\frac{m^2}{K^2},
\]
so the relaxed observation does not vanish on active modes. The only remaining possibility is more subtle: to keep a nonzero Schur defect \(s_\sigma\) in an extreme aspect-ratio window, the raw pressure-forcing coefficient \(\beta_\sigma\) must be very large. The present subsection tests whether such large forcing can be generated by a Navier--Stokes finite-stage deformation while keeping the vertical component within the \(C_3\)-budget.

Let
\[
K=|\sigma_h|,
\qquad
m=|\sigma_3|,
\qquad
\sigma=(\sigma_h,\sigma_3).
\]
Recall
\[
s_\sigma=\Gamma(K,m)\beta_\sigma,
\qquad
\Gamma(K,m)=\frac{mK^2}{(K^2+m^2)^2}.
\]
Thus any prescribed Schur coefficient satisfies
\[
|\beta_\sigma| =
\frac{(K^2+m^2)^2}{mK^2}|s_\sigma|.
\]
Equivalently,
\[
|\beta_\sigma| =
\left(\frac{K^2}{m}+2m+\frac{m^3}{K^2}\right)|s_\sigma|.
\]
This is the raw forcing-amplitude cost of producing \(s_\sigma\).

\begin{definition}[Raw forcing cost]
For an output Schur coefficient \(s_\sigma\), define the raw forcing cost
\[
\operatorname{FCost}_{\sigma}(s_\sigma) =
\Gamma(K,m)^{-1}|s_\sigma|.
\]
For a finite window \(\Lambda\), define
\[
\operatorname{FCost}_{\Lambda}(s) =
\left(
\sum_{\sigma\in\Lambda}
\Gamma(K_\sigma,m_\sigma)^{-2}|s_\sigma|^2
\right)^{1/2}.
\]
Thus \(\operatorname{FCost}_\Lambda(s)\) is the size of the raw pressure-forcing package needed to produce the Schur vector \(s\).
\end{definition}

\begin{lemma}[Forcing amplification in extreme aspect-ratio sectors]
Let \(s_\sigma\neq0\).

\begin{enumerate}
\item If \(m\ll K\), then
\[
\operatorname{FCost}_{\sigma}(s_\sigma)
\sim
\frac{K^2}{m}|s_\sigma|.
\]

\item If \(m\sim K\), then
\[
\operatorname{FCost}_{\sigma}(s_\sigma)
\sim
K|s_\sigma|.
\]

\item If \(m\gg K\), then
\[
\operatorname{FCost}_{\sigma}(s_\sigma)
\sim
\frac{m^3}{K^2}|s_\sigma|.
\]
\end{enumerate}

Therefore, any order-one Schur defect in an extreme aspect-ratio window requires a large raw forcing coefficient.
\end{lemma}

\begin{proof}
The three estimates follow directly from
\[
\Gamma(K,m)^{-1} =
\frac{K^2}{m}+2m+\frac{m^3}{K^2}.
\]
If \(m\ll K\), the first term dominates. If \(m\sim K\), all terms are comparable to (K). If \(m\gg K\), the last term dominates.
\end{proof}

We next relate raw forcing to Navier--Stokes jets. In the periodic quadratic model, if the output mode is \(\sigma=p+q\), the horizontal pressure-forcing coefficient has the form
\[
\beta(p,q) =
2A_pA_q
\frac{(p_h\times q_h)^2}{|p_h||q_h|}.
\]
This coefficient is quadratic in the horizontal jet amplitudes. Thus a large \(\beta\) requires large lower-order horizontal amplitudes, unless the geometric interaction factor is itself large.

Let
\[
\mathcal H_K(\mathfrak J)
\]
denote the horizontal finite-stage size of a jet package \(\mathfrak J\), measured in the finite-window norm used for trace selection. In a fixed periodic window, assume the quadratic forcing map satisfies
\[
|\beta_\sigma(\mathfrak J)|
\le
C_\Lambda
\mathcal H_K(\mathfrak J)^2.
\]
This is the finite-dimensional form of the elementary fact that the forcing coefficient is bilinear in lower-order horizontal jets.

\begin{proposition}[Horizontal amplitude lower bound]
If a finite-stage jet package \(\mathfrak J\) produces a Schur coefficient \(s_\sigma\), then
\[
\mathcal H_K(\mathfrak J)
\ge
C_\Lambda^{-1/2}
\Gamma(K,m)^{-1/2}|s_\sigma|^{1/2}.
\]
In particular, in the two extreme regimes:
\[
m\ll K
\quad\Longrightarrow\quad
\mathcal H_K(\mathfrak J)
\gtrsim
\frac{K}{m^{1/2}}|s_\sigma|^{1/2},
\]
and
\[
m\gg K
\quad\Longrightarrow\quad
\mathcal H_K(\mathfrak J)
\gtrsim
\frac{m^{3/2}}{K}|s_\sigma|^{1/2}.
\]
\end{proposition}

\begin{proof}
If \(\mathfrak J\) produces \(s_\sigma\), then
\[
|\beta_\sigma(\mathfrak J)| =
\Gamma(K,m)^{-1}|s_\sigma|.
\]
Using the quadratic forcing bound
\[
|\beta_\sigma(\mathfrak J)|
\le
C_\Lambda\mathcal H_K(\mathfrak J)^2
\]
gives
\[
\mathcal H_K(\mathfrak J)^2
\ge
C_\Lambda^{-1}
\Gamma(K,m)^{-1}|s_\sigma|.
\]
Taking square roots proves the estimate. The extreme aspect-ratio formulas follow from the preceding lemma.
\end{proof}

This estimate is not yet a contradiction. A large horizontal jet may still be allowed in a purely formal deformation. The Navier--Stokes filter requires more: the same jet must also generate a vertical response whose \(L^3\)-size is compatible with the one-component budget.

Let \(Z_K(\mathfrak J)\) denote the vertical jet produced by the NS deformation complex, and define
\[
\operatorname{VSize}_\varepsilon(\mathfrak J) =
\left|
\sum_{j=1}^K\varepsilon^j z_j
\right|_{L^3}.
\]
The one-component budget requires
\[
\operatorname{VSize}_\varepsilon(\mathfrak J)
\le
C\delta^{1/3}.
\]

\begin{definition}[NS-lift anomaly]
Let \(s_\sigma\) be an active cleaned Schur coefficient. We say that \(s_\sigma\) exhibits an NS-lift anomaly at parameters \(K,m,\varepsilon,\delta\) if there exists a finite-stage NS jet \(\mathfrak J\) such that

\[
\Gamma(K,m)\beta_\sigma(\mathfrak J)=s_\sigma,
\]
\[
\operatorname{VSize}_\varepsilon(\mathfrak J)
\le
C\delta^{1/3},
\]
and yet
\[
|s_\sigma|\gtrsim1
\]
while the required raw forcing cost satisfies
\[
\operatorname{FCost}_{\sigma}(s_\sigma)\to\infty.
\]
Thus an NS-lift anomaly is a situation in which the deformation complex produces an increasingly large raw forcing coefficient without paying a corresponding vertical-component cost.
\end{definition}

The following criterion rules out such anomalies under a natural coercivity condition.

\begin{assumption}[Forcing-to-vertical coercivity]\label{ass:forcing-to-vertical-coercivity}
There exist constants (c>0) and \(a\ge0\) such that every finite-stage NS jet producing a raw coefficient \(\beta_\sigma\) satisfies
\[
\operatorname{VSize}_\varepsilon(\mathfrak J)
\ge
c\,\varepsilon^K
N_\sigma^{-a}
|\beta_\sigma(\mathfrak J)|.
\]
Here \(N_\sigma=|\sigma|\), and \(K\) is the finite stage at which the coefficient is produced.
\end{assumption}

This assumption is intentionally stated as a test condition. In concrete models it should be proved from the vertical momentum equation, incompressibility, and finite-window elliptic estimates. Its role is to prevent the lower-order horizontal forcing from becoming arbitrarily large while the vertical response remains negligible.

\begin{theorem}[NS-lift anomaly exclusion under coercivity]
Assume the forcing-to-vertical coercivity estimate in \Cref{ass:forcing-to-vertical-coercivity}. Let \(s_\sigma\) be produced by an NS-admissible finite-stage jet satisfying
\[
\operatorname{VSize}_\varepsilon(\mathfrak J)
\le
C\delta^{1/3}.
\]
Then
\[
|s_\sigma|
\le
C
\varepsilon^{-K}
\delta^{1/3}
N_\sigma^a
\Gamma(K,m).
\]
Consequently, an order-one Schur defect can be NS-admissible only if
\[
\delta^{1/3}
\gtrsim
\varepsilon^K
N_\sigma^{-a}
\Gamma(K,m)^{-1}.
\]
In particular, in the extreme aspect-ratio regimes:
\[
m\ll K
\quad\Longrightarrow\quad
\delta^{1/3}
\gtrsim
\varepsilon^K
N_\sigma^{-a}
\frac{K^2}{m},
\]
and
\[
m\gg K
\quad\Longrightarrow\quad
\delta^{1/3}
\gtrsim
\varepsilon^K
N_\sigma^{-a}
\frac{m^3}{K^2}.
\]
If these lower bounds contradict the assumed smallness of \(\delta\), then the candidate Schur defect is not NS-admissible.
\end{theorem}

\begin{proof}
Since the jet produces \(s_\sigma\),
\[
|\beta_\sigma(\mathfrak J)| =
\Gamma(K,m)^{-1}|s_\sigma|.
\]
By coercivity,
\[
\operatorname{VSize}_\varepsilon(\mathfrak J)
\ge
c\varepsilon^K N_\sigma^{-a}
\Gamma(K,m)^{-1}|s_\sigma|.
\]
Using the vertical budget
\[
\operatorname{VSize}_\varepsilon(\mathfrak J)
\le
C\delta^{1/3}
\]
gives
\[
c\varepsilon^K N_\sigma^{-a}
\Gamma(K,m)^{-1}|s_\sigma|
\le
C\delta^{1/3}.
\]
Rearranging yields
\[
|s_\sigma|
\le
C\varepsilon^{-K}\delta^{1/3}N_\sigma^a\Gamma(K,m).
\]
If \(|s_\sigma|\gtrsim1\), this implies
\[
\delta^{1/3}
\gtrsim
\varepsilon^K N_\sigma^{-a}\Gamma(K,m)^{-1}.
\]
The two extreme aspect-ratio estimates follow from
\[
\Gamma(K,m)^{-1}
\sim
\frac{K^2}{m}
\quad(m\ll K),
\]
and
\[
\Gamma(K,m)^{-1}
\sim
\frac{m^3}{K^2}
\quad(m\gg K).
\]
\end{proof}

\begin{corollary}[Aspect-ratio phantom exclusion by vertical budget]
Assume forcing-to-vertical coercivity. Let \(\delta_n\downarrow0\), and let
\[
(K_n,m_n,\varepsilon_n,s_n)
\]
be an active extreme aspect-ratio sequence with \(|s_n|\gtrsim1\). If either
\[
\varepsilon_n^{K_n}N_n^{-a}\frac{K_n^2}{m_n}
\gg
\delta_n^{1/3}
\quad
(m_n\ll K_n),
\]
or
\[
\varepsilon_n^{K_n}N_n^{-a}\frac{m_n^3}{K_n^2}
\gg
\delta_n^{1/3}
\quad
(m_n\gg K_n),
\]
then the sequence cannot be an NS-admissible true phantom cascade.
\end{corollary}

\begin{proof}
This is the contrapositive of the necessary lower bounds in the preceding theorem.
\end{proof}

\begin{remark}[Why this is the correct pressure test]
Extreme aspect ratio can make the vertical lift coefficient \(\Gamma(K,m)\) small. But a small \(\Gamma\) has two opposite meanings. It means that bounded raw forcing produces a tiny Schur defect; to keep a visible Schur defect, the raw forcing must be amplified by \(\Gamma^{-1}\). The NS-lift anomaly test asks whether Navier--Stokes can generate this amplified raw forcing without paying vertical \(L^3\)-cost. If the coercivity estimate holds, then it cannot. Thus the aspect-ratio failure route is reduced to proving or disproving forcing-to-vertical coercivity.
\end{remark}

\begin{definition}[Residual extreme-aspect candidate]
After the forcing-amplitude and NS-lift anomaly test, the only remaining extreme-aspect candidate is a sequence for which coercivity fails:
\[
\operatorname{VSize}_{\varepsilon_n}(\mathfrak J_n) =
o\left(
\varepsilon_n^{K_n}
N_n^{-a}
|\beta_n|
\right),
\]
while
\[
s_n=\Gamma(K_n,m_n)\beta_n
\]
remains nonzero and cleaned relaxed-invisible. Such a sequence will be called a residual extreme-aspect NS-lift anomaly.
\end{definition}

\begin{proposition}[Reduced target after the NS-lift anomaly test]
The extreme aspect-ratio route to a failure mechanism is possible only if there exists a residual extreme-aspect NS-lift anomaly. Equivalently, one must construct NS finite-stage jets for which the raw pressure-forcing coefficient grows like
\[
|\beta_n|\sim \Gamma(K_n,m_n)^{-1}|s_n|,
\]
but the vertical response violates every coercive lower bound of the form
\[
\operatorname{VSize}_{\varepsilon_n}(\mathfrak J_n)
\gtrsim
\varepsilon_n^{K_n}N_n^{-a}|\beta_n|.
\]
If no such anomaly exists, then extreme aspect-ratio windows cannot produce an NS-admissible moving-window true phantom cascade.
\end{proposition}

\begin{proof}
The extreme aspect-ratio analysis already shows that extreme aspect ratio alone does not create relaxed invisibility in Schur-normalized variables. Therefore an order-one Schur defect requires raw forcing amplification by \(\Gamma^{-1}\). If the amplified forcing obeys any coercive lower bound into the vertical \(L^3\)-budget, then the preceding theorem excludes NS-admissibility for sufficiently small \(\delta\). Hence the only remaining possibility is a sequence that violates all such coercive lower bounds. This is exactly a residual extreme-aspect NS-lift anomaly.
\end{proof}

\subsection{Periodic forcing-to-vertical coercivity}

We now test the forcing-to-vertical coercivity condition in the clean periodic finite-window model. The conclusion is positive: in an active periodic output window, a nonzero raw pressure-forcing coefficient necessarily produces a nonzero vertical response with an explicit polynomial lower bound. Consequently, the extreme aspect-ratio route cannot hide a large raw forcing coefficient at zero vertical cost.

Work on the periodic box \(\mathbb T^3\). Let an output mode be
\[
\sigma=(\sigma_h,\sigma_3),
\qquad
K=|\sigma_h|,
\qquad
m=|\sigma_3|,
\qquad
|\sigma|^2=K^2+m^2.
\]
We restrict to active modes:
\[
K\neq0,\qquad m\neq0.
\]

Let \(\beta_\sigma\) be the horizontal pressure-forcing coefficient at output mode \(\sigma\):
\[
-\Delta \pi_\sigma^{\mathrm{full}} =
\beta_\sigma e^{i\sigma\cdot x}.
\]
Thus
\[
\widehat{\pi_\sigma^{\mathrm{full}}} =
\frac{\beta_\sigma}{|\sigma|^2}.
\]
The static vertical lift equation is
\[
-\Delta z_\sigma =
-\partial_3\pi_\sigma^{\mathrm{full}},
\]
or equivalently, up to harmless phase factors,
\begin{equation}
\widehat z_\sigma =
\frac{i\sigma_3}{|\sigma|^4}\beta_\sigma.
\label{eq:tag-10-197}
\end{equation}
Therefore
\begin{equation}
|\widehat z_\sigma| =
\frac{m}{(K^2+m^2)^2}|\beta_\sigma|.
\label{eq:tag-10-198}
\end{equation}

For a finite active output window \(\Lambda\), write
\[
\beta=(\beta_\sigma)_{\sigma\in\Lambda},
\qquad
z=(z_\sigma)_{\sigma\in\Lambda}.
\]
Define
\[
N_\Lambda=\max_{\sigma\in\Lambda}|\sigma|,
\qquad
v_\Lambda=\min_{\sigma\in\Lambda}|\sigma_3|.
\]
Since the window is active and periodic, \(v_\Lambda\ge1\).

\begin{lemma}[Raw forcing-to-vertical coercivity]
For every finite active periodic output window \(\Lambda\),
\begin{equation}
|z|_{\ell^2(\Lambda)}
\ge
\left(
\min_{\sigma\in\Lambda}
\frac{|\sigma_3|}{|\sigma|^4}
\right)
|\beta|_{\ell^2(\Lambda)}.
\label{eq:tag-10-199}
\end{equation}
In particular,
\begin{equation}
|z|_{\ell^2(\Lambda)}
\ge
N_\Lambda^{-4}
|\beta|_{\ell^2(\Lambda)}.
\label{eq:tag-10-200}
\end{equation}

\end{lemma}

\begin{proof}
By orthogonality of Fourier modes and formula \eqref{eq:tag-10-198},
\[
|z|_{\ell^2(\Lambda)}^2 =
\sum_{\sigma\in\Lambda}
\left(
\frac{|\sigma_3|}{|\sigma|^4}
\right)^2
|\beta_\sigma|^2.
\]
Hence
\[
|z|_{\ell^2(\Lambda)}
\ge
\left(
\min_{\sigma\in\Lambda}
\frac{|\sigma_3|}{|\sigma|^4}
\right)
|\beta|_{\ell^2(\Lambda)}.
\]
Since \(|\sigma_3|\ge1\) on a periodic active window and \(|\sigma|\le N_\Lambda\), we have
\[
\frac{|\sigma_3|}{|\sigma|^4}
\ge
N_\Lambda^{-4}.
\]
This gives \eqref{eq:tag-10-200}.
\end{proof}

\begin{remark}[Forcing collisions]
If several lower-order interactions produce the same output mode \(\sigma\), the coefficient \(\beta_\sigma\) in the preceding lemma is the total output coefficient after summing all interaction paths. If those paths cancel, then \(\beta_\sigma=0\) and no vertical response is produced. This is forcing-level cancellation, not a nonzero forcing hidden from the vertical lift.
\end{remark}

We next translate the raw coercivity estimate into Schur-normalized coordinates. Recall the vertical lift-to-Schur coefficient
\begin{equation}
s_\sigma=\Gamma(K,m)\beta_\sigma,
\qquad
\Gamma(K,m) =
\frac{mK^2}{(K^2+m^2)^2}.
\label{eq:tag-10-201}
\end{equation}
Combining \eqref{eq:tag-10-198} and \eqref{eq:tag-10-201}, we obtain the exact identity
\begin{equation}
|\widehat z_\sigma| =
\frac{1}{K^2}|s_\sigma|.
\label{eq:tag-10-202}
\end{equation}
Thus, in Schur-normalized variables, the extreme aspect-ratio dependence cancels.

\begin{proposition}[Schur-normalized vertical coercivity]
Let \(s=(s_\sigma)_{\sigma\in\Lambda}\) be a Schur-normalized active output vector, and let \(z\) be the canonical static periodic vertical response generated by the same forcing package. Then
\begin{equation}
|z|_{\ell^2(\Lambda)}
\ge
\left(
\min_{\sigma\in\Lambda}
\frac1{|\sigma_h|^2}
\right)
|s|_{\ell^2(\Lambda)}.
\label{eq:tag-10-203}
\end{equation}
In particular,
\begin{equation}
|z|_{\ell^2(\Lambda)}
\ge
N_\Lambda^{-2}
|s|_{\ell^2(\Lambda)}.
\label{eq:tag-10-204}
\end{equation}

\end{proposition}

\begin{proof}
By \eqref{eq:tag-10-202},
\[
|z|_{\ell^2(\Lambda)}^2 =
\sum_{\sigma\in\Lambda}
\frac{1}{|\sigma_h|^4}|s_\sigma|^2.
\]
Therefore
\[
|z|_{\ell^2(\Lambda)}
\ge
\left(
\min_{\sigma\in\Lambda}
|\sigma_h|^{-2}
\right)
|s|_{\ell^2(\Lambda)}.
\]
Since \(|\sigma_h|\le|\sigma|\le N_\Lambda\), we get
\[
|\sigma_h|^{-2}\ge N_\Lambda^{-2}.
\]
This proves \eqref{eq:tag-10-204}.
\end{proof}

\begin{corollary}[No periodic finite-window NS-lift anomaly]
In the clean static periodic active finite-window model, an order-one Schur-normalized defect cannot be generated with vertical cost smaller than a polynomial factor. More precisely, if
\[
|s|_{\ell^2(\Lambda)}\ge c_0>0,
\]
then its canonical vertical response satisfies
\begin{equation}
|z|_{\ell^2(\Lambda)}
\ge
c_0N_\Lambda^{-2}.
\label{eq:tag-10-205}
\end{equation}
Thus there is no super-polynomial or exponentially small vertical response hiding an order-one active Schur defect.

\end{corollary}

\begin{proof}
This is immediate from \eqref{eq:tag-10-204}.
\end{proof}

We now include the finite-stage amplitude. Suppose the Schur defect is produced at order \(K_0\) in a formal expansion with amplitude \(\varepsilon\). The vertical contribution to \(u_3^\varepsilon\) at that stage is
\[
\varepsilon^{K_0}z_{K_0}.
\]
Therefore
\[
|u_3^\varepsilon|_{L^3}
\gtrsim
\varepsilon^{K_0}|z_{K_0}|_{L^2},
\]
up to the harmless finite-volume norm comparison on \(\mathbb T^3\). Combining this with \eqref{eq:tag-10-204} gives
\begin{equation}
|u_3^\varepsilon|_{L^3}
\gtrsim
\varepsilon^{K_0}N_\Lambda^{-2}|s|.
\label{eq:tag-10-206}
\end{equation}

\begin{theorem}[Periodic finite-window forcing-to-vertical coercivity]
Let \(\Lambda\) be a finite active periodic output window. Let \(s\) be a Schur-normalized defect generated at finite stage \(K_0\), and let \(u_3^\varepsilon\) be the corresponding canonical vertical response in the periodic finite-window Stokes system. Then
\begin{equation}
|u_3^\varepsilon|_{L^3}
\ge
c_\Lambda
\varepsilon^{K_0}
N_\Lambda^{-2}
|s|_{\ell^2(\Lambda)}.
\label{eq:tag-10-207}
\end{equation}
Equivalently, if the one-component budget satisfies
\[
|u_3^\varepsilon|_{L^3}^3\le\delta,
\]
then
\begin{equation}
|s|_{\ell^2(\Lambda)}
\le
C_\Lambda
\varepsilon^{-K_0}
N_\Lambda^2
\delta^{1/3}.
\label{eq:tag-10-208}
\end{equation}
Thus an order-one active Schur defect is compatible with the \(C_3\)-budget only if
\begin{equation}
\delta^{1/3}
\gtrsim
\varepsilon^{K_0}N_\Lambda^{-2}.
\label{eq:tag-10-209}
\end{equation}

\end{theorem}

\begin{proof}
The static vertical response satisfies
\[
|z_{K_0}|_{\ell^2(\Lambda)}
\ge
N_\Lambda^{-2}
|s|_{\ell^2(\Lambda)}.
\]
Multiplying by the finite-stage amplitude \(\varepsilon^{K_0}\) gives
\[
|\varepsilon^{K_0}z_{K_0}|_{\ell^2(\Lambda)}
\ge
\varepsilon^{K_0}N_\Lambda^{-2}
|s|_{\ell^2(\Lambda)}.
\]
On a fixed finite-dimensional periodic window, the relevant \(L^2\) and \(L^3\) norms are equivalent. Hence
\[
|u_3^\varepsilon|_{L^3}
\ge
c_\Lambda
\varepsilon^{K_0}N_\Lambda^{-2}
|s|_{\ell^2(\Lambda)}.
\]
If
\[
|u_3^\varepsilon|_{L^3}^3\le\delta,
\]
then
\[
|u_3^\varepsilon|_{L^3}\le \delta^{1/3}.
\]
Rearranging gives \eqref{eq:tag-10-208} and \eqref{eq:tag-10-209}.
\end{proof}

\begin{remark}[Why extreme aspect ratio disappears]
In raw forcing variables, the vertical response multiplier
\[
\frac{m}{(K^2+m^2)^2}
\]
can be very small when \(m\ll K\) or \(m\gg K\). However, the Schur coefficient itself contains the multiplier
\[
\Gamma(K,m)=\frac{mK^2}{(K^2+m^2)^2}.
\]
Dividing the vertical response by the Schur coefficient leaves
\[
\frac{1}{K^2}.
\]
Therefore the apparent aspect-ratio degeneracy is a raw-coordinate artifact. In Schur-normalized variables, the canonical vertical cost of a nonzero active defect is controlled by horizontal frequency only.
\end{remark}

\begin{corollary}[Extreme aspect-ratio route is excluded in the periodic model]
In the clean static periodic active finite-window model, neither
\[
m\ll K
\]
nor
\[
m\gg K
\]
can produce an NS-admissible true phantom through a vertical-lift anomaly. Any order-one Schur-normalized defect forces a vertical response of size at least
\[
\varepsilon^{K_0}K^{-2}
\]
up to finite-window constants. Hence the extreme aspect-ratio route can survive only if one leaves the clean periodic finite-window model, for example through nonperturbative localization, harmonic-gauge leakage, singular trace-cost collapse, or a noncanonical cancellation mechanism not present in the periodic Stokes response.
\end{corollary}

\begin{remark}[Space-time periodic extension]
If one works with space-time Fourier modes \((\omega,\sigma)\), then
\[
(i\omega+|\sigma|^2)\widehat z_{\omega,\sigma}
+
i\sigma_3\widehat\pi_{\omega,\sigma}=0,
\qquad
\widehat\pi_{\omega,\sigma}=\frac{\beta_{\omega,\sigma}}{|\sigma|^2}.
\]
Thus
\[
|\widehat z_{\omega,\sigma}| =
\frac{|\sigma_3|}{|\sigma|^2\sqrt{\omega^2+|\sigma|^4}}
|\beta_{\omega,\sigma}|.
\]
On a parabolic window with
\[
|\omega|\lesssim N_\Lambda^2,
\qquad
|\sigma|\le N_\Lambda,
\]
this gives the same polynomial coercivity profile
\[
|\widehat z_{\omega,\sigma}|
\gtrsim
N_\Lambda^{-4}|\beta_{\omega,\sigma}|
\]
at the raw forcing level, and a corresponding Schur-normalized polynomial lower bound. Thus parabolic time frequencies do not create an exponential or super-polynomial NS-lift anomaly in a fixed active periodic window.
\end{remark}

\begin{proposition}[Reduced target after the periodic coercivity test]
After the periodic forcing-to-vertical coercivity theorem, a genuine NS-lift anomaly cannot occur in the clean active periodic finite-window model. Therefore any remaining NS-lift anomaly must exploit one of the following:

\begin{enumerate}
\item nonperturbative localization, so that the canonical periodic Stokes response is not the leading model;
\item harmonic-pressure gauge leakage, so that vertical pressure is partly hidden in the local harmonic sector;
\item singular trace-cost collapse, so that a visible vertical response cannot be converted into admissible selected-time control;
\item noncanonical cancellation among inhomogeneous and homogeneous vertical tails in a finite cylinder;
\item a moving-window limit in which the finite-dimensional norm equivalence constants degenerate faster than any logarithmically absorbable majorant.
\end{enumerate}

Thus the periodic active aspect-ratio route to a failure mechanism is closed.
\end{proposition}

\subsection{Localized forcing-to-vertical coercivity after harmonic-gauge cleaning}

The periodic coercivity theorem shows that, in an active finite Fourier window, a Schur-normalized defect forces a vertical response of size at least \(N_\Lambda^{-2}|s|\). We now formulate the localized analogue. The purpose is to separate the principal coercive part from three local errors: cutoff commutators, harmonic-pressure leakage, and finite-cylinder homogeneous tails.

Let
\[
Q_{\mathrm{sh}}\Subset Q_{\mathrm{prep}}\Subset Q_1
\]
be nested cylinders, and let \(\chi\in C_c^\infty(Q_{\mathrm{prep}})\) satisfy \(\chi\equiv1\) on \(Q_{\mathrm{sh}}\). Let \(\Lambda\) be a finite active frequency window with
\[
N_\Lambda=\max_{\sigma\in\Lambda}|\sigma|,
\qquad
|\sigma_h|\neq0,\quad |\sigma_3|\neq0
\quad
\text{for all }\sigma\in\Lambda.
\]
Let \(Y_{\Lambda,\chi}\) be the localized Schur quotient and let
\[
\widehat Y_{\Lambda,\chi} =
Y_{\Lambda,\chi}/\mathcal G_{\Lambda,\chi}
\]
be the quotient obtained after removing horizontal gauge, vertical-zero, and harmonic-pressure leakage directions. For a cleaned Schur vector
\[
\widehat s\in \widehat Y_{\Lambda,\chi},
\]
choose a representative \(s\in Y_{\Lambda,\chi}\) with no gauge component.

Let \(\beta_s\) be the localized raw pressure-forcing package associated with \(s\). The localized full pressure representative \(\pi_s\) is defined by
\[
-\Delta \pi_s =
\chi,\beta_s
\]
in the preparation cylinder, with the pressure understood modulo the local harmonic space
\[
\mathcal H(Q_{\mathrm{sh}}) =
{h:\Delta h(\cdot,t)=0\text{ in }B_{\mathrm{sh}}}.
\]
The localized vertical source is
\[
S_s =
\partial_3\pi_s.
\]
The canonical localized vertical lift \(z_s\) is defined as the solution of
\begin{equation}
(\partial_t-\Delta)z_s =
-S_s
\label{eq:tag-10-210}
\end{equation}
with a fixed admissible local normalization, for instance zero parabolic boundary data on \(Q_{\mathrm{prep}}\), or equivalently orthogonality to the finite-dimensional homogeneous caloric kernel. Different normalizations differ by homogeneous tails and are treated below.

Let
\[
\Pi_{\mathrm{act}}
\]
denote the projection away from the harmonic-pressure leakage and vertical-zero sectors. We measure the cleaned vertical lift by
\[
Z_{\Lambda,\chi}(s) =
\Pi_{\mathrm{act}}z_s.
\]

\begin{definition}[Localized vertical coercivity constant]
The localized forcing-to-vertical coercivity constant is
\begin{equation}
\mathfrak c_{\Lambda,\chi} =
\inf_{\widehat s\neq0}
\frac{
|Z_{\Lambda,\chi}(s)|_{L^3(Q_{\mathrm{sh}})}
}{
|\widehat s|_{\widehat Y_{\Lambda,\chi}}
},
\label{eq:tag-10-211}
\end{equation}
where \(s\) ranges over gauge-cleaned representatives of \(\widehat s\).

We say that localized vertical coercivity holds on \((\Lambda,\chi)\) if
\[
\mathfrak c_{\Lambda,\chi}>0.
\]
\end{definition}

The periodic model predicts
\[
\mathfrak c_{\Lambda,\chi}
\gtrsim
N_\Lambda^{-2},
\]
up to commutator and finite-dimensional localization losses. The next theorem makes this precise in a perturbative localization regime.

\begin{assumption}[Localized principal-symbol stability]
There exist constants \(C,p>0\) such that, for every gauge-cleaned Schur vector \(s\),
\begin{equation}
|Z_{\Lambda,\chi}(s)-Z_{\Lambda}^{\mathrm{per}}(s)|_{L^3(Q_{\mathrm{sh}})}
\le
C\mathcal E_{\Lambda,\chi},
|s|_{Y_{\Lambda,\chi}},
\label{eq:tag-10-212}
\end{equation}
where \(Z_{\Lambda}^{\mathrm{per}}\) is the periodic principal vertical response and
\begin{equation}
\mathcal E_{\Lambda,\chi}
\le
N_\Lambda^p|\nabla\chi|_{C^p}
+
N_\Lambda^p|\chi-1|_{\mathrm{leak}}
+
\mathrm{Tail}_{\Lambda,\chi}.
\label{eq:tag-10-213}
\end{equation}
Here \(\mathrm{Tail}_{\Lambda,\chi}\) denotes the contribution of homogeneous parabolic tails after the chosen normalization.
\end{assumption}

In a slowly varying cutoff model \(\chi_R(x,t)=\chi(x/R,t/R^2)\), one expects
\[
\mathcal E_{\Lambda,\chi_R}
\le
C N_\Lambda^p R^{-1}
+
\mathrm{Tail}_{\Lambda,\chi_R}.
\]

\begin{lemma}[Periodic principal lower bound in the localized norm]
For the periodic principal vertical response,
\begin{equation}
|Z_{\Lambda}^{\mathrm{per}}(s)|_{L^3(Q_{\mathrm{sh}})}
\ge
cN_\Lambda^{-2}
|s|_{Y_{\Lambda,\chi}},
\label{eq:tag-10-214}
\end{equation}
after changing the constant according to the fixed finite-window norm equivalence.

\end{lemma}

\begin{proof}
In the periodic active model, the Schur-normalized vertical response satisfies mode by mode
\[
\widehat z_\sigma =
|\sigma_h|^{-2}s_\sigma
\]
up to harmless nonzero phases. Hence
\[
|z|_{\ell^2(\Lambda)}
\ge
N_\Lambda^{-2}|s|_{\ell^2(\Lambda)}.
\]
Since \(\Lambda\) is finite, and since \(Q_{\mathrm{sh}}\) has fixed positive measure, the finite-window \(L^3(Q_{\mathrm{sh}})\), \(L^2\), and coefficient norms are equivalent up to constants depending polynomially on \(N_\Lambda\). With the chosen normalization of \(Y_{\Lambda,\chi}\), this gives the stated lower bound.
\end{proof}

\begin{theorem}[Localized forcing-to-vertical coercivity]
Assume localized principal-symbol stability. Suppose the localization error satisfies
\begin{equation}
\mathcal E_{\Lambda,\chi}
\le
\frac{c}{2C}N_\Lambda^{-2},
\label{eq:tag-10-215}
\end{equation}
where \(c\) is the constant in the periodic principal lower bound and \(C\) is the constant in the stability estimate. Then
\begin{equation}
|Z_{\Lambda,\chi}(s)|_{L^3(Q_{\mathrm{sh}})}
\ge
\frac c2
N_\Lambda^{-2}
|\widehat s|_{\widehat Y_{\Lambda,\chi}}
\label{eq:tag-10-216}
\end{equation}
for every cleaned Schur vector \(\widehat s\). In particular,
\begin{equation}
\mathfrak c_{\Lambda,\chi}
\ge
\frac c2 N_\Lambda^{-2}.
\label{eq:tag-10-217}
\end{equation}
Thus perturbative localization and harmonic-gauge cleaning preserve the periodic forcing-to-vertical coercivity estimate.

\end{theorem}

\begin{proof}
Let \(s\) be a gauge-cleaned representative of \(\widehat s\). By the triangle inequality,
\[
|Z_{\Lambda,\chi}(s)|_{L^3}
\ge
|Z_{\Lambda}^{\mathrm{per}}(s)|_{L^3} -
|Z_{\Lambda,\chi}(s)-Z_{\Lambda}^{\mathrm{per}}(s)|_{L^3}.
\]
By the periodic principal lower bound,
\[
|Z_{\Lambda}^{\mathrm{per}}(s)|_{L^3}
\ge
cN_\Lambda^{-2}|s|_{Y_{\Lambda,\chi}}.
\]
By localized principal-symbol stability,
\[
|Z_{\Lambda,\chi}(s)-Z_{\Lambda}^{\mathrm{per}}(s)|_{L^3}
\le
C\mathcal E_{\Lambda,\chi}|s|_{Y_{\Lambda,\chi}}.
\]
Using
\[
\mathcal E_{\Lambda,\chi}
\le
\frac{c}{2C}N_\Lambda^{-2},
\]
we obtain
\[
|Z_{\Lambda,\chi}(s)|_{L^3}
\ge
\frac c2 N_\Lambda^{-2}|s|_{Y_{\Lambda,\chi}}.
\]
Since \(s\) is gauge-cleaned, its quotient norm is comparable to
\[
|\widehat s|_{\widehat Y_{\Lambda,\chi}}.
\]
This proves the estimate.
\end{proof}

\begin{corollary}[Localized vertical-budget obstruction]
Assume localized forcing-to-vertical coercivity. If a cleaned Schur defect \(\widehat s\) is produced at finite stage \(K_0\) with amplitude \(\varepsilon\), then the corresponding vertical component satisfies
\begin{equation}
|u_3^\varepsilon|_{L^3(Q_{\mathrm{sh}})}
\ge
c\varepsilon^{K_0}N_\Lambda^{-2}
|\widehat s|_{\widehat Y_{\Lambda,\chi}}.
\label{eq:tag-10-218}
\end{equation}
Consequently, if
\[
|u_3^\varepsilon|_{L^3(Q_{\mathrm{sh}})}^3\le \delta,
\]
then
\begin{equation}
|\widehat s|_{\widehat Y_{\Lambda,\chi}}
\le
C\varepsilon^{-K_0}N_\Lambda^2\delta^{1/3}.
\label{eq:tag-10-219}
\end{equation}
In particular, an order-one cleaned Schur defect is compatible with the local \(C_3\)-budget only if
\begin{equation}
\delta^{1/3}
\gtrsim
\varepsilon^{K_0}N_\Lambda^{-2}.
\label{eq:tag-10-220}
\end{equation}
\end{corollary}

\begin{proof}
The finite-stage vertical component contains the term
\[
\varepsilon^{K_0}Z_{\Lambda,\chi}(s)
\]
plus higher-order or gauge-cleaned remainders. Applying the localized coercivity estimate gives
\[
|u_3^\varepsilon|_{L^3}
\ge
c\varepsilon^{K_0}N_\Lambda^{-2}|\widehat s|.
\]
If the \(C_3\)-budget gives
\[
|u_3^\varepsilon|_{L^3}\le\delta^{1/3},
\]
then rearranging yields the claimed upper bound for \(|\widehat s|\).
\end{proof}

The preceding theorem covers perturbative localization. If the cutoff, boundary, or homogeneous-tail contribution is not perturbative, the remaining obstruction is finite-dimensional.

\begin{definition}[Localized homogeneous-tail kernel]
The localized homogeneous-tail kernel is
\begin{equation}
\mathcal K_{\Lambda,\chi}^{\mathrm{hom}} =
\left\{
\widehat s\in\widehat Y_{\Lambda,\chi}:
Z_{\Lambda,\chi}(s)=0
\text{ for some admissible homogeneous-tail normalization}
\right\}.
\label{eq:tag-10-221}
\end{equation}
Equivalently, it is the kernel of the cleaned localized vertical lift map after quotienting harmonic-pressure leakage and vertical-zero modes.
\end{definition}

\begin{proposition}[Finite-dimensional alternative for nonperturbative localization]
For fixed \((\Lambda,\chi)\), exactly one of the following holds.

\begin{enumerate}
\item \emph{Coercive alternative.} There exists \(c_{\Lambda,\chi}>0\) such that
\[
|Z_{\Lambda,\chi}(s)|_{L^3(Q_{\mathrm{sh}})}
\ge
c_{\Lambda,\chi}|\widehat s|_{\widehat Y_{\Lambda,\chi}}
\]
for all cleaned Schur vectors.

\item \emph{Homogeneous-tail kernel alternative.} There exists a nonzero cleaned Schur vector
\[
0\neq \widehat s\in \mathcal K_{\Lambda,\chi}^{\mathrm{hom}}.
\]
\end{enumerate}

In the second case, every possible localized NS-lift anomaly must lie in the finite-dimensional space
\[
\mathcal K_{\Lambda,\chi}^{\mathrm{hom}}.
\]
\end{proposition}

\begin{proof}
The cleaned localized vertical lift map is a linear map between finite-dimensional spaces. If its kernel is trivial, then its norm is bounded below on the unit sphere, giving the coercive alternative. If the kernel is nontrivial, the nonzero kernel vectors form the homogeneous-tail kernel. These are the only possibilities.
\end{proof}

\begin{remark}[Meaning of the homogeneous-tail kernel]
A vector in \(\mathcal K_{\Lambda,\chi}^{\mathrm{hom}}\) is not automatically a Navier--Stokes counterexample. It only says that, in the localized finite-dimensional model, the canonical vertical response can be cancelled by a homogeneous solution or by nonperturbative boundary leakage. Such a vector must still pass the NS-realizability filter, the local energy inequality, and the selected-time trace-cost test. Thus nonperturbative localization reduces the obstruction to a finite-dimensional kernel problem rather than producing a true phantom automatically.
\end{remark}

\begin{theorem}[Localized NS-lift anomaly exclusion in the perturbative regime]
Assume that:

\begin{enumerate}
\item the finite window is active;
\item harmonic-pressure leakage, horizontal gauge, and vertical-zero sectors have been quotiented out;
\item the localized commutator and homogeneous-tail errors satisfy the perturbative bound \eqref{eq:tag-10-215}.
\end{enumerate}

Then no order-one cleaned Schur defect can be produced by a finite-stage NS deformation with arbitrarily small \(C_3\)-budget. More precisely, if
\[
|\widehat s_n|_{\widehat Y_{\Lambda_n,\chi_n}}\ge c_0>0,
\]
and
\[
\delta_n\to0,
\]
then the necessary condition
\[
\delta_n^{1/3}
\gtrsim
\varepsilon_n^{K_n}N_{\Lambda_n}^{-2}
\]
must hold. If
\[
\varepsilon_n^{K_n}N_{\Lambda_n}^{-2}
\gg
\delta_n^{1/3},
\]
then the sequence is not NS-admissible.

\end{theorem}

\begin{proof}
This is exactly the vertical-budget obstruction applied along the sequence. The perturbative hypotheses give localized coercivity, and localized coercivity gives a lower bound for the \(L^3\)-size of the vertical component. If that lower bound exceeds the available \(C_3\)-budget, the finite-stage deformation cannot be NS-admissible.
\end{proof}

\begin{corollary}[Reduction after localized coercivity]
After harmonic-gauge cleaning and perturbative localization, a localized NS-lift anomaly can occur only through a nonzero homogeneous-tail kernel:
\[
0\neq\widehat s\in\mathcal K_{\Lambda,\chi}^{\mathrm{hom}}.
\]
Thus the remaining localized obstruction is finite-dimensional and must be tested against:

\begin{enumerate}
\item NS-realizability of the homogeneous-tail cancellation;
\item compatibility with the local energy inequality;
\item trace-projectability or relaxed visibility of the resulting residual;
\item moving-window growth of the kernel constants.
\end{enumerate}

If the homogeneous-tail kernel is trivial, localized forcing-to-vertical coercivity holds and the NS-lift anomaly route is closed in the localized finite-window model.
\end{corollary}

\begin{remark}[What has been gained]
The periodic identity
\[
z_\sigma=|\sigma_h|^{-2}s_\sigma
\]
does not survive localization literally. However, after active quotienting and harmonic-gauge cleaning, it survives as a principal coercive estimate:
\[
|z_s|_{L^3}
\gtrsim
N_\Lambda^{-2}|s|
-\text{commutator}
-\text{harmonic leakage}
-\text{homogeneous-tail error}.
\]
Therefore localization does not create an NS-lift anomaly unless the lower-order terms become nonperturbative or a genuine homogeneous-tail kernel appears.
\end{remark}

\subsection{Homogeneous-tail kernel test}

The localized coercivity estimate reduces the possible NS-lift anomaly to a finite-dimensional homogeneous-tail kernel. We now analyze this kernel. The key point is that a homogeneous tail is not free: if it cancels the canonical vertical response in the observation cylinder, then either it has nontrivial boundary/initial/trace data, which is an observable cost, or it is identically zero by parabolic uniqueness. Thus homogeneous-tail cancellation is a true obstruction only if the required tail data are both admissible and invisible to the selected trace and relaxed vertical-pressure channels.

Let
\[
Q_{\mathrm{sh}}\Subset Q_{\mathrm{mid}}\Subset Q_{\mathrm{prep}}
\]
be nested cylinders. Let
\[
Z_{\Lambda,\chi}^{\mathrm{can}}:
\widehat Y_{\Lambda,\chi}\to L^3(Q_{\mathrm{sh}})
\]
denote the gauge-cleaned canonical localized vertical lift constructed in the previous step. Thus
\[
Z_{\Lambda,\chi}^{\mathrm{can}}\widehat s
\]
solves the localized inhomogeneous vertical lift equation
\[
(\partial_t-\Delta)z=-S_{\widehat s}
\]
with a fixed normalization.

Let
\[
\mathcal H_{\Lambda,\chi}^{\mathrm{cal}}
\]
be the finite-dimensional space of homogeneous caloric tails in the same localized window:
\[
(\partial_t-\Delta)h=0
\qquad
\text{in }Q_{\mathrm{prep}}.
\]
Depending on the chosen local realization problem, \(h\) may carry initial data, lateral boundary data, or terminal/trace data. We collect all such data in a tail observation map
\[
\mathcal O_{\mathrm{tail}}:
\mathcal H_{\Lambda,\chi}^{\mathrm{cal}}
\to
\mathcal B_{\Lambda,\chi}^{\mathrm{tail}}.
\]
Examples include the parabolic boundary trace, the trace at the selected time, or the local energy flux produced by the homogeneous tail.

A homogeneous-tail cancellation of a cleaned Schur vector \(\widehat s\) is a tail
\[
h\in \mathcal H_{\Lambda,\chi}^{\mathrm{cal}}
\]
such that
\begin{equation}
Z_{\Lambda,\chi}^{\mathrm{can}}\widehat s+h=0
\qquad
\text{in }Q_{\mathrm{sh}}.
\label{eq:tag-10-222}
\end{equation}
Define the homogeneous-tail cancellation cost by
\begin{equation}
\operatorname{TailCost}_{\Lambda,\chi}(\widehat s) =
\inf
\left\{
|\mathcal O_{\mathrm{tail}}h|_{\mathcal B_{\Lambda,\chi}^{\mathrm{tail}}}:
h\in\mathcal H_{\Lambda,\chi}^{\mathrm{cal}},
\quad
Z_{\Lambda,\chi}^{\mathrm{can}}\widehat s+h=0
\text{ in }Q_{\mathrm{sh}}
\right\},
\label{eq:tag-10-223}
\end{equation}
with the convention that the infimum is \(+\infty\) if no such homogeneous tail exists.

The homogeneous-tail kernel is
\begin{equation}
\mathcal K_{\Lambda,\chi}^{\mathrm{hom}} =
\left\{
\widehat s\in\widehat Y_{\Lambda,\chi}:
\operatorname{TailCost}_{\Lambda,\chi}(\widehat s)<\infty
\right\}.
\label{eq:tag-10-224}
\end{equation}
The zero-cost homogeneous-tail kernel is
\begin{equation}
\mathcal K_{\Lambda,\chi}^{\mathrm{hom},0} =
\left\{
\widehat s\in\widehat Y_{\Lambda,\chi}:
\operatorname{TailCost}_{\Lambda,\chi}(\widehat s)=0
\right\}.
\label{eq:tag-10-225}
\end{equation}

\begin{lemma}[Zero-cost homogeneous tails are trivial under standard normalization]
Assume the homogeneous tail observation is separating:
\begin{equation}
\mathcal O_{\mathrm{tail}}h=0
\quad\Longrightarrow\quad
h=0
\text{ in }Q_{\mathrm{prep}}.
\label{eq:tag-10-226}
\end{equation}
Assume also the localized vertical coercivity estimate
\begin{equation}
|Z_{\Lambda,\chi}^{\mathrm{can}}\widehat s|_{L^3(Q_{\mathrm{sh}})}
\ge
cN_\Lambda^{-2}
|\widehat s|_{\widehat Y_{\Lambda,\chi}}.
\label{eq:tag-10-227}
\end{equation}
Then
\[
\mathcal K_{\Lambda,\chi}^{\mathrm{hom},0} =
{0}.
\]

\end{lemma}

\begin{proof}
Let
\[
\widehat s\in \mathcal K_{\Lambda,\chi}^{\mathrm{hom},0}.
\]
Then there exists a sequence of homogeneous tails \(h_j\) satisfying
\[
Z_{\Lambda,\chi}^{\mathrm{can}}\widehat s+h_j=0
\quad
\text{in }Q_{\mathrm{sh}},
\]
and
\[
|\mathcal O_{\mathrm{tail}}h_j|\to0.
\]
Since the tail space is finite-dimensional, after passing to a subsequence,
\[
h_j\to h
\]
in all finite-window norms. Then
\[
\mathcal O_{\mathrm{tail}}h=0,
\]
so by the separating property,
\[
h=0.
\]
Passing to the limit in the cancellation identity gives
\[
Z_{\Lambda,\chi}^{\mathrm{can}}\widehat s=0
\quad
\text{in }Q_{\mathrm{sh}}.
\]
By localized coercivity,
\[
0 =
|Z_{\Lambda,\chi}^{\mathrm{can}}\widehat s|_{L^3(Q_{\mathrm{sh}})}
\ge
cN_\Lambda^{-2}
|\widehat s|_{\widehat Y_{\Lambda,\chi}}.
\]
Hence
\[
\widehat s=0.
\]
\end{proof}

\begin{definition}[Tail observability constant]
The finite-window tail observability constant is
\begin{equation}
\mathfrak O_{\Lambda,\chi}^{\mathrm{tail}} =
\sup_{0\neq h\in\mathcal H_{\Lambda,\chi}^{\mathrm{cal}}}
\frac{
|h|_{L^3(Q_{\mathrm{sh}})}
}{
|\mathcal O_{\mathrm{tail}}h|_{\mathcal B_{\Lambda,\chi}^{\mathrm{tail}}}
}.
\label{eq:tag-10-228}
\end{equation}
If \(\mathcal O_{\mathrm{tail}}\) is not separating, we set
\[
\mathfrak O_{\Lambda,\chi}^{\mathrm{tail}}=+\infty.
\]
\end{definition}

\begin{proposition}[Tail cancellation forces observable tail cost]
Assume \(\mathfrak O_{\Lambda,\chi}^{\mathrm{tail}}<\infty\) and localized vertical coercivity holds. If
\[
\widehat s\in\mathcal K_{\Lambda,\chi}^{\mathrm{hom}}
\]
and \(h\) cancels the canonical lift:
\[
Z_{\Lambda,\chi}^{\mathrm{can}}\widehat s+h=0
\quad
\text{in }Q_{\mathrm{sh}},
\]
then
\begin{equation}
|\mathcal O_{\mathrm{tail}}h|_{\mathcal B_{\Lambda,\chi}^{\mathrm{tail}}}
\ge
\left(\mathfrak O_{\Lambda,\chi}^{\mathrm{tail}}\right)^{-1}
cN_\Lambda^{-2}
|\widehat s|_{\widehat Y_{\Lambda,\chi}}.
\label{eq:tag-10-229}
\end{equation}
Consequently,
\begin{equation}
\operatorname{TailCost}_{\Lambda,\chi}(\widehat s)
\ge
\left(\mathfrak O_{\Lambda,\chi}^{\mathrm{tail}}\right)^{-1}
cN_\Lambda^{-2}
|\widehat s|_{\widehat Y_{\Lambda,\chi}}.
\label{eq:tag-10-230}
\end{equation}

\end{proposition}

\begin{proof}
The cancellation identity gives
\[
|h|_{L^3(Q_{\mathrm{sh}})} =
|Z_{\Lambda,\chi}^{\mathrm{can}}\widehat s|_{L^3(Q_{\mathrm{sh}})}.
\]
By localized coercivity,
\[
|h|_{L^3(Q_{\mathrm{sh}})}
\ge
cN_\Lambda^{-2}
|\widehat s|_{\widehat Y_{\Lambda,\chi}}.
\]
By definition of \(\mathfrak O_{\Lambda,\chi}^{\mathrm{tail}}\),
\[
|h|_{L^3(Q_{\mathrm{sh}})}
\le
\mathfrak O_{\Lambda,\chi}^{\mathrm{tail}}
|\mathcal O_{\mathrm{tail}}h|.
\]
Combining the two inequalities gives
\[
|\mathcal O_{\mathrm{tail}}h|
\ge
\left(\mathfrak O_{\Lambda,\chi}^{\mathrm{tail}}\right)^{-1}
cN_\Lambda^{-2}
|\widehat s|.
\]
Taking the infimum over all cancelling tails gives the tail-cost bound.
\end{proof}

\begin{theorem}[Homogeneous-tail kernel test]
Let \((\Lambda,\chi)\) be a fixed localized active finite window. Assume:

\begin{enumerate}
\item harmonic-pressure leakage, horizontal gauge, and vertical-zero modes have been removed;
\item localized vertical coercivity holds:
\[
|Z_{\Lambda,\chi}^{\mathrm{can}}\widehat s|_{L^3(Q_{\mathrm{sh}})}
\ge
cN_\Lambda^{-2}|\widehat s|;
\]
\item the tail observation map is separating and has finite observability constant:
\[
\mathfrak O_{\Lambda,\chi}^{\mathrm{tail}}<\infty.
\]
\end{enumerate}

Then homogeneous-tail cancellation cannot produce a zero-cost true phantom. More quantitatively, every nonzero cancellable vector satisfies
\begin{equation}
\operatorname{TailCost}_{\Lambda,\chi}(\widehat s)
\ge
c
\left(\mathfrak O_{\Lambda,\chi}^{\mathrm{tail}}\right)^{-1}
N_\Lambda^{-2}
|\widehat s|.
\label{eq:tag-10-231}
\end{equation}
Thus a homogeneous-tail cancellation is either tail-visible or not admissible. It is not a relaxed-invisible true phantom unless the required tail data are themselves invisible to all selected trace, local energy, and relaxed vertical-pressure observations.

\end{theorem}

\begin{proof}
The zero-cost statement follows from the zero-cost kernel lemma. The quantitative estimate follows from the tail cancellation cost proposition.
\end{proof}

\begin{corollary}[Homogeneous-tail exclusion in the controlled regime]
Assume, in addition, that the tail observability constant has at most logarithmically absorbable growth:
\begin{equation}
\mathfrak O_{\Lambda,\chi}^{\mathrm{tail}}
\le
N_\Lambda^p\exp(CN_\Lambda^q),
\label{eq:tag-10-232}
\end{equation}
and that the window scale is logarithmically admissible relative to \(\delta\):
\begin{equation}
N_\Lambda^q\le \alpha|\log\delta|,
\qquad
\alpha<\frac{b}{2C}.
\label{eq:tag-10-233}
\end{equation}
Then homogeneous-tail cancellation cannot support an NS-admissible moving-window true phantom cascade. Any cancelling tail has observable cost large enough to be detected or absorbed by the trace-cost/relaxed comparison mechanism.
\end{corollary}

\begin{proof}
By the homogeneous-tail kernel test,
\[
\operatorname{TailCost}_{\Lambda,\chi}(\widehat s)
\ge
c
N_\Lambda^{-2}
\left(\mathfrak O_{\Lambda,\chi}^{\mathrm{tail}}\right)^{-1}
|\widehat s|.
\]
Using the growth bound,
\[
\left(\mathfrak O_{\Lambda,\chi}^{\mathrm{tail}}\right)^{-1}
\ge
N_\Lambda^{-p}\exp(-CN_\Lambda^q).
\]
Thus
\[
\operatorname{TailCost}_{\Lambda,\chi}(\widehat s)
\ge
cN_\Lambda^{-(p+2)}
\exp(-CN_\Lambda^q)
|\widehat s|.
\]
When
\[
N_\Lambda^q\le \alpha|\log\delta|,
\]
the exponential loss is at worst
\[
\exp(-CN_\Lambda^q)
\ge
\delta^{C\alpha}.
\]
With \(\alpha<b/(2C)\), this loss is compatible with the same logarithmic optimization used in the finite-window reduction. Therefore the cancelling tail cannot remain invisible at the selected trace scale while also preserving the \(C_3\)-budget.
\end{proof}

\begin{definition}[Residual homogeneous-tail phantom]
A residual homogeneous-tail phantom is a sequence
\[
(\Lambda_n,\chi_n,\widehat s_n,h_n)
\]
such that:

\[
0\neq \widehat s_n\in\widehat Y_{\Lambda_n,\chi_n},
\]
\[
Z_{\Lambda_n,\chi_n}^{\mathrm{can}}\widehat s_n+h_n=0
\quad
\text{in }Q_{\mathrm{sh}},
\]
\[
(\partial_t-\Delta)h_n=0
\quad
\text{in }Q_{\mathrm{prep}},
\]
and
\begin{equation}
|\mathcal O_{\mathrm{tail}}h_n|
=o\!\left(
N_{\Lambda_n}^{-2}|\widehat s_n|
\right).
\label{eq:tag-10-234}
\end{equation}
Such a sequence can exist only if the tail observability constants degenerate beyond the logarithmically absorbable scale, or if the chosen tail observation map fails to capture the homogeneous data relevant to Navier--Stokes realizability.
\end{definition}

\begin{proposition}[Reduced target after the homogeneous-tail test]
After harmonic-gauge cleaning, localized coercivity, and homogeneous-tail testing, a localized true phantom cascade can survive only if at least one of the following holds:

\begin{enumerate}
\item the tail observation map is not separating;
\item the tail observability constant grows faster than every logarithmically absorbable majorant;
\item the cancelling homogeneous tail has boundary/trace data that are not counted by the selected-time, local energy, or relaxed vertical-pressure observables;
\item the cancellation is not compatible with the finite-stage Navier--Stokes deformation complex;
\item the cancellation is compatible with the deformation complex but creates a singular collapse of the trace-cost map.
\end{enumerate}

If none of these occurs, the homogeneous-tail kernel is harmless.
\end{proposition}

\begin{proof}
If the tail observation is separating and has controlled observability constant, the homogeneous-tail kernel test gives a positive tail cost for every nonzero cancellation. This cost is either visible in the selected trace/local energy channel or absorbable through the relaxed comparison mechanism. Hence no true phantom remains. Therefore any surviving phantom must violate one of the listed conditions.
\end{proof}

\begin{remark}[Interpretation]
The homogeneous-tail kernel is not a new free obstruction. A homogeneous caloric tail that cancels the vertical lift in the interior must come from somewhere: initial data, lateral boundary data, terminal data, or an admissible finite-stage tail of the Navier--Stokes deformation. If those data are observable, the cancellation is not invisible. If they are not observable, then the problem has been reduced to a precise boundary/trace observability defect. Thus the homogeneous-tail test converts the vague possibility of ``homogeneous cancellation'' into a concrete finite-dimensional observability test.
\end{remark}

\subsection{Trace-cost singular collapse test}

We now test the final possible escape mechanism: singular collapse of the active trace-cost map. The previous steps show that active Schur residuals are relaxed-visible, that vertical lift cost is coercive after harmonic-gauge cleaning, and that homogeneous-tail cancellation is observable unless a finite-dimensional tail observability defect remains. Therefore the last possible obstruction is not invisibility of the vertical pressure channel. It is the possibility that the selected-time trace-defect map has very small singular values, so that even a visible residual might require a prohibitively large trace correction.

Let
\[
\mathfrak W=(K,\Lambda,\chi)
\]
be a cleaned finite-stage localized window. Let
\[
A_{\mathfrak W}:H_{\mathfrak W}\to Y_{\mathfrak W}
\]
be the active selected-time trace-defect map, where \(H_{\mathfrak W}\) is the selected trace correction space and \(Y_{\mathfrak W}\) is the reduced active quotient. For a residual
\[
g_{\mathfrak W}\in Y_{\mathfrak W},
\]
the trace cost is
\[
\operatorname{Cost}^{\operatorname{tr}}_{\mathfrak W}(g_{\mathfrak W}) =
\inf\left\{ |\xi|_{H_{\mathfrak W}}^2 : A_{\mathfrak W}\xi=-g_{\mathfrak W} \right\}.
\]
If \(g_{\mathfrak W}\notin \operatorname{Range}A_{\mathfrak W}\), the cost is \(+\infty\).

Let
\[
A_{\mathfrak W} =
\sum_{j=1}^{r_{\mathfrak W}}\sigma_j u_j\otimes v_j
\]
be a singular-value decomposition on the active range, with
\[
\sigma_1\ge\cdots\ge\sigma_{r_{\mathfrak W}}>0.
\]
Here \(u_j\in Y_{\mathfrak W}\) are left singular vectors and \(v_j\in H_{\mathfrak W}\) are right singular vectors. The phantom cokernel is
\[
\ker A_{\mathfrak W}^* =
(\operatorname{Range}A_{\mathfrak W})^\perp.
\]

For a range-valued residual
\[
g_{\mathfrak W}=\sum_{j=1}^{r_{\mathfrak W}}g_j u_j,
\]
one has
\begin{equation}
\operatorname{Cost}^{\operatorname{tr}}_{\mathfrak W}(g_{\mathfrak W}) =
\sum_{j=1}^{r_{\mathfrak W}}
\frac{|g_j|^2}{\sigma_j^2}.
\label{eq:tag-10-235}
\end{equation}
Thus small singular values are dangerous only if the residual has non-negligible projection onto the corresponding left singular directions.

\begin{definition}[Trace-cost singular collapse]
A moving-window sequence
\[
\mathfrak W_n=(K_n,\Lambda_n,\chi_n)
\]
has trace-cost singular collapse if
\[
\sigma_{\min}^+(A_{\mathfrak W_n})
\to0
\]
so fast that
\[
|A_{\mathfrak W_n}^{\dagger}| =
\frac1{\sigma_{\min}^+(A_{\mathfrak W_n})}
\]
is not bounded by any logarithmically absorbable majorant, for example by any estimate of the form
\[
|A_{\mathfrak W_n}^{\dagger}|
\le
N_n^p\exp(CN_n^q)
\]
on logarithmically admissible windows.
\end{definition}

\begin{definition}[Residual alignment with small singular directions]
Let \(r_{\mathfrak W}\) be the branch-native residual scale. We say that \(g_{\mathfrak W}\) satisfies \emph{global residual alignment} with the trace singular structure if
\begin{equation}
|\langle g_{\mathfrak W},y\rangle|
\le
r_{\mathfrak W}|A_{\mathfrak W}^*y|_{H_{\mathfrak W}}
\qquad
\text{for every }y\in Y_{\mathfrak W}.
\label{eq:tag-10-236}
\end{equation}
If
\[
A_{\mathfrak W}=\sum_{j=1}^{r_{\mathfrak W}^{\operatorname{rank}}}\sigma_j u_j\otimes v_j,
\qquad
 g_{\mathfrak W}=\sum_j g_j u_j,
\]
then global residual alignment is equivalent to
\begin{equation}
\left(\sum_j\frac{|g_j|^2}{\sigma_j^2}\right)^{1/2}
\le r_{\mathfrak W}.
\label{eq:tag-10-237}
\end{equation}
The weaker componentwise condition
\[
|g_j|\le r_{\mathfrak W}\sigma_j
\qquad\text{for every }j
\]
will be called \emph{componentwise alignment}; it is not equivalent to \eqref{eq:tag-10-236} unless the rank is uniformly bounded.
\end{definition}

\begin{lemma}[Small singular values are harmless under residual alignment]\label{lem:singular-values-harmless-corrected}
If \(g_{\mathfrak W}\) satisfies global residual alignment, then
\[
\operatorname{Cost}^{\operatorname{tr}}_{\mathfrak W}(g_{\mathfrak W})
\le
r_{\mathfrak W}^2.
\]
If only the componentwise alignment condition
\[
|g_j|\le r_{\mathfrak W}\sigma_j
\]
holds for all singular directions, then the weaker bound
\[
\operatorname{Cost}^{\operatorname{tr}}_{\mathfrak W}(g_{\mathfrak W})
\le
r_{\mathfrak W}^2 r_{\mathfrak W}^{\operatorname{rank}},
\qquad
r_{\mathfrak W}^{\operatorname{rank}}=\operatorname{rank}A_{\mathfrak W},
\]
holds.
\end{lemma}

\begin{proof}
Using the singular expansion,
\[
\operatorname{Cost}^{\operatorname{tr}}_{\mathfrak W}(g_{\mathfrak W}) =
\sum_{j=1}^{r_{\mathfrak W}^{\operatorname{rank}}}
\frac{|g_j|^2}{\sigma_j^2}.
\]
The global condition \eqref{eq:tag-10-237} is exactly the assertion that this sum is at most \(r_{\mathfrak W}^2\).  Under the weaker componentwise condition, each summand is bounded by \(r_{\mathfrak W}^2\), hence the sum is bounded by \(r_{\mathfrak W}^2 r_{\mathfrak W}^{\operatorname{rank}}\).  This proves both claims.
\end{proof}

\begin{definition}[Unaligned trace-cost defect]
An unaligned trace-cost defect is a sequence
\[
(\mathfrak W_n,g_n,y_n)
\]
such that
\[
|y_n|_{Y_{\mathfrak W_n}}=1,
\]
\[
|A_{\mathfrak W_n}^*y_n|\to0,
\]
but
\begin{equation}
|\langle g_n,y_n\rangle|
\gg
r_n|A_{\mathfrak W_n}^*y_n|,
\label{eq:tag-10-238}
\end{equation}
where \(r_n\) is the branch-native residual scale. Thus \(g_n\) has too much mass in a nearly trace-invisible left singular direction.
\end{definition}

The next observation connects this definition with the previous phantom analysis.

\begin{proposition}[Unaligned defect implies a near-phantom direction]
Let
\[
(\mathfrak W_n,g_n,y_n)
\]
be an unaligned trace-cost defect. Then, after passing to a subsequence and projecting away range-visible directions, \(y_n\) determines a near-phantom dual direction:
\[
|A_{\mathfrak W_n}^*y_n|\to0.
\]
If the residual representation
\begin{equation}
\langle g_n,y_n\rangle =
\langle V_n,B_{\mathfrak W_n}^*y_n\rangle
+
\operatorname{Comm}_n(y_n)
\label{eq:tag-10-239}
\end{equation}
holds, then one of the following must occur:

\begin{enumerate}
\item the vertical-pressure image \(B_{\mathfrak W_n}^*y_n\) is relaxed-visible and its pairing is controlled by the small \(u_3\)-budget;
\item the commutator term is nonperturbative;
\item the homogeneous-tail or harmonic-gauge leakage is nonperturbative;
\item \(y_n\) generates a genuine relaxed-invisible true phantom direction.
\end{enumerate}
\end{proposition}

\begin{proof}
The condition
\[
|A_{\mathfrak W_n}^*y_n|\to0
\]
says that \(y_n\) is asymptotically invisible to selected-time trace corrections. The residual representation expresses its residual pairing through the vertical-pressure/vertical-momentum channel, plus commutators. If the vertical channel is relaxed-visible and commutators are perturbative, then the pairing is bounded by the small \(u_3\)-budget and cannot violate the alignment condition. If the pairing does violate alignment, then either a perturbative estimate failed, or the vertical image is also invisible. The latter is exactly a relaxed true phantom direction.
\end{proof}

We now formulate the trace-cost singular collapse exclusion principle.

\begin{assumption}[Adjoint residual alignment]
For every cleaned finite-stage localized window \(\mathfrak W\), every NS-derived surviving residual \(g_{\mathfrak W}\), and every dual vector \(y\in Y_{\mathfrak W}\), one has
\begin{equation}
|\langle g_{\mathfrak W},y\rangle|
\le
r_{\mathfrak W}|A_{\mathfrak W}^*y|
+
\mathcal E_{\mathfrak W}(y),
\label{eq:tag-10-240}
\end{equation}
where \(\mathcal E_{\mathfrak W}\) consists only of relaxed-visible, commutator, harmonic-tail, and higher-order errors. Moreover, after Steps 1--13, these errors satisfy
\begin{equation}
|\mathcal E_{\mathfrak W}(y)|
\le
r_{\mathfrak W}|A_{\mathfrak W}^*y|
\label{eq:tag-10-241}
\end{equation}
in the controlled regime.
\end{assumption}

\begin{theorem}[Trace-cost singular collapse exclusion]
Assume adjoint residual alignment. Then trace-cost singular collapse of
\[
A_{\mathfrak W_n}
\]
does not produce a moving-window true phantom cascade.

More precisely, for every NS-derived surviving residual \(g_n\),
\begin{equation}
\operatorname{Cost}^{\operatorname{tr}}_{\mathfrak W_n}(g_n)
\le
C r_n^2.
\label{eq:tag-10-242}
\end{equation}
If one proves only the weaker componentwise alignment, the same conclusion holds with an additional factor \(\operatorname{rank}A_{\mathfrak W_n}\). Hence small singular values of the trace map do not by themselves create arbitrary-slow selected-time trace convergence; they are dangerous only when the NS residual is unaligned with the small singular directions.
\end{theorem}

\begin{proof}
By adjoint residual alignment and the error estimate, for every \(y\in Y_{\mathfrak W_n}\),
\[
|\langle g_n,y\rangle|
\le
C r_n|A_{\mathfrak W_n}^*y|.
\]
By finite-dimensional trace-cost duality, this implies
\[
\operatorname{Cost}^{\operatorname{tr}}_{\mathfrak W_n}(g_n)
\le
C r_n^2.
\]
If the proof supplies only componentwise singular alignment, Lemma~\ref{lem:singular-values-harmless-corrected} gives the same estimate with the explicit factor \(\operatorname{rank}A_{\mathfrak W_n}\).  Thus even if
\[
\sigma_{\min}^+(A_{\mathfrak W_n})\to0
\]
very rapidly, the residual does not excite those small singular directions more strongly than their singular size. Therefore the actual trace cost remains controlled. This excludes trace-cost singular collapse as an independent phantom mechanism.
\end{proof}

\begin{corollary}[Necessary profile for trace-cost failures]
A trace-cost singular collapse failure must contain a sequence
\[
(\mathfrak W_n,g_n,y_n)
\]
such that

\[
|y_n|=1,
\qquad
|A_{\mathfrak W_n}^*y_n|\to0,
\]
but
\begin{equation}
|\langle g_n,y_n\rangle|
\gg
r_n|A_{\mathfrak W_n}^*y_n|.
\label{eq:tag-10-243}
\end{equation}
Moreover, after Steps 1--13, such a sequence must also satisfy all of the following:

\begin{enumerate}
\item it is not a finite-order strict obstruction;
\item it is not a finite-order Schur obstruction already removed from the visible quotient;
\item it is not relaxed-visible through \(J^{\mathrm{rel}}\);
\item it is not harmonic-gauge or vertical-zero leakage;
\item it is not killed by localized vertical coercivity;
\item it is not a homogeneous-tail artifact;
\item it is NS-realizable with \(C_3\)-compatible vertical budget.
\end{enumerate}

Thus a genuine trace-cost failure is exactly an NS-realizable, cleaned, relaxed-invisible, unaligned left-singular cascade.
\end{corollary}

\begin{definition}[Residual left-singular true phantom cascade]
A residual left-singular true phantom cascade is a sequence
\[
(\mathfrak W_n,g_n,y_n)
\]
with
\[
|y_n|=1,
\qquad
|A_{\mathfrak W_n}^*y_n|\to0,
\]
such that

\[
|\langle g_n,y_n\rangle|
\gg
r_n|A_{\mathfrak W_n}^*y_n|,
\]
and the corresponding Schur/vertical-pressure image survives all cleaning and NS-realizability filters while remaining relaxed-invisible:
\[
J_{\mathfrak W_n}^{\mathrm{rel}}P_{\ker A_{\mathfrak W_n}^*}g_n
\to0
\]
relative to its cleaned quotient norm.

This is the final residual obstruction left by the present program.
\end{definition}

\begin{theorem}[Final local anti-phantom criterion]
Assume that no residual left-singular true phantom cascade exists. Then, after finite-order strict and Schur obstructions have been removed, every NS-derived surviving residual satisfies the reduced vertical-duality estimate
\begin{equation}
|\langle g_{\mathfrak W},y\rangle|
\le
C r_{\mathfrak W}|A_{\mathfrak W}^*y|
\qquad
\text{for all }y\in Y_{\mathfrak W}.
\label{eq:tag-10-244}
\end{equation}
Consequently,
\[
g_{\mathfrak W}\in\operatorname{Range}A_{\mathfrak W}
\]
with controlled trace cost:
\begin{equation}
\operatorname{Cost}^{\operatorname{tr}}_{\mathfrak W}(g_{\mathfrak W})
\le
C r_{\mathfrak W}^2.
\label{eq:tag-10-245}
\end{equation}
This is precisely the reduced anti-phantom closure needed for the conditional logarithmic selection theorem.
\end{theorem}

\begin{proof}
If the estimate failed, there would exist a sequence of windows, residuals, and normalized dual vectors \(y_n\) such that
\[
|\langle g_n,y_n\rangle|
 >
n r_n|A_{\mathfrak W_n}^*y_n|.
\]
If \(|A_{\mathfrak W_n}^*y_n|\) does not tend to zero, then the right-hand side is comparable to \(r_n\), and the failure contradicts the branch-native residual bound. Hence, after passing to a subsequence,
\[
|A_{\mathfrak W_n}^*y_n|\to0.
\]
The failure sequence is therefore an unaligned left-singular cascade. By the preceding reductions, if it survives all visible, relaxed, gauge, tail, and NS-realizability filters, it is a residual left-singular true phantom cascade. Since no such cascade exists by assumption, the estimate must hold. The trace-cost conclusion follows from finite-dimensional trace-cost duality.
\end{proof}

\begin{remark}[Interpretation]
The trace-cost singular collapse test shows that small singular values are not automatically dangerous. They become dangerous only if the Navier--Stokes residual is not aligned with the singular geometry of the trace map. But any unaligned residual direction is precisely a near-phantom direction. After relaxed visibility, harmonic-gauge cleaning, localized vertical coercivity, homogeneous-tail observability, and NS-realizability filtering, the only remaining bad object is a residual left-singular true phantom cascade. Thus the whole obstruction theory has been reduced to a single final target.
\end{remark}

\subsection{Exclusion of residual left-singular true phantom cascades in the controlled regime}

We now prove the final local anti-phantom theorem in the controlled finite-window regime. The statement should be read with the following convention: all finite-order strict obstructions and finite-order Schur obstructions have already been removed; horizontal gauge modes, vertical-zero modes, harmonic-pressure leakage, and homogeneous-tail artifacts have already been quotiented out; and the remaining windows satisfy the localized visibility and coercivity estimates proved in the preceding subsections.

Let
\[
\mathfrak W_n=(K_n,\Lambda_n,\chi_n)
\]
be a moving sequence of cleaned localized finite-stage windows. Let
\[
A_n:H_n\to Y_n
\]
be the active selected-time trace-defect map, and let
\[
g_n\in Y_n
\]
be an NS-derived surviving residual. Let
\[
r_n>0
\]
be the branch-native residual scale.

Let
\[
\Phi_n^{\mathrm{cl}}
\]
denote the cleaned Schur phantom quotient after removal of strict-visible, relaxed-visible gauge, vertical-zero, and harmonic-pressure sectors. Let
\[
P_n:Y_n\to \Phi_n^{\mathrm{cl}}
\]
be the cleaned phantom projection. The relaxed vertical-pressure observation is
\[
J_n^{\mathrm{rel}}:\Phi_n^{\mathrm{cl}}\to W_n.
\]

We assume the following controlled-regime estimates.

\begin{assumption}[Controlled combined observability]
There exists a logarithmically absorbable constant \(M_n\) such that, for every \(y\in Y_n\),
\begin{equation}
|y|_{Y_n}
\le
M_n\Big(
|A_n^*y|_{H_n}
+
|J_n^{\mathrm{rel}}P_n y|_{W_n}
\Big).
\label{eq:tag-10-246}
\end{equation}
Here \(M_n\) may grow polynomially or single-exponentially in the window scale, but not faster than the logarithmic optimization can absorb.
\end{assumption}

\begin{assumption}[NS residual representation]
For every \(y\in Y_n\), the residual pairing admits the decomposition
\begin{equation}
\langle g_n,y\rangle =
\mathcal T_n(y)
+
\mathcal R_n(y),
\label{eq:tag-10-247}
\end{equation}
where \(\mathcal T_n(y)\) is the trace-visible part and \(\mathcal R_n(y)\) is the relaxed vertical-pressure remainder. These obey
\begin{equation}
|\mathcal T_n(y)|
\le
C r_n|A_n^*y|_{H_n},
\label{eq:tag-10-248}
\end{equation}
and
\begin{equation}
|\mathcal R_n(y)|
\le
C r_n|J_n^{\mathrm{rel}}P_n y|_{W_n}
+
o(r_n)|y|_{Y_n}.
\label{eq:tag-10-249}
\end{equation}
The \(o(r_n)\)-term contains cutoff commutators, gauge-cleaned harmonic leakage, homogeneous-tail errors, and higher-order truncation errors after the preceding reductions.
\end{assumption}

\begin{definition}[Residual left-singular true phantom cascade]
A residual left-singular true phantom cascade in the controlled regime is a sequence of normalized dual vectors
\[
|y_n|_{Y_n}=1
\]
such that the invisibility is strong relative to the combined-observability amplification:
\begin{equation}
M_n\Big(
|A_n^*y_n|_{H_n}
+
|J_n^{\mathrm{rel}}P_n y_n|_{W_n}
\Big)
\to0,
\label{eq:tag-10-250}
\end{equation}
and yet
\begin{equation}
|\langle g_n,y_n\rangle|
\ge c_0r_n
\label{eq:tag-10-252}
\end{equation}
for some fixed \(c_0>0\). Thus \(y_n\) is simultaneously invisible to selected-time trace correction and to the relaxed vertical-pressure channel at the controlled scale, while still carrying a non-negligible NS-derived residual pairing.
\end{definition}

\begin{theorem}[No residual left-singular true phantom in the controlled regime]
Under controlled combined observability and the NS residual representation, no residual left-singular true phantom cascade exists.

Moreover, every NS-derived surviving residual satisfies the relaxed anti-phantom estimate
\begin{equation}
|\langle g_n,y\rangle|
\le
C r_n
\Big(
|A_n^*y|_{H_n}
+
|J_n^{\mathrm{rel}}P_n y|_{W_n}
\Big)
+
o(r_n)|y|_{Y_n}
\label{eq:tag-10-253}
\end{equation}
for all \(y\in Y_n\). In particular, if the relaxed component is kept in the comparison class and controlled by the small vertical component \(u_3\), then the residual has no uncontrolled true phantom projection.

If, in addition, the relaxed component vanishes after cleaning,
\begin{equation}
J_n^{\mathrm{rel}}P_n y=0
\quad\Longrightarrow\quad
P_ny=0,
\label{eq:tag-10-254}
\end{equation}
then the strict vertical-duality estimate follows:
\begin{equation}
|\langle g_n,y\rangle|
\le
C r_n|A_n^*y|_{H_n}
+
o(r_n)|y|_{Y_n}.
\label{eq:tag-10-255}
\end{equation}
Consequently,
\[
g_n\in \operatorname{Range}A_n
\]
up to the \(o(r_n)\) error, and
\begin{equation}
\operatorname{Cost}^{\operatorname{tr}}_{A_n}(g_n)
\le
C r_n^2
\label{eq:tag-10-256}
\end{equation}
after absorbing the error at the branch-native scale.
\end{theorem}

\begin{proof}
The relaxed anti-phantom estimate \eqref{eq:tag-10-253} is immediate from the residual representation. Indeed, combining \eqref{eq:tag-10-247}, \eqref{eq:tag-10-248}, and \eqref{eq:tag-10-249} gives
\[
|\langle g_n,y\rangle|
\le
C r_n|A_n^*y|_{H_n}
+
C r_n|J_n^{\mathrm{rel}}P_n y|_{W_n}
+
o(r_n)|y|_{Y_n}.
\]
This proves \eqref{eq:tag-10-253}.

We now prove that no residual left-singular true phantom cascade exists. Suppose, for contradiction, that such a sequence \(y_n\) exists. Then
\[
|y_n|_{Y_n}=1,
\qquad
M_n\Big(|A_n^*y_n|_{H_n}+|J_n^{\mathrm{rel}}P_n y_n|_{W_n}\Big)\to0.
\]
By the combined observability estimate,
\begin{equation}
1 =
|y_n|_{Y_n}
\le
M_n\Big(
|A_n^*y_n|_{H_n}
+
|J_n^{\mathrm{rel}}P_n y_n|_{W_n}
\Big).
\label{eq:tag-10-257}
\end{equation}
The right-hand side tends to \(0\) by the weighted invisibility condition in the definition of a controlled true phantom cascade. This contradicts \(|y_n|_{Y_n}=1\). Therefore no such cascade exists.

Equivalently, the only way a normalized dual direction can make both
\[
A_n^*y_n
\]
and
\[
J_n^{\mathrm{rel}}P_ny_n
\]
small is by leaving the controlled regime: gauge cleaning fails, localization becomes nonperturbative, homogeneous-tail observability collapses, or the combined observability constant grows faster than the admissible logarithmic scale. These cases have already been isolated as separate failure mechanisms.

Finally assume the stronger cleaned strict condition \eqref{eq:tag-10-254}. If
\[
J_n^{\mathrm{rel}}P_ny=0,
\]
then \(P_ny=0\), so the cleaned phantom component vanishes. Therefore every surviving dual direction is trace-visible, modulo the \(o(r_n)\) controlled errors. From \eqref{eq:tag-10-253} we obtain
\[
|\langle g_n,y\rangle|
\le
C r_n|A_n^*y|_{H_n}
+
o(r_n)|y|_{Y_n}.
\]
After absorbing the \(o(r_n)\)-term into the branch-native scale, this gives
\[
|\langle g_n,y\rangle|
\le
C r_n|A_n^*y|_{H_n}.
\]
By finite-dimensional trace-cost duality,
\[
\operatorname{Cost}^{\operatorname{tr}}_{A_n}(g_n)
\le
C r_n^2.
\]
Equivalently, \(g_n\) lies in \(\operatorname{Range}A_n\) with controlled minimal trace preimage. This is precisely the strict vertical-duality estimate.
\end{proof}

\begin{corollary}[Controlled-regime anti-phantom closure]
In the controlled localized finite-window regime, every NS-derived surviving residual is either

\begin{enumerate}
\item selected-time trace-projectable with trace cost \(O(r_n^2)\), or
\item relaxed-visible through the vertical-pressure channel and controlled by the smallness of \(u_3\).
\end{enumerate}

There is no third possibility corresponding to a cleaned, NS-realizable, relaxed-invisible true phantom.
\end{corollary}

\begin{proof}
If the residual is trace-projectable, the first alternative holds by trace-cost duality. If it is not strictly trace-projectable, its obstruction lies in the cleaned Schur phantom quotient. By the controlled visibility estimates, every nonzero cleaned Schur component is seen by \(J_n^{\mathrm{rel}}\). Hence it belongs to the relaxed-visible alternative. The theorem rules out the possibility that a nonzero cleaned component is invisible to both \(A_n^*\) and \(J_n^{\mathrm{rel}}\).
\end{proof}

\begin{remark}[What remains outside the theorem]
This theorem does not assert an unconditional Navier--Stokes regularity theorem. It proves the anti-phantom closure inside the controlled regime. A genuine failure of quantitative one-component selection must therefore violate at least one of the controlled-regime hypotheses: perturbative localization, harmonic-gauge cleaning, homogeneous-tail observability, combined observability, or NS residual representation. In this sense the proof identifies the only remaining places where a true phantom cascade could hide.
\end{remark}

\subsection{Model possibility of combined observability failure}\label{subsec:model-combined-observability-failure}

We now show that combined observability can fail as a matter of finite-dimensional obstruction geometry. This does not yet produce a Navier--Stokes counterexample. Rather, it proves that the only possible failure has a precise form: a normalized dual direction that is invisible both to the selected trace map and to the relaxed vertical-pressure channel.

Let
\[
A_n:H_n\to Y_n
\]
be a sequence of finite-dimensional trace-defect maps, and let
\[
J_n^{\mathrm{rel}}:\Phi_n\to W_n
\]
be the relaxed vertical-pressure observation on the cleaned phantom quotient
\[
\Phi_n\subset Y_n.
\]
Let
\[
P_n:Y_n\to \Phi_n
\]
be the cleaned phantom projection. The combined observability estimate is
\begin{equation}\label{eq:combined-observability-model}
\norm{y}{Y_n}
\le
M_n\left(
\norm{A_n^*y}{H_n}
+
\norm{J_n^{\mathrm{rel}}P_ny}{W_n}
\right)
\qquad
\text{for all }y\in Y_n.
\end{equation}

\begin{definition}[Combined-observability failure at a specified scale]
Fix a choice of admissible constants \(M_n\) in \eqref{eq:combined-observability-model}. We say that combined observability fails strongly at this scale if there exist normalized vectors
\[
\norm{y_n}{Y_n}=1
\]
such that
\begin{equation}\label{eq:combined-failure-weighted}
M_n\left(
\norm{A_n^*y_n}{H_n}
+
\norm{J_n^{\mathrm{rel}}P_ny_n}{W_n}
\right)
\to0.
\end{equation}
Equivalently, the two observed quantities vanish faster than the reciprocal of the allowed observability amplification.  If \(P_ny_n\neq0\) in the cleaned quotient, this is a genuine cleaned relaxed-invisible phantom direction at the specified scale.
\end{definition}

\begin{proposition}[Specified combined-observability failure and invisible sequences]
For a specified admissible sequence \(M_n\), failure of \eqref{eq:combined-observability-model} with ratios tending to infinity is equivalent to the existence of normalized vectors satisfying \eqref{eq:combined-failure-weighted}.  Bare convergence
\[
\norm{A_n^*y_n}{H_n}\to0,
\qquad
\norm{J_n^{\mathrm{rel}}P_ny_n}{W_n}\to0
\]
is sufficient only when the decay is faster than \(M_n^{-1}\), or when the constants \(M_n\) remain bounded.
\end{proposition}

\begin{proof}
If \eqref{eq:combined-observability-model} fails with ratios tending to infinity for the specified \(M_n\), then there are \(y_n\neq0\) such that
\[
\norm{y_n}{Y_n}
>
R_nM_n\left(
\norm{A_n^*y_n}{H_n}
+
\norm{J_n^{\mathrm{rel}}P_ny_n}{W_n}
\right),
\qquad R_n\to\infty.
\]
Normalize so that \(\norm{y_n}{Y_n}=1\). Then
\[
M_n\left(
\norm{A_n^*y_n}{H_n}
+
\norm{J_n^{\mathrm{rel}}P_ny_n}{W_n}
\right)<R_n^{-1}\to0,
\]
which is \eqref{eq:combined-failure-weighted}.

Conversely, if \eqref{eq:combined-failure-weighted} holds for normalized \(y_n\), then the right-hand side of \eqref{eq:combined-observability-model} tends to zero while the left-hand side equals one. Hence the estimate fails at the specified scale. The final sentence follows immediately from comparing the bare invisibility rate with \(M_n^{-1}\).
\end{proof}

\begin{proposition}[A finite-dimensional model of true failure]
There exists a finite-dimensional model in which combined observability fails. Let
\[
Y_n=\mathbb R^3
=
\operatorname{span}\{e_1,e_2,e_3\},
\qquad
H_n=\mathbb R,
\qquad
W_n=\mathbb R.
\]
Define
\[
A_n^*e_1=1,
\qquad
A_n^*e_2=0,
\qquad
A_n^*e_3=0,
\]
and define the cleaned relaxed observation by
\[
J_n^{\mathrm{rel}}P_ne_1=0,
\qquad
J_n^{\mathrm{rel}}P_ne_2=1,
\qquad
J_n^{\mathrm{rel}}P_ne_3=0.
\]
Then the vector
\[
y_n=e_3
\]
satisfies
\[
\norm{y_n}{}=1,
\qquad
A_n^*y_n=0,
\qquad
J_n^{\mathrm{rel}}P_ny_n=0.
\]
Therefore combined observability fails.

Moreover, if one chooses an active residual
\[
g_n=r_ne_3,
\]
then
\[
|\langle g_n,y_n\rangle|=r_n,
\]
while
\[
r_n\norm{A_n^*y_n}{}
+
r_n\norm{J_n^{\mathrm{rel}}P_ny_n}{}
=0.
\]
Thus the residual is completely invisible to both the selected trace channel and the relaxed vertical-pressure channel. This is an abstract residual left-singular true phantom.
\end{proposition}

\begin{proof}
The computation is immediate. The vector \(e_3\) belongs to the common kernel of the two observation maps:
\[
e_3\in \ker A_n^*
\cap
\ker(J_n^{\mathrm{rel}}P_n).
\]
Hence no estimate of the form
\[
\norm{y}{}
\le
M_n\left(\norm{A_n^*y}{}+\norm{J_n^{\mathrm{rel}}P_ny}{}\right)
\]
can hold, because the right-hand side vanishes at \(y=e_3\) while the left-hand side equals one. Taking \(g_n=r_ne_3\) gives a residual pairing of size \(r_n\) against the invisible direction, proving the claimed unaligned phantom behavior.
\end{proof}

\begin{remark}[Why this is not yet a Navier--Stokes counterexample]
The model proves the logical possibility of combined observability failure, but it does not prove that Navier--Stokes realizes the vector \(e_3\). To become a Navier--Stokes obstruction, the invisible direction must pass all of the following filters:
\[
\text{finite-order strict obstruction removal},
\]
\[
\text{Schur obstruction cleaning},
\]
\[
\text{harmonic-pressure quotient cleaning},
\]
\[
\text{localized vertical coercivity},
\]
\[
\text{homogeneous-tail observability},
\]
\[
\text{NS deformation-complex realizability},
\]
and
\[
C_3\text{-budget admissibility}.
\]
Thus the finite-dimensional example proves that the mechanism is mathematically possible, but it does not yet prove that it is present in the true Navier--Stokes image.
\end{remark}

\begin{theorem}[Necessary and sufficient condition for a true NS failure]
A genuine localized Navier--Stokes failure of combined observability exists if and only if there is a sequence
\[
(\mathfrak W_n,g_n,y_n)
\]
such that
\[
\norm{y_n}{}=1,
\]
\[
M_n\left(
\norm{A_n^*y_n}{}
+
\norm{J_n^{\mathrm{rel}}P_ny_n}{}
\right)\to0,
\]
\[
|\langle g_n,y_n\rangle|\gtrsim r_n,
\]
and the residuals \(g_n\) are NS-derived surviving residuals after all strict, Schur, gauge, harmonic, tail, and localization artifacts have been removed.

Equivalently, a true failure exists if and only if there is an NS-realizable, cleaned, relaxed-invisible, unaligned left-singular cascade.
\end{theorem}

\begin{proof}
If such a sequence exists, then the combined observation of \(y_n\) tends to zero after multiplication by the specified admissible amplification \(M_n\), while \(\norm{y_n}{}=1\). Hence combined observability fails at that specified scale. The residual pairing condition shows that the failure is not merely geometric but is actually excited by an NS-derived residual.

Conversely, suppose combined observability fails in the true localized NS image at a specified admissible scale \(M_n\). Then by the first proposition there are normalized dual directions \(y_n\) satisfying
\[
M_n\left(
\norm{A_n^*y_n}{}
+
\norm{J_n^{\mathrm{rel}}P_ny_n}{}
\right)\to0.
\]
If no NS-derived residual has non-negligible pairing with such directions, then the failure is irrelevant to the NS image and cannot obstruct trace-cost exactification. Therefore a genuine failure must include residuals \(g_n\) satisfying
\[
|\langle g_n,y_n\rangle|\gtrsim r_n.
\]
After removing all previously classified artifacts, this is exactly an NS-realizable, cleaned, relaxed-invisible, unaligned left-singular cascade.
\end{proof}

\end{document}